\documentclass[a4paper,11pt]{amsart}
\usepackage[utf8]{inputenc}
\usepackage{hyperref, mathrsfs, enumerate, dsfont, mathtools, amssymb, xcolor, bm, bbm}
\hypersetup{backref,colorlinks=true,allcolors=black}
\usepackage[foot]{amsaddr}
\usepackage{amsmath, amsthm, soul, comment}
\usepackage{rotating, graphicx, caption, subcaption, float}
\graphicspath{ {./images/} }
\usepackage[super]{nth}
\newtheorem{lemma}{Lemma}[section]

\newtheorem{definition}{Definition}[section]

\newtheorem{theorem}{Theorem}[section]
\newtheorem{remark}{Remark}[section]

\newcommand{\continuation}{??}

\newtheorem{example}{Example}[section]
\newtheorem{assumption}{Assumption}[section]

\newtheorem*{acknow*}{Acknowledgments}

\newcommand{\ignore}[1]{}
\allowdisplaybreaks
\hypersetup{colorlinks=true, linkcolor=blue, filecolor=magenta, urlcolor=cyan}
\newcommand\numberthis{\addtocounter{equation}{1}\tag{\theequation}}
\newcommand\Label[1]{&\refstepcounter{equation}\left(\theequation\right)\ltx@label{#1}&}

\newcommand{\R}{\mathbb{R}}

\newcommand{\norm}[1]{\lVert #1 \rVert}
\newcommand{\abs}[1]{\vert #1 \rvert}

\newcommand{\floor}[1]{\lfloor #1 \rfloor}

\newcommand{\vmu}{\bm{\mu}}
\newcommand{\vxi}{\bm{\xi}}
\newcommand{\vxibar}{\bar{\bm{\xi}}}
\newcommand{\vrho}{\bm{\rho}}

\newcommand{\ntzj}{\nu_t^{\zeta,j}}

\newcommand{\mtzj}{\mu_t^{\zeta,j}}
\newcommand{\msmzj}{\mu_{s-}^{\zeta,j}}
\newcommand{\msmzi}{\mu_{s-}^{\zeta,i}}

\newcommand{\mszi}{\mu_{s}^{\zeta,i}}
\newcommand{\mszk}{\mu_{s}^{\zeta,k}}
\newcommand{\msmzk}{\mu_{s-}^{\zeta,k}}

\newcommand{\dpi}{\delta_{p_i}}
\newcommand{\Rd}{\mathbb{R}^d}
\newcommand{\HQ}{H_Q^i(\gamma\mu_{s^-}^{\zeta, j})}
\newcommand{\HV}{H_V^i(\gamma\mu_{s^-}^{\zeta, j})}

\newcommand{\HQirj}{H_Q^{i_r^{(j)}}(\gamma\mu_{s^-}^{\zeta, j})}
\newcommand{\HVirj}{H_V^{i_r^{(j)}}(\gamma\mu_{s^-}^{\zeta, j})}
\newcommand{\fmi}{\left\langle f, \mu_{0}^{\zeta, j}\right\rangle}
\newcommand{\sumall}{\sum_{i=1}^{\gamma\left\langle 1, \mu_{s^-}^{\zeta, j}\right\rangle}}
\newcommand{\all}{\gamma\left\langle 1, \mu_{s^-}^{\zeta, j}\right\rangle}
\newcommand{\ginverse}{\frac{1}{\gamma}}
\newcommand{\pfpv}{\frac{\partial f}{\partial V}}
\newcommand{\ppfppv}{\frac{\partial^2 f}{\partial V^2}}
\newcommand{\pfpq}{\frac{\partial f}{\partial Q}}
\newcommand{\alj}{\alpha_{\ell j}}
\newcommand{\blj}{\beta_{\ell j}}
\newcommand{\klg}{K_{\ell}^{\gamma}}
\newcommand{\pl}{\mathcal{P}^{(\ell)}(\gamma\mu_{s^-}^{\zeta}, \bm{i})}
\newcommand{\lmeasure}{\lambda^{(\ell)}[\mu_{s-}^{\zeta}](d\bm{x}, d\bm{v})}

\newcommand{\llmeasure}{\lambda^{(\ell)}[\xi_{s}](d\bm{x}, d\bm{v})}
\newcommand{\llbmeasure}{\lambda^{(\ell)}[\bar{\xi}_{s}](d\bm{x}, d\bm{v})}

\newcommand{\xrj}{x_r^{(j)}}
\newcommand{\vrj}{v_r^{(j)}}
\newcommand{\txrj}{\tilde{x}_r^{(j)}}
\newcommand{\tvrj}{\tilde{v}_r^{(j)}}
\newcommand{\yrj}{y_r^{(j)}}
\newcommand{\vprj}{{v'}_{r}^{(j)}}

\newcommand{\txtv}{\tilde{\mathbb{X}}^{(\ell)} \otimes \tilde{\mathbb{V}}^{(\ell)}}
\newcommand{\xv}{\mathbb{X}^{(\ell)} \otimes \mathbb{V}^{(\ell)}}
\allowdisplaybreaks[3]
\newcommand{\yv}{\mathbb{Y}^{(\ell)} \otimes \mathbb{V}'^{(\ell)}}
\allowdisplaybreaks[3]

\newcommand{\fjm}{f_{j,m}}

\newcommand{\Mt}{M_t^{f,j}}
\newcommand{\At}{A_t^{f,j}}
\newcommand{\cfjt}{\mathcal{C}_{t}^{f, j}}
\newcommand{\dfjt}{\mathcal{D}_{t}^{f, j}}
\newcommand{\supint}{\displaystyle\sup_{t\in [0,T]}}
\newcommand{\gtgamma}{g_t(m,\zeta)}
\renewcommand{\vec}[1]{\bm{#1}}
\newcommand{\paren}[1]{\left(#1\right)}

\newcommand{\Vti}{V_{t}^{i}}

\newcommand{\ind}{\mathbbm{1}}

\newcommand{\la}{\left \langle}
\newcommand{\ra}{\right\rangle}

\allowdisplaybreaks[3]

\author{S. A. Isaacson}
\email[Samuel A. Isaacson]{isaacsas@bu.edu}
\author{Q. Liu}
\email[Qianhan Liu]{liuq19@bu.edu}
\author{K. Spiliopoulos}
\email[Konstantinos Spiliopoulos]{kspiliop@bu.edu}
\thanks{Corresponding Authors: SAI, KS. The authors were supported by ARO W911NF2510078. SAI was also supported by NSF-DMS 2325185. Competing interests: The author(s) declare none}

\title[Mean Field Limits for Underdamped Langevin PBSRD]{Mean Field Limits for Stochastic, Underdamped Reactive Langevin Dynamics Models}

\date{\today}

\begin{document}

\maketitle

\begin{abstract}
We rigorously derive the effective large-population, mean-field dynamics of particle-based reactive Langevin dynamics (PBRLD) models. These models extend particle-based stochastic reaction-diffusion (PBSRD) descriptions by incorporating velocities, inertial effects, and underdamped motion. In Isaacson, Liu, Spiliopoulos, and Yao, SIAP 2026, PBRLD models were formulated and shown to recover Doi’s volume reactivity PBSRD model in the overdamped limit. In this work we prove convergence of the associated measure-valued stochastic processes, representing species concentration fields on position-velocity phase space, to a deterministic mean-field limit. The limiting equations form a novel system of nonlocal kinetic reaction-diffusion partial integro-differential equations, coupling hypoelliptic transport with reaction terms that retain the spatial and velocity structure of the underlying particle interactions.
\end{abstract}



\section{Introduction}
In many systems across the biological sciences, the effective population dynamics of large collections of discrete agents or particles are driven by the stochastic movements and reactions of individual agents or particles~\cite{NaylorSimbiotics2017,DustinHermann2018dw,Chakraborty2010kb,BartolSynapse2012,HellanderHes12012,WoldeEgfrdPNAS2010}. When particles move by diffusion, as commonly assumed in modeling cellular, neurological, and immune signaling processes~\cite{Chakraborty2010kb,BartolSynapse2012,HellanderHes12012,WoldeEgfrdPNAS2010}, particle-based stochastic reaction-diffusion (PBSRD) models are a popular approach for explicitly modelling the stochastic dynamics of such systems. In recent years it has been well-established how such models rigorously relate to more coarse-grained deterministic and stochastic partial differential equation (PDE/SPDE) models~\cite{Nolen2019,IsaacsonSIAP2021, IsaacsonSIMA2022, HeldmanSPA2022, IsaacsonMeanfieldDrift2023,VeberAnnals2023}. The latter are often postulated as approximations to the macroscopic dynamics of PBSRD models in certain large-population or thermodynamic limits where the population size becomes unbounded but species concentrations are held fixed.

The goal of this work is to rigorously derive new deterministic, macroscopic equations describing the large-population mean-field limit for the behavior of particle-based reactive Langevin dynamics (PBRLD) models. Such models are more microscopic analogues to PBSRD models that retain inertial forces, with agents and/or particles moving with \emph{stochastically} forced velocities as part of their spatial transport. PBRLD models are \emph{more microscopic} than PBSRD models, with the latter corresponding to the overdamped, i.e. high friction and/or small mass, limit of the former~\cite{IsaacsonChapman2016bv,SamChenLanlanKostas2025}.


Adding inertial forces to PBSRD and agent-based models is critical to accurately modeling interacting populations of cells~\cite{HillenLDMeanField04,NaylorSimbiotics2017,SongtagSynBioRev2018}, swarming and flocking behavior in insects and robots~\cite{LocustMotion2015,Robotics2019}, and the spread of disease among populations~\cite{AnimalMovement2010}. In particular, it was pointed out in~\cite{HillenLDMeanField04} that because velocity is an essential component of the process, modeling systems with flocking, aggregation and/or chemotaxis of cells, or alignment of filaments can require the use of velocity-based models. The PBRLD model and reaction kernels used here are motivated by our recent work in~\cite{SamChenLanlanKostas2025}, where we formulated novel PBRLD models that recover the general volume reactivity PBSRD model in the overdamped limit. The classical Doi model, in which two substrates may react with a fixed probability per time when within a fixed separation, is an example of one such PBSRD model. This work therefore provides a complementary investigation of these models, rigorously establishing their large-population mean-field limit, and the resulting equations that govern their macroscopic dynamics.

To illustrate our main result, consider the reversible $A + B \leftrightarrows C$ reaction in $\R^d$. For each chemical species we can define a measure-valued stochastic process (MVSP) representing their (phase-space) molar concentration field, $A^{\gamma}_t(dx,dv)$, $B^{\gamma}_t(dx,dv)$, and $C^{\gamma}_t(dx,dv)$, where $\gamma$ is a large system parameter (for example Avogadro's constant). Here $x \in \R^d$ denotes position while $v \in \R^d$ denotes velocity. Integrating one of these measures over a phase-space set, $U_x \times U_v \subset \R^{d} \times \R^{d}$, gives us the stochastic process for the number of particles of that species in that set, i.e.
\begin{equation*}
\text{Number of A in } U_x \times U_v \text{ at time } t = \int_{U_x \times U_v} A^{\gamma}_t(dx, dv).
\end{equation*}
Denote by $\left(Q_{t},V_t \right)$ the position and velocity for a given particle of species $A$. In the absence of reactions, the dynamics for $\left(Q_{t},V_t\right)$ are governed by underdamped Langevin dynamics
\begin{align*}
    dQ_t &= V_t dt,\\
    dV_t &= -b_{A} V_t dt + b_{A} \sqrt{2D_A} dW_t^,
\end{align*}
where $W_{t}$ is a standard Brownian motion in $\mathbb{R}^{d}$, $D_A$ is the diffusivity of a species $A$ particle, and $b_A$ is the mass-scaled friction coefficient for species $A$ particles (with units of per time). Reactions are modeled via Poisson random measures, and defined via two key components. For example, suppose an $A$ particle has (phase-space) position $(x_1,v_1)$ and a $B$ particle has position $(x_2,v_2)$. The first component is a reaction kernel that defines the probability per unit time that the $A$ particle and the $B$ particle react to form a $C$ particle, $K^{\gamma}_+(x_1,x_2,v_1,v_2) := K_+(x_1,x_2,v_1,v_2) / \gamma$. The assumed $\gamma$ dependence of the reaction components is discussed in Section~\ref{S:MainAssumptions}. The second component is a particle placement density, $m_+(x_3,v_3 | x_1, x_2, v_1, v_2)$, which gives the probability density a $C$ particle is created at $(x_3,v_3)$ given the substrate $A$ and $B$ particles at $(x_1,v_1)$ and $(x_2,v_2)$ react. The reaction kernel, $K_-(x_3,v_3)$, and particle placement density, $m_-(x_1,x_2,v_1,v_2|x_3,v_3)$, for the reverse reaction, $C \to A + B$, are defined analogously. The reaction kernels and particle placement densities we use are defined in such a way that in the overdamped limit, we recover the general volume reactivity PBSRD model of Doi~\cite{DoiSecondQuantA,DoiSecondQuantB} as described in~\cite{SamChenLanlanKostas2025}.

At a high level, we consider the large-population limit $\gamma \to \infty$, with initial empirical concentration measures chosen so that $A^{\gamma}_0(dx,dv) \to \bar{A}(x,v,0) dx dv$, and analogously for $B^{\gamma}$ and $C^{\gamma}$. We prove that the associated measure-valued processes, $A_t^{\gamma}$, $B_t^{\gamma}$, and $C_t^{\gamma}$ converge to deterministic limiting measures, $\xi^{A}_t$, $\xi ^B_t$, $\xi^C_t$. If these limiting measures have associated densities, i.e. $\xi_t^{A}(dx,dv) = \bar{A}(x,v,t) dx dv$, and similarly for $B$ and $C$, then these phase-space molar concentration fields satisfy the following system of deterministic kinetic reaction-diffusion PDIEs. Let
\begin{align*}
    (\mathcal{L}^*_A)f(x,v)&\coloneqq -v \nabla_{x}f(x,v) + b_{A} \nabla_{v}\cdot [v f(x,v)  + b_{A} D_A \nabla_{v}f(x,v)]
\end{align*}
denote the kinetic Langevin transport operator for species $A$ (with $\mathcal{L}^*_B$ and $\mathcal{L}^*_C$ defined analogously). Then the system of PIDEs for the macroscopic concentration fields is given by
\begin{align*}
\partial_t \bar{A}(x,v,t) &= \mathcal{L}^*_A \bar{A}(x,v,t) - \int_{\R^{2d}} K_+(x,x',v,v') \bar{A}(x,v,t) \bar{B}(x',v',t) dx' dv' \\
&\phantom{=} + \int_{\R^{2d}} K_-(x',v') \left( \int_{\R^{2d}} m_-(x,v,x'',v''|x',v') dx''dv''\right) \bar{C}(x',v',t) dx' dv',\\
\partial_t \bar{B}(x,v,t) &= \mathcal{L}^*_B \bar{B}(x,v,t) - \int_{\R^{2d}} K_+(x',x,v',v) \bar{A}(x',v',t) \bar{B}(x,v,t) dx' dv' \\
&\phantom{=} + \int_{\R^{2d}} K_-(x',v') \left( \int_{\R^{2d}}  m_-(x'',v'',x,v|x',v') dx''dv''\right) \bar{C}(x',v',t) dx' dv',\\
\partial_t \bar{C}(x,v,t) &= \mathcal{L}^*_C \bar{C}(x,v,t)  - K_-(x,v) \bar{C}(x,v,t) \\
&\phantom{=} + \int_{\R^{4d}} K_+(x',x'',v',v'') m_+(x,v|x',x'',v',v'') \bar{A}(x',v',t) \bar{B}(x'',v'',t) dx' dv' dx'' dv''.
\end{align*}
These equations show that the macroscopic dynamics are driven by a kinetic Langevin transport term, $\mathcal{L}^*$, and nonlocal reaction terms that retain the spatial and velocity structure of the underlying particle interactions.

In the remainder, we focus on identifying and then proving the large population limit of the (weak) MVSP representation for the dynamics of PBRLD models of general reaction systems with up to second-order reactions, see Theorem~\ref{T:MainTheorem}. We prove that these measures weakly converge in the large population limit to a unique limiting measure-valued solution of the identified PIDEs for a class of reaction kernels and particle placement densities that are consistent in the overdamped limit with PBSRD models. To prove this result we generalize the martingale problem approach for studying solutions to stochastic differential equations developed by Stroock and Varadhan~\cite{EthierKurtz, StroockVaradhan} to our weak MVSP PBRLD representation. Adaptations of this method have been successfully used to study large-population limits in stochastic models for population dynamics, evolutionary dynamics, interacting particle systems, financial models, and in our work on the large population limit for PBSRD models \cite{TypicalDefaults,LargePortfolio,DaiPra2,DaiPra3,Delarue,Inglis,Moynot,Touboul,Sompolinsky,IsaacsonSIMA2022,IsaacsonDriftDB2023}.

At a high level, the martingale problem approach  developed by Stroock and Varadhan~\cite{EthierKurtz,StroockVaradhan} requires us to establish several mathematically rigorous results:
\begin{enumerate}
\item Identifying the macroscopic hypoelliptic PIDEs that are the large-population limit of the MVSP.
\item Proving that the family of MVSPs generated in this limit is a tight sequence.
\item Proving that the large-population limit of the MVSP converges weakly to a (weak) solution of the martingale problem associated with the formal limiting macroscopic hypoelliptic PIDEs.
\item Proving the uniqueness of the solution to the limiting hypoelliptic  PIDEs.
\end{enumerate}
Combined, these steps establish that the MVSP's large population limit converges (in a weak sense) to the solution of the derived (macroscopic) deterministic hypoelliptic PIDE model.

There are a number of mathematical challenges that we address in order to prove our main result, particularly with regard to showing tightness of the MVSPs for the molar concentration fields as a family over $\gamma$, and in proving uniqueness of the limiting system. For the former, we follow the high level approach we used in earlier works, see~\cite{IsaacsonSIMA2022} and its references, and first prove tightness in $\mathbb{D}_{M_F'(\mathbb{R}^{d}\times \R^d)}[0, T]$, where ${M_F'(\mathbb{R}^{d}\times \R^d)}$ is the space of finite measures endowed with the vague topology. Then,  we require technical estimates to control the mass of measures outside of compact sets in order to go from tightness in $\mathbb{D}_{M_F'(\mathbb{R}^{d}\times \R^d)}[0, T]$ to tightness in $\mathbb{D}_{M_F(\mathbb{R}^{d}\times \R^d)}[0, T]$, where $M_F$ denotes the space of finite measures endowed with the weak topology, see~\cite{MS:1993}. The proof is complicated by the need to handle general combinations of zero, first and second order reactions in the presence of hypoellipticity due to the spatial transport operators for the Langevin dynamics. For the interested reader, the book \cite{Pavliotis2014} has a good introduction to the topic of hypoellipticity.

In addition, for the second order reactions of the form $S_i + S_k \rightarrow S_j + S_r$ conservation of momentum is assumed, which is also consistent with detailed balance, see \cite{SamChenLanlanKostas2025}. This means that a-priori there is no restriction on how far away the velocities of the particles before and after reactions can be. This creates a number of technical issues which do not appear in the overdamped case such as \cite{IsaacsonSIMA2022}, where particle's velocities would not enter the models. In particular, one needs to control carefully what happens when velocities of products are close to and when they are far from the velocities of the particles before the reaction occurred.  The details of the calculations showing how to handle the different regions in this case can be found in Lemma \ref{L:boundoffm}.

The rest of the paper is organized as follows. In Section \ref{S:NotationPrelim} we introduce our notation for describing general reaction systems, the underlying MVSPs, and spaces for substrate/particle positions we will use in defining the dynamics. In Section \ref{S:MainAssumptions} we go over our main assumptions, including the $\gamma$ dependence of the reaction kernels that ensures a well-defined limit. Section \ref{S:MainResult} contains our main result on the mean field large-population limit (Theorem \ref{T:MainTheorem}), and illustrates it via a series of examples. The rest of the sections contain the proof of Theorem \ref{T:MainTheorem}. In particular, we identify the limit in Section \ref{S:Identification}, prove tightness of the measure-valued stochastic processes governing the particle dynamics in Section \ref{S:tightness}, and prove uniqueness of the mean field limit in Section \ref{S:uniqueness}. Section \ref{S:MomentBounds} contains the proofs of a number of technical moment bounds involving the reaction terms and the hypoelliptic operator.
\section{Notations and preliminary definitions}\label{S:NotationPrelim}

We consider a collection of particles with $J$ different types, with $\mathcal{S}=\{S_{1}, \cdots, S_{J}\}$ the set of different possible types and $p_{i} \in \mathcal{S}$ the type of the species of the $i$th particle. For the rest of the paper, we will interchangeably use the terms particle or molecule and type or species. We also define all random variables on an underlying probability triple, $(\Omega, \mathcal{F}, \mathbb{P})$.

In our model, molecules move via underdamped Langevin dynamics in space $\mathbb{R}^{d}$, and can undergo $L$ possible types of reactions, denoted as $\mathcal{R}_{1}, \cdots, \mathcal{R}_{L}$. We use non-negative integer stoichiometric coefficients $\{\alpha_{\ell j}\}_{j=1}^{J}$ and $\{\beta_{\ell j}\}_{j=1}^{J}$ to describe the $\mathcal{R}_{\ell}$th reaction, $\ell \in\{1, \ldots, L\}$, as
$$
\sum_{j=1}^{J} \alpha_{\ell j} S_{j} \rightarrow \sum_{j=1}^{J} \beta_{\ell j} S_{j},
$$
and the multi-index vectors $\bm{\alpha}{ }^{(\ell)}=$ $\bigl(\alpha_{\ell 1}, \alpha_{\ell 2}, \cdots, \alpha_{\ell J}\bigr)$ and $\bm{\beta}^{(\ell)}=\bigl(\beta_{\ell 1}, \beta_{\ell 2}, \cdots, \beta_{\ell J}\bigr)$ to collect the coefficients of the $\ell$th reaction. We denote the substrate and product orders of the reaction by $|\bm{\alpha}^{(\ell)}| \doteq \sum_{i=1}^{J} \alpha_{\ell i} \leq 2$ and $|\bm{\beta}^{(\ell)}| \doteq \sum_{j=1}^{J} \beta_{\ell j} \leq 2$. The implicit assumption that all reactions are at most second order is based on the premise that the probability that three substrates in a dilute system simultaneously have the proper configuration and energy levels to react is small. This assumption is further supported by the fact that reactions of order three and above are often considered approximations of sequences of bimolecular reactions in biological models. For subsequent notational purposes, we label the reactions such that the first $\tilde{L}$ reactions correspond to those that have no products, i.e., annihilation reactions of the form
$$
\sum_{j=1}^{J} \alpha_{\ell j} S_{j} \rightarrow \emptyset
$$
for $\ell \in\{1, \ldots, \tilde{L}\}$. We assume that the remaining $L-\tilde{L}$ reactions have one or more product particles.

Let $D^{i}$ label the diffusion coefficient for the $i$th particle, taking values in $\{D_{1}, \ldots, D_{J}\}$, where $D_{j}$ is the diffusion coefficient for species $S_{j}, j=1, \cdots, J$. Analogously, let $b^{i}$ label the (mass-scaled) friction coefficient for the $i$th particle, taking values in $\{b_{1}, \ldots, b_{J}\}$, where $b_{j}$ is the scaled friction coefficient for species $S_{j}, j=1, \cdots, J$. That is, $b_j = \gamma_j / m_j$, where $\gamma_j$ is the physical friction coefficient and $m_j$ is the mass for species $S_{j}$. As such, $b_j$ has units of inverse time. We denote by $Q_{t}^{i} \in \mathbb{R}^{d}$ the position and  $V_{t}^{i} \in \mathbb{R}^{d}$ the velocity of the $i$th particle, $i \in \mathbb{N}_{+}$, at time $t$. In the absence of reactions, the dynamics for $\left(Q_{t}^{i},\Vti \right)$ are governed by Eq \eqref{E:LD0}, i.e. the underdamped Langevin dynamics
\begin{align*}
    dQ^i_t &= V_t^i dt,\\
    dV_t^i &= -b^{i} V_t^i dt + b^{i} \sqrt{2D^i} dW_t^i, \numberthis \label{E:LD0}
\end{align*}
where $\{W_{t}^{n}\}_{n \in \mathbb{N}_{+}}$ is a countable collection of standard independent Brownian motions in $\mathbb{R}^{d}$.

A particle's state can be represented as a vector in $\hat{P}=\mathbb{R}^{d} \times \Rd \times \mathcal{S}$, the combined space encoding particle position, velocity, and type. 

We now formulate our representation for the (number) concentration, equivalently number density, fields of each species. Consider a complete metric space $E$ and $M(E)$ the collection of measures on $E$. Let $\mathcal{M}(E)$ be the subset of $M(E)$ consisting of all finite, non-negative point measures of the form
\begin{align*}
\mathcal{M}(E)&=\left\{\sum_{i=1}^{N}\delta_{z^i}, N\geq 1, z^1\cdots, z^N\in E\right\}.
\end{align*}
We will typically choose $E = \Rd \times \Rd$ to represent the position and velocity space of particles, but we will also use $E = \hat{P} := \Rd \times \Rd \times \mathcal{S}$ to represent the combined position, velocity and type space.
For $f: E \mapsto \mathbb{R}$ and $\mu \in M(E)$, define
$$
\langle f, \mu\rangle_{E}=\int_{z \in E} f(z) \mu(dz).
$$
In the most common case that $E=\mathbb{R}^{d}\times \Rd$, we will often omit the subscript $E$ and simply write $\langle f, \mu\rangle := \int_{\mathbb{R}^{2d}} f(x,v) \mu(dx, dv)$. For each $t \geq 0$, we define the density of particles in the system at time $t$ by the distribution
\begin{equation} \label{ncen}
\nu_{t}=\sum_{i=1}^{N(t)} \delta_{(Q_{t}^{i},V_{t}^{i})}\delta_{p_{i}}\in \hat{P}
\end{equation}
where borrowing notation from \cite{Bansaye2015}, $N(t)=\langle 1, \nu_{t}\rangle$ represents the stochastic process for the total number of particles at time $t$. To investigate the behavior of different types of particles, we denote the marginal distribution on the $j$th type, i.e., the concentration field for species $j$, by
$$
\nu_{t}^{j}(\cdot)=\nu_{t}(\cdot \times\{S_{j}\})
$$
a distribution on $\mathbb{R}^{d} \times \Rd$. $N_{j}(t)=\langle 1, \nu_{t}^{j}\rangle$ will similarly label the total number of particles of type $S_{j}$ at time $t$. For $\nu$ any fixed particle distribution of the form \eqref{ncen}, we will also use an alternative representation in terms of the marginal distributions \textcolor{black}{$\nu^j\in \mathcal{M}(\R^d\times \Rd)$ for particles of type $j$,
\begin{equation} \label{eq:densitymeasdefmargrep}
\nu = \sum_{j = 1}^{J} \nu^j\delta_{S_j}  \in \mathcal{M}(\hat{P}).
\end{equation}}

Note that in the remainder, in any rigorous calculation $\nu_{t}$ and $\nu_{t}^{j}$ will be measures and treated as such. We will, however, abuse notation and also refer to them as concentration fields (though it would be better to refer to them as number densities since they are functions of space and velocity). Strictly speaking, the latter should refer to the densities associated with such measures, but we ignore this distinction in the subsequent discussions.

In our rigorous large population limit, we will be taking a simultaneous limit of two parameters. One will be the population scaling parameter $\gamma \to \infty$, and the other will be a displacement range parameter $\eta \to 0$. The parameter $\gamma$ can physically represent Avogadro's number or the system size, while $\eta$'s role is to mollify delta Dirac distributions that arise in modeling the placement measures for reaction product particles. See Section~\ref{S:MainAssumptions} for rigorous definitions of the latter. In particular, this dual limit is encoded via the vector limit parameter
\begin{equation*}
\zeta := \left(\frac{1}{\gamma}, \eta\right) \to 0.
\end{equation*}
To study this limit we will work with rescaled measures for each species, denoted by
\begin{equation*}
    \mu_{t}^{\zeta, j} := \frac{1}{\gamma} \nu_{t}^{\zeta, j} := \frac{1}{\gamma} \nu_{t}^{j}, \quad j = 1,\dots,J.
\end{equation*}
When $\gamma$ corresponds to Avogadro's number, $\mu_{t}^{\zeta, j}$ physically corresponds to the measure for which the associated density would represent the molar concentration field for species $j$ at time $t$ after integrating out velocity (but we will again abuse notation and refer to $\mu_t^{\zeta,j}$ as the molar concentration field to indicate it is the $\gamma$-scaled measure). We similarly let
\begin{equation*}
    \mu_{t}^{\zeta} := \frac{1}{\gamma} \nu_{t}^{\zeta} := \sum_{j=1}^{J} \mu_{t}^{\zeta,j} \delta_{S_{j}},
\end{equation*}
and define the vector of the molar concentrations for each species by
\begin{equation*}
    \vmu^{\zeta}_t := (\mu^{\zeta,1}_t,\dots,\mu^{\zeta,J}_t).
\end{equation*}

In the remainder, we will often write $N^{\zeta}(t) = \langle 1, \gamma \mu_t^{\zeta} \rangle$ and $N^{\zeta}_j(t) = \langle 1, \gamma \mu_t^{\zeta,j}\rangle$ to make explicit that $N$ and $N_j$ depend on $\zeta$.

We now introduce state vectors to store the positions and velocities of particles of a given type. Define the particle index maps $\{\sigma_{j}(k)\}_{k=1}^{N_{j}(t)}$, which encode a fixed ordering for the position and velocity combinations of individual particles of species $j$, $(Q,V)^{\sigma_{j}(1)} \preceq \cdots \preceq (Q,V)^{\sigma_{j}(N_{j}(t))}$, arising from an (assumed) fixed underlying ordering on $\mathbb{R}^{d}\times \Rd$. We slightly abuse notation and write $(Q,V)^{\sigma_{j}(k)}$ as $(Q^{\sigma_{j}(k)},V^{\sigma_{j}(k)}),$ for $k  = 1, \cdots, N_{j}(t)$. Following the notation established in \cite{Bansaye2015} (see Section 6.3 therein), we let $\mathbb{N}^{*}=\mathbb{N}\setminus\{0\}$ and let $H=(H^{1}_{(Q,V)},\cdots, H^{k}_{(Q,V)},\cdots ):\mathcal{M}(\mathbb{R}^d \times \Rd)\mapsto (\mathbb{R}^{d}\times\Rd)^{\mathbb{N}^{*}}$. With some abuse of notation, we often write $H^{k}_{(Q,V)}=(H^{k}_{Q},H^{k}_{V})$, with $H^{k}_{Q},H^{k}_{V}$ the corresponding projected position and velocity variables respectively such that
\begin{align}
    H_Q\bigl(\nu_{t}^{j}\bigr) &\coloneqq \bigl(Q_{t}^{\sigma_{j}(1)}, \cdots, Q_{t}^{\sigma_{j}(N_{j}(t))}, 0,0, \cdots\bigr),\nonumber\\
        H_V\bigl(\nu_{t}^{j}\bigr) &\coloneqq \bigl(V_{t}^{\sigma_{j}(1)}, \cdots, V_{t}^{\sigma_{j}(N_{j}(t))}, 0,0, \cdots\bigr).\nonumber
\end{align}

$H_Q(\nu^j_t)$ represents the position state vector for type $j$ particles, and $H_V(\nu^j_t)$ similarly represents the velocity state vector for type $j$ particles. We analogously let $H_Q^{i}(\nu_{t}^{j}) \in \mathbb{R}^{d}$ label the $i$th entry of the vector $H_Q(\nu_{t}^{j})$ and $H_V^{i}(\nu_{t}^{j}) \in \mathbb{R}^{d}$ label the $i$th entry of the vector $H_V(\nu_{t}^{j})$. Note that the zero entries after the $Q_{t}^{\sigma_{j}(N_{j}(t))}$ and the $V_{t}^{\sigma_{j}(N_{j}(t))}$ terms merely serve as placeholders. Using this notation, we will often write the marginal distribution of species $j$ as
\begin{align} \label{mu}
    \mu_{t}^{\zeta, j}(dx, dv)
    &= \frac{1}{\gamma} \sum_{i=1}^{N^{\zeta}_j(t)} \delta_{\left(H_Q^{i}(\gamma \mu_{t}^{\zeta, j}),H_V^{i}(\gamma \mu_{t}^{\zeta, j})\right)}(d x, d v),
    \quad j \in\{1, \ldots, J\}.
\end{align}

As commented in \cite{Bansaye2015}, this function $H$ allows us to address a notational issue. In particular, choosing a particle of a type $j$ uniformly among all particles in $\nu^j_t \in\mathcal{M}(\mathbb{R}^d \times \Rd)$ amounts to choosing an index uniformly in the set $\{1,\cdots,\langle 1,\nu_t^j \rangle \}$, and then choosing the individual particle from the arbitrary fixed ordering. As particles of the same type are assumed to be indistinguishable, there is no ambiguity in the value of \(H(\nu_{t}^{j})\) when two particles of type \(j\) have the same position and velocity.

Analogously to \cite{IsaacsonSIMA2022}, we now introduce a system of notation to encode the positions and velocities of reaction substrate and product particles, as well as configuration spaces to specify reaction processes later on. Comparing with \cite{IsaacsonSIMA2022}, the definitions below also account for the new velocity component.

\begin{definition}\label{D:indexSpace}
To describe the dynamics of $\nu_t$, we will sample vectors containing the indices of the specific substrate particles participating in a single $\ell$-type reaction from the substrate index space
$$
\mathbb{I}^{(\ell)}=(\mathbb{N} \backslash\{0\})^{|\alpha^{(\ell)}|}.
$$
For the allowable reactions considered in this work, we label the elements of $\mathbb{I}^{(\ell)}$ according to their species types:
\begin{enumerate}
    \item For $\mathcal{R}_{\ell}$ of the form $\varnothing \rightarrow \cdots$
    $$
\mathbb{I}^{(\ell)}=\varnothing.
$$
    \item For $\mathcal{R}_{\ell}$ of the form $S_j \rightarrow \cdots$
    $$
\mathbb{I}^{(\ell)}=\{i_1^{(j)}\in\mathbb{N}\backslash\{0\}\}.
$$
    \item For $\mathcal{R}_{\ell}$ of the form $S_j + S_k \rightarrow \cdots$ with $j<k$
        $$
\mathbb{I}^{(\ell)}=\{(i_1^{(j)},i_1^{(k)})\in (\mathbb{N}\backslash\{0\})^2\}.
$$
    \item For $\mathcal{R}_{\ell}$ of the form $2S_j \rightarrow \cdots$
           $$
\mathbb{I}^{(\ell)}=\{(i_1^{(j)},i_2^{(j)})\in (\mathbb{N}\backslash\{0\})^2\}.
$$
\end{enumerate}
We write a particular sampled set of substrate indices $\bm{i}\in \mathbb{I}^{(\ell)}$ as
$$
\bm{i}=\bigl(i_{1}^{(1)}, \cdots, i_{\alpha_{\ell 1}}^{(1)}, \cdots, i_{1}^{(J)}, \cdots, i_{\alpha_{\ell J}}^{(J)}\bigr).
$$
\end{definition}

\begin{definition}\label{D:substrateSpace} We define the substrate particle position and velocity product space analogously to $\mathbb{I}^{(\ell)}$ as
$$\mathbb{X}^{(\ell)} \otimes \mathbb{V}^{(\ell)} \in (\mathbb{R}^d \times \Rd)^{|\alpha^{(\ell)}|},
$$
with $\xv \supseteq (\bm{x},\bm{v}) =\bigl((x_{1}^{(1)},{v}_{1}^{(1)}), \cdots, (x_{\alpha_{\ell 1}}^{(1)},{v}_{\alpha_{\ell 1}}^{(1)}), \cdots, (x_{1}^{(J)},{v}_{1}^{(J)}), \cdots, (x_{\alpha_{\ell J}}^{(J)},{v}_{\alpha_{\ell J}}^{(J)})\bigr)$. For $(\bm{x},\bm{v}) \in \mathbb{X}^{(\ell)} \otimes \mathbb{V}^{(\ell)}$, a sampled substrate position and velocity configuration for one individual $\mathcal{R}_{\ell}$ reaction, $x_{r}^{(j)}$ then labels the sampled position for the $r$th substrate particle of species $j$ involved in the reaction, and ${v}_{r}^{(j)}$ labels the sampled velocity for the same substrate particle. Let $d \bm{x}d \bm{v}=\bigl(\bigwedge_{j=1}^{J}(\bigwedge_{r=1}^{\alpha_{\ell j}} d x_{r}^{(j)}d v_{r}^{(j)})\bigr)$ be the corresponding volume form on $\mathbb{X}^{(\ell)}\otimes \mathbb{V}^{(\ell)}$ which also naturally defines an associated Lebesgue measure.
\end{definition}

\begin{definition}\label{D:productSpace}For reaction $\mathcal{R}_{\ell}$ with $\tilde{L}+1 \leq \ell \leq L$, i.e., having at least one product particle, define the product position and velocity product space analogously to $\yv$,
$$
\mathbb{Y}^{(\ell)} \otimes \mathbb{V'}^{(\ell)}\in (\mathbb{R}^d \times \Rd)^{|\beta^{(\ell)}|},
$$
where we write
\begin{align*}
\yv & \supseteq (\bm{y},\bm{v'}) =\bigl((y_{1}^{(1)},{v'}_{1}^{(1)}), \cdots, (y_{\beta_{\ell 1}}^{(1)},{v'}_{\beta_{\ell 1}}^{(1)}), \cdots, (y_{1}^{(J)},{v'}_{1}^{(J)}), \cdots, (y_{\beta_{\ell J}}^{(J)},{v'}_{\beta_{\ell J}}^{(J)})\bigr).
\end{align*}

 For $(\bm{y},\bm{v'}) \in \mathbb{Y}^{(\ell)} \otimes \mathbb{V'}^{(\ell)}$ a sampled product position and velocity configuration for one individual $\mathcal{R}_{\ell}$ reaction, $y_{r}^{(j)}$ then labels the sampled position for the $r$th product particle of species $j$ involved in the reaction, and ${v'}_{r}^{(j)}$ labels the sampled velocity for the same product particle. Let $d \bm{y} d\bm{v'}=\bigl(\bigwedge_{j=1}^{J}(\bigwedge_{r=1}^{\beta_{\ell_{j}}} d y_{r}^{(j)}d {v'}_{r}^{(j)})\bigr)$ be the corresponding volume form on $\yv$, which also naturally defines an associated Lebesgue measure.
\end{definition}

\begin{definition}\label{D:projectionMapping}Consider a fixed reaction $\mathcal{R}_{\ell}$, with $\bm{i}\in \mathbb{I}^{(\ell)}$ and $\nu$ corresponding to a fixed particle distribution given by \eqref{ncen} with representation \eqref{eq:densitymeasdefmargrep}. We define the $\ell$th projection mapping \textcolor{black}{$\mathcal{P}^{(\ell)} :   \mathcal{M}(\hat{P})\times  \mathbb{I}^{(\ell)} \rightarrow  \mathbb{X}^{(\ell)} \otimes \mathbb{V}^{(\ell)}$} as
$$
\mathcal{P}^{(\ell)}(\nu, \bm{i})=\bigl(H^{i_{1}^{(1)}}(\nu^{1}), \cdots, H^{i_{\alpha_{\ell 1}}^{(1)}}(\nu^{1}), \cdots, H^{i_{1}^{(J)}}(\nu^{J}), \cdots, H^{i_{\alpha_{\ell J}}^{(J)}}(\nu^{J})\bigr).
$$
When substrates with indices $\bm{i}$ in particle distribution $\nu$ are chosen to undergo a reaction of type $\ell, \mathcal{P}^{(\ell)}(\nu, \bm{i})$ then gives the vector of the corresponding substrate particles' positions and velocities. For simplicity of notation, in the remainder, we will sometimes evaluate $\mathcal{P}^{(\ell)}$ with inconsistent particle distributions and index vectors. In all of these cases the inconsistency will occur in terms that are zero, and hence not matter in any practical way.
\end{definition}

\begin{definition}\label{omega} Consider a fixed reaction $\mathcal{R}_{\ell}$, with $\nu$ a fixed particle distribution given by \eqref{ncen} with representation \eqref{eq:densitymeasdefmargrep}. Using the notation of Definition \ref{D:indexSpace}, we define the allowable substrate index sampling space $\Omega^{(\ell)}(\nu) \subset \mathbb{I}^{(\ell)}$ as
$$
\Omega^{(\ell)}(\nu)= \begin{cases}\varnothing, & |\bm{\alpha}^{(\ell)}|=0, \\ \{\bm{i}=i_{1}^{(j)} \in \mathbb{I}^{(\ell)} | i_{1}^{(j)} \leq \langle 1, \nu^{j}\rangle\}, & |\bm{\alpha}^{(\ell)}|=\alpha_{\ell j}=1, \\ \{i=\bigl(i_{1}^{(j)}, i_{2}^{(j)}\bigr) \in \mathbb{I}^{(\ell)} | i_{1}^{(j)}<i_{2}^{(j)} \leq\langle 1, \nu^{j}\rangle\}, & |\bm{\alpha}^{(\ell)}|=\alpha_{\ell j}=2, \\ \{i=(i_{1}^{(j)}, i_{1}^{(k)}) \in \mathbb{I}^{(\ell)} | i_{1}^{(j)} \leq\langle 1, \nu^{j}\rangle, i_{1}^{(k)} \leq\langle 1, \nu^{k}\rangle\}, & |\bm{\alpha}^{(\ell)}|=2, \alpha_{\ell j}=\alpha_{\ell k}=1, j<k .\end{cases}
$$
Note that in the calculations that follow $\Omega^{(\ell)}(\nu)$ will change over time due to the fact that $\nu=\nu_{t}$ changes over time, but this will not be explicitly denoted for notational convenience.
\end{definition}

\begin{definition}\label{D:6} Consider a fixed reaction $\mathcal{R}_{\ell}$, with $\nu$ any element of $M(\hat{P})$ with the representation \eqref{eq:densitymeasdefmargrep}. We define the $\ell$th substrate measure mapping $\lambda^{(\ell)}[\cdot]: M(\hat{P}) \rightarrow M(\xv)$ evaluated at $(\bm{x},\bm{v}) \in \xv$ via $\lambda^{(\ell)}[\nu](dx, dv)= \otimes_{j=1}^{J}\bigl(\otimes_{r=1}^{\alpha_{\ell j}} \nu^{j}(d x_{r}^{(j)},d v_{r}^{(j)})\bigr).$ \label{lambda}
\end{definition}

\begin{definition}\label{xtilde}
For reaction $\mathcal{R}_{\ell}$, define subspaces $\tilde{\mathbb{X}}^{(\ell)}\subset \mathbb{X}^{(\ell)}$ and $\tilde{\mathbb{V}}^{(\ell)}\subset \mathbb{V}^{(\ell)}$ by removing all particle substrate position and velocity vectors in $\mathbb{X}^{(\ell)}$ and $\mathbb{V}^{(\ell)}$ respectively for which two particles of the same species have the same position or the same velocity. That is
\begin{align*}
\tilde{\mathbb{X}}^{(\ell)}&=\mathbb{X}^{(\ell)} \backslash\{\bm{x}
 \in \mathbb{X}^{(\ell)} | x_{r}^{(j)}=x_{k}^{(j)} \textrm { for some } 1 \leq j \leq J, 1 \leq k \neq r \leq \alpha_{\ell j}\},\nonumber\\
\tilde{\mathbb{V}}^{(\ell)}&=\mathbb{V}^{(\ell)} \backslash\{\bm{v}
 \in \mathbb{V}^{(\ell)} | v_{r}^{(j)}=v_{k}^{(j)} \textrm { for some } 1 \leq j \leq J, 1 \leq k \neq r \leq \alpha_{\ell j}\}.\nonumber
\end{align*}

\end{definition}

\section{Assumptions}\label{S:MainAssumptions}

In this section we state the main assumptions of this work. The main result then follows in Section \ref{S:MainResult}.
\begin{assumption}\label{A:molarBdd}We assume that the total (molar) population concentration satisfies $\sum_{j=1}^{J}\langle 1, \mu_{t}^{\zeta, j}\rangle \newline \leq C_{\circ}$ for all $t<\infty$, i.e., is uniformly in time bounded by some constant $C_{\circ}<\infty$. To simplify constants in later proofs, we further take $C_{\circ} > 1$.
\end{assumption}

Assumption \ref{A:molarBdd} essentially implies that the reaction networks we consider have bounded concentration globally in time. In this paper, we restrict attention to those networks. Reaction networks that satisfy this assumption include the networks $A+B\leftrightarrows C+D$ and $A+B\leftrightarrows C$, see \cite{IsaacsonSIMA2022,Nolen2019,VeberAnnals2023}, but it is clear that not all reactions would satisfy such a property (for instance $A\rightarrow 2A$ does not satisfy it). It is a hard, but very interesting open research question to characterize the class of reaction networks that maintain the global-in-time well posedness.

\begin{assumption}\label{A:initialconvergence}
We assume that for all $1 \leq j \leq J$, the initial distribution $\mu_{0}^{\zeta, j} \rightarrow \xi_{0}^{j}$ weakly as $\zeta \rightarrow 0$, where $\xi_{0}^{j}$ is a compactly supported measure with finite mass.
\end{assumption}

Associated with the $\ell$th reaction are two functions, a reaction rate kernel $K_{\ell}(\bm{x},\bm{v})$ and a placement density $m_{\ell}(\bm{y},\bm{v'}|\bm{x},\bm{v})$. $K_{\ell}(\bm{x},\bm{v})$ encodes the probability per time substrates with space-velocity configuration $(\bm{x},\bm{v}) \in \mathbb{X}^{(\ell)} \otimes \mathbb{V}^{(\ell)}$ will react. $m_{\ell}(\bm{y},\bm{v'}|\bm{x},\bm{v})$ encodes the probability density for the reaction products to be placed at space-velocity configuration $(\bm{y},\bm{v'}) \in \mathbb{Y}^{(\ell)} \otimes \mathbb{V'}^{(\ell)}$ given that the substrates had space-velocity configuration $(\bm{x},\bm{v}) \in \mathbb{X}^{(\ell)} \otimes \mathbb{V}^{(\ell)}$. We now state assumptions on these two functions.
 \begin{assumption}\label{A:kernalBdd}
  We assume that for all $1 \leq \ell \leq L$, the reaction rate kernel $K_{\ell}(\bm{x},\bm{v})$ is uniformly bounded for all $(\bm{x},\bm{v}) \in \xv$, and denote generic constants dependent upon this bound by $C(K)$.
 \end{assumption}

Studying the large population limit, leads us to identify physical ways in which the reaction kernels depend on the large system parameter $\gamma$. We follow here the intuition developed in Appendix A of \cite{IsaacsonSIMA2022} and we require that the formal well mixed (i.e., infinitely fast diffusion) limit of the purely-diffusive volume reactivity particle-based stochastic reaction-diffusion model matches the corresponding classical spatially homogeneous stochastic chemical kinetic model. We refer the interested reader to Appendix A of \cite{IsaacsonSIMA2022} for the mathematical details. These considerations lead to Assumption \ref{A:KernelScaling}.
 \begin{assumption}\label{A:KernelScaling}
  The reaction kernel is assumed to have the explicit $\gamma$ dependence that
$$
K_{\ell}^{\gamma}(\bm{x},\bm{v})=\gamma^{1-|\bm{\alpha}^{(\ell)}|} K_{\ell}(\bm{x},\bm{v})
$$
for any $(\bm{x},\bm{v}) \in \xv, 1 \leq \ell \leq L$. 
 \end{assumption}

 \begin{assumption}\label{A:PlacementDensity}
We assume that for any $\eta \geq 0, 1 \leq \ell \leq L, (\bm{y},\bm{v'}) \in \yv$ and $(\bm{x},\bm{v}) \in \xv$, the placement density $m_{\ell}^{\eta}(\bm{y},\bm{v'}|\bm{x},\bm{v})$ is \textcolor{black}{uniformly bounded in $(\bm{x},\bm{v})$ and $(\bm{y},\bm{v'})$, and is a probability density in $(\bm{y},\bm{v'})$,} i.e., $\int_{\yv} m_{\ell}^{\eta}(\bm{y},\bm{v'}|\bm{x},\bm{v}) d \bm{y} d\bm{v'}=1$. \label{m}
\end{assumption}

To define placement densities $m_{\ell}^{\eta}(\bm{y}, \bm{v'}|\bm{x}, \bm{v})$ in terms of delta-functions in a mathematically rigorous way, we introduce the displacement (i.e., smoothing) range parameter $\eta$ in order to mollify the limiting Dirac delta densities $m_{\ell}(\cdot, \cdot | \bm{x},\bm{v})$ in the standard way.

 \begin{definition}
 For $x \in \mathbb{R}^{d}$, let $G(x)$ denote a standard positive mollifier and $G_{\eta}(x)=\eta^{-d} G(x / \eta)$. That is, $G(x)$ is a smooth function on $\mathbb{R}^{d}$ satisfying the following four requirements
 \begin{enumerate}
     \item $G(x) \geq 0$;
     \item $G(x)$ is compactly supported in $B(0,1)$, the unit ball in $\mathbb{R}^{d}$;
     \item $\int_{\mathbb{R}^{d}} G(x) d x=1$;
     \item $\displaystyle\lim_{\eta \rightarrow 0} G_{\eta}(x)=\displaystyle\lim_{\eta \rightarrow 0} \eta^{-d} G(x / \eta)=\delta_{0}(x)$, where $\delta_{0}(x)$ is the Dirac delta function and the limit is taken in the space of Schwartz distributions.
 \end{enumerate}
 \end{definition}

The allowable forms of the placement density for each possible reaction are summarized below. The main principle of the placement rule is that the products of a reaction are placed based on the line connecting the substrates, with their velocities determined by enforcing the conservation of total momentum. Although other choices are possible, such as the exchange of momentum, conservation of kinetic energy, etc., and can be included in our theory, we focus the discussion here on this specific rule as stated in Assumption \ref{A:PlacementDensity} due to its concreteness and physical relevance.
   \begin{assumption}\label{A:PlacementDensity}
The distributional limit of $m_{\ell}^{\eta}(\bm{y},\bm{v'}|\bm{x},\bm{v})$ as $\eta \rightarrow 0$ is given by $m_{\ell}(\bm{y},\bm{v'}|\bm{x},\bm{v})$, a linear combination of Dirac delta functions, for any $(\bm{y},\bm{v'}) \in \yv$ and $(\bm{x},\bm{v}) \in \xv$. Below we summarize the specific forms of these placement densities for each possible reaction type. Note that $m_i$ will denote the mass of the particle with position-velocity $(x_i,v_i)$ below. The specific forms of the velocity kernels are motivated by our work in~\cite{SamChenLanlanKostas2025}, where we established a set of kernels for the types of reactive Langvein dynamics models we consider here that both satisfy detailed balance for reversible reactions, and are consistent in the over-damped limit with the corresponding kernels for the purely diffusive volume reactivity particle-based stochastic reaction-diffusion model studied in~\cite{IsaacsonSIMA2022}. See~\cite{SamChenLanlanKostas2025} for further details and discussion on the motivation for these choices.
\begin{enumerate}
    \item For a first order reaction $\mathcal{R}_{\ell}$ of the form $S_i \rightarrow S_j$,
 we assume that the placement density $m_{\ell}^{\eta}(y,v' | x,v)$ takes the mollified form of
$$
m_{\ell}^{\eta}(y,v' | x,v)=G_{\eta}(y-x)G_{\eta}(v'-v),
$$
with the distributional limit as $\eta \rightarrow 0$ given by
$$
m_{\ell}(y,v' | x,v)=\delta_{x}(y)\delta_{v}(v').
$$
\item For a first-order reaction $\mathcal{R}_{\ell}$ of the form $S_i \rightarrow S_j + S_k$,  we assume that the unbinding displacement density $m_{\ell}^{\eta}(y_1, y_2,v_1',v_2' | x,v)$ is in the mollified form of
\begin{align*}
m_{\ell}^{\eta}(y_1, y_2,v_1',v_2' | x,v)&=\rho\left(|y_1-y_2|,|v_1'-v_2'|\right) \sum_{i=1}^{I} p_{i} G_{\eta}\left(x-\left(\alpha_{i} y_1+(1-\alpha_{i}) y_2\right)\right)\\
&\phantom{=}\times G_{\eta}\left(v-\frac{m_1 v_1'+ m_2 v_2'}{m_3}\right),
\end{align*}
where $m_3$ denotes the mass of the substrate $(x,v)$ particle.
The distributional limit as $\eta \rightarrow 0$ is given by
\begin{align*}
m_{\ell}(y_1, y_2,v_1',v_2' | x,v)&=\rho\left(|y_1-y_2|,|v_1'-v_2'|\right) \sum_{i=1}^{I} p_{i} \delta\left(x-\left(\alpha_{i} y_1+(1-\alpha_{i}) y_2\right)\right)\\
&\phantom{=}\times\delta\left(v-\frac{m_1 v_1'+ m_2 v_2'}{m_3}\right),
\end{align*}
with $p_i, \alpha_i \in [0,1],$ for $i \in \{1,\cdots,I\}$ and $\displaystyle\sum_{i=1}^{I} p_{i}=1$.
Here we assume the spatial separation and velocity difference of the product $S_{j}$ and $S_{k}$ particles, $|y_1 - y_2|$ and $|v_1'-v_2'|$ are sampled from the joint probability density $\rho(|y_1 - y_2|,|v_1'-v_2'|)$.
\item For a second order reaction $\mathcal{R}_{\ell}$ of the form $S_i + S_k \rightarrow S_j,$ we assume that the binding placement density $m_{\ell}(y, v'| x_1, x_2,v_1,v_2)$ takes the mollified form of
$$
m_{\ell}^{\eta}(y, v'| x_1, x_2,v_1,v_2)=\sum_{i=1}^{I} p_{i} G_{\eta}\left(y-\left(\alpha_{i} x_1+(1-\alpha_{i}) x_2\right)\right)G_{\eta}\left(v'-\frac{m_1 v_1+ m_2 v_2}{m_3}\right),
$$
where $m_3$ denotes the product particle's mass. The distributional limit as $\eta \rightarrow 0$ is given by
$$
m_{\ell}(y, v'| x_1, x_2,v_1,v_2)=\sum_{i=1}^{I} p_{i} \delta\left(y-\left(\alpha_{i} x_1+(1-\alpha_{i}) x_2\right)\right)\delta \left(v'-\frac{m_1 v_1+ m_2 v_2}{m_3}\right),
$$
with $p_i, \alpha_i \in [0,1],$ for $i \in \{1,\cdots,I\}$ and $\displaystyle\sum_{i=1}^{I} p_{i}=1$.


\item For a second order reaction $\mathcal{R}_{\ell}$ of the form $S_i + S_k \rightarrow S_j + S_r$,we assume that the placement density $m_{\ell}(y_1,y_2,v'_1,v'_2 | x_1, x_2,v_1,v_2)$ takes the mollified form of
\begin{align*}
    &m_{\ell}^{\eta}(y_1,y_2,v'_1,v'_2 | x_1, x_2,v_1,v_2)\\=&(m_3 + m_4)^{d} \rho(|v_1' - v_2'|)G_{\eta}(m_3 v_1' + m_4 v_2' - m_1 v_1 - m_2 v_2)\\
&\phantom{=} \times \biggl[p \times G_{\eta}(x_1-y_1) G_{\eta}(x_2-y_2)+(1-p) \times G_{\eta}(x_1-y_2) G_{\eta}(x_2-y_1)\biggr],
\end{align*}
for $m_3$ and $m_4$ denoting the masses of the product particles. The distributional limit as $\eta \rightarrow 0$ is given by
\begin{align*}
&m_{\ell}(y_1,y_2,v'_1,v'_2 | x_1, x_2,v_1,v_2)\\=&(m_3 + m_4)^{d} \rho(|v_1' - v_2'|)\delta(m_3 v_1' + m_4 v_2' - m_1 v_1 - m_2 v_2)\\
&\phantom{=} \times \biggl[p \times \delta_{(x_1, x_2)}\left((y_1, y_2)\right)+(1-p) \times \delta_{(x_1, x_2)}\left((y_2, y_1)\right)\biggr],
\end{align*}
with $p \in [0,1].$
Here we assume the velocity difference of the product $S_{j}$ and $S_{r}$ particles, $|v_1'-v_2'|$ are sampled from the marginal probability density $\rho(|v_1'-v_2'|)$, which satisfies Assumption \ref{A:AssumptionRho}.
\end{enumerate}
 \end{assumption}

In order for Assumptions \ref{A:PlacementDensity} to be compatible, we now further impose Assumption \ref{A:AssumptionRho} on the joint and marginal probability densities. We slightly abuse notation and denote both the joint probability density and the marginal probability density with $\rho$.

\begin{assumption}\label{A:AssumptionRho}
We abuse notation and let $\rho$ denote either the joint probability density $\rho(w,u)$ on $\R^d\times\R^d$ or the velocity only probability density $\rho(u)$ on $\R^d$ as in Assumption~\ref{A:PlacementDensity}, depending on the given reaction.  In each case we assume that $\rho$  is even and is normalized to one, i.e.,
\[
\int_{\R^d\times\R^d}\rho(w,u)dwdu = 1,
\qquad
\int_{\R^d}\rho(u)du = 1.
\]
Additionally, since $\rho$ is a probability density, the previous condition implies that the tail probability $\rho$
\[
\int_{|u|>R}\rho(u)du < \epsilon,
\qquad
\int_{|w|+|u|>R}\rho(w,u)dwdu < \epsilon,
\]for any $\epsilon >0$ when $R$ is chosen sufficiently large. Finally, we assume a finite fourth moment
\[
\int_{\R^d\times\R^d}\bigl(|w|^4 + |u|^4\bigr)\rho(w,u)dwdu<\infty.
\]
\end{assumption}

 \begin{assumption} \label{A:conservationofMomentum}
     For the reversible reaction $\mathcal{R}_{\ell}$ of the form $S_i + S_k \leftrightarrows S_j,$ we assume conservation of mass,
$m_i + m_k = m_j$,
and conservation of momentum,
$m_i v_i + m_k v_k = m_j v_j$. These properties ensure that the product particle is located on the line connecting the two substrates and that the magnitude of its velocity is bounded by the magnitudes of the velocities of the two substrates. We use the two conservation assumptions in the proof of Lemma \ref{L:boundoffm}.
For general bimolecular reactions $S_i + S_k \leftrightarrows S_j + S_r$, we impose conservation of mass $m_i + m_k = m_j + m_r$ and conservation of momentum $m_i v_{i}+ m_k v_{k}= m_j v_{j}+ m_r v_{r}$.
 \end{assumption}

\begin{assumption}\label{A:AssumptionMomentFiniteness}
We assume that $\sum_{j=1}^{J}\mathbb{E}\la |x|^{4} +|v|^{4},\mu_{t}^{\zeta, j}(dx,dv)\ra<\infty$ for $t \in [0,T]$. For $t=0$ specifically, we assume that $\displaystyle\sup_{\zeta\in\mathbb{R}_{+} \times \mathbb{R}_{+}}\sum_{j=1}^{J}\mathbb{E}\la |x|^{4} +|v|^{4},\mu_{0}^{\zeta, j}(dx,dv)\ra<C<\infty$.
\end{assumption}

We note here that Assumption \ref{A:AssumptionMomentFiniteness} is slightly stronger than it needs to be. Instead of $4-$moment bounds, we only need $(2+\epsilon)-$moment bounds, for some $\epsilon>0$. Without loss of generality, we, however, proceed with Assumption \ref{A:AssumptionMomentFiniteness} as it makes the presentation of the arguments easier and leads to less notation cluttering.

\begin{definition}
For a complete measurable space $E$, we define the variation norm of finite measures $\|\cdot\|_{M_{F}(E)}$ on $M_{F}(E)$ as
$$\|\nu\|_{M_{F}(E)}=\displaystyle\sup_{f \in L^{\infty}(E),\|f\|_{L^{\infty}} \leq 1}|\langle f, \nu\rangle_{E}|.$$ One can show via a density argument that an equivalent formulation is (see step 4 of Theorem 3.2 of \cite{Jourdain2012}) $$\|\nu\|_{M_{F}(E)}=\displaystyle\sup _{f \in C_{b}^{2}(E),\|f\|_{L^{\infty}} \leq 1}|\langle f, \nu\rangle_{E}|.$$
\end{definition}

\section{Main Result on the Mean Field Limit}\label{S:MainResult}
We recall by (\ref{E:LD0}) that, in the Langevin dynamics-based (LD) models and in the absence of reactions, each particle moves according to
\begin{align*}
    dQ^i_t &= V_t^i dt,\\
    dV_t^i &= -b^{i} V_t^i dt + b^{i} \sqrt{2D^i} dW_t^i. \numberthis \label{E:LD}
\end{align*}
In addition, in our model, molecules can undergo $L$ possible reactions, denoted as $\mathcal{R}_{1}, \cdots, \mathcal{R}_{L}$. Without loss of generality, we will examine the mean field limit for up to second-order reactions with products, though for simplification of the exposition we will only consider the four reaction types summarized in Assumption \ref{A:PlacementDensity} and ignore zero order reactions.

Consider the empirical measure $\nu_t^\zeta = \sum_{i=1}^{N^\zeta(t)}\delta_{Q^{i}_{t},V^{i}_{t}} \dpi.$ We now formulate a weak representation for the time evolution of scaled empirical measures $\mu_{t}^{\zeta, j} = \frac{1}{\gamma}\ntzj$, $j=1, \dots, J$. Recall that $\HQ,\HV $ denote the projection of $H^{i}_{Q,V}(\gamma\mu_{s^-}^{\zeta, j})$ to the position and velocity components of the $i$th particle of species $j$ respectively. We need to specify the noise terms that drive the particles' random motion and reactions. Let $\{W_t^n\}_{n \in \mathbb{N}_+}$ be a countable collection of standard independent Brownian motions in $\mathbb{R}^d$ that will drive particles' random velocity dynamics. Associated with the $\ell$th reaction, $\mathcal{R}^\ell$, we have a Poisson random measure $dN_{\ell}(s,\bm{i},\bm{y},\bm{v'},\theta_1,\theta_2)$ on $\mathbb{R}_{+} \times \mathbb{I}^{(\ell)} \times \yv \times \mathbb{R}_{+}^2$ with intensity measure $d\bar{N}(s,\bm{i},\bm{y},\bm{v'},\theta_1,\theta_2)  = ds  \left(\bigwedge_{j=1}^J \bigwedge_{r = 1}^{\alpha_{\ell j}} (\sum_{k \geq 0} \delta_k(i_r^{(j)}))\right) d\bm{i} \, d\bm{y} \, d\bm{v'} \, d\theta_1 \, d\theta_2$. Here $\mathbb{I}^{(\ell)}$ is the set of all possible substrate particle index combinations for reaction $\mathcal{R}_{\ell}$, and $\bm{i}$ denotes a specific substrate particle index combination in $\mathbb{I}^{(\ell)}$. $s$ denotes time, and $(\bm{y},\bm{v}')$ the product particle positions and velocities. The Poisson random measure $N_{\ell}$ encodes the stochasticity in the reaction dynamics of $\mathcal{R}_{\ell}$, with $\theta_1$ and $\theta_2$ being used to determine whether a reaction event occurs based on the reaction kernel and placement density respectively. We will later make use of the martingales, $d\tilde{N}_{\ell}(s,\bm{i},\bm{y},\bm{v'},\theta_1,\theta_2) = dN_{\ell}(s,\bm{i},\bm{y},\bm{v'},\theta_1,\theta_2) - d\bar{N}(s,\bm{i},\bm{y},\bm{v'},\theta_1,\theta_2)$, in the proof of the main Theorem~\ref{T:MainTheorem}.

For a test function $f(x_i,v_i) \in C^2_b(\Rd \times \Rd),$ and for each species $j=1, \ldots, J$, a weak representation of the dynamics of $\mtzj$ is given by
\begin{align*}
 \left\langle f, \mu_{t}^{\zeta, j}\right\rangle
&= \fmi + \ginverse \sum_{i\geq1} \int_{0}^{t} \mathbbm{1}_{\{i \leq \all\}}b_{j}\sqrt{2 D_{j}} \pfpv\left(\HQ,\HV \right) d W_{s}^{i}\\
&\phantom{=} +\ginverse \int_0^t \sumall \biggl[ \HV \pfpq \left(\HQ,\HV \right)  \\
&\phantom{=} -b_{j} \HV \pfpv \left(\HQ,\HV \right) + b_{j}^2 D_j \ppfppv \left(\HQ,\HV \right) \biggr] ds\\
&\phantom{=} +\sum_{\ell=1}^{L} \int_{0}^{t} \int_{\mathbb{I}^{(\ell)}} \int_{\yv} \int_{\mathbb{R}_{+}^2}\biggl(\left\langle f, \mu_{s^-}^{\zeta, j}-\frac{1}{\gamma} \sum_{r=1}^{\alpha_{\ell j}} \delta_{H^{i_{r}^{(j)}}_Q(\gamma \mu_{s^-}^{\zeta, j})}\delta_{H^{i_{r}^{(j)}}_V(\gamma \mu_{s^-}^{\zeta, j})}+\ginverse\sum_{r=1}^{\beta_{\ell j}}\delta_{y_{r}^{(j)}}\delta_{{v'}_{r}^{(j)}} \right\rangle\\
&\phantom{=} -\left\langle f, \mu_{s^-}^{\zeta, j}\right\rangle\biggr)\times \mathbbm{1}_{\{\bm{i} \in \Omega^{(\ell)}(\gamma\mu_{s^-}^{\zeta})\}} \times \mathbbm{1}_{\{\theta_1 \leq K_{\ell}^{\gamma}\left(\bm{x},\bm{v}\right)\}} \times \mathbbm{1}_{\{\theta_2 \leq m_{\ell}^{\eta}\left(\bm{y},\bm{v'}|\bm{x},\bm{v}\right)\}}d N_{\ell}(s,\bm{i},\bm{y},\bm{v'},\theta_1,\theta_2).\numberthis\label{Eq:PathDescription}
\end{align*}

For a test function $f \in C_{b}^{2}(\Rd \times \Rd)$ and for each species $j=1, \ldots, J$, we define the following generator for the diffusion of a particle of type $i$ at position $x_i$ with velocity $v_i$
\begin{align}
(\mathcal{L}_{i}f)(x_i,v_i)&\coloneqq v_i \nabla_{x_i}f(x_i,v_i) - b_{i} v_i \nabla_{v_i}f(x_i,v_i) + b_{i}^2 D_i \Delta_{v_i} f(x_i,v_i),\label{Eq:Operators}
\end{align}
and the corresponding formal adjoint operator,
\begin{align}
    (\mathcal{L}^*_if)(x_i,v_i)&\coloneqq -v_i \nabla_{x_i}f(x_i,v_i) + b_{i} \nabla_{v_i}\cdot [v_i f(x_i,v_i)  + b_{i} D_i \nabla_{v_i}f(x_i,v_i)].\label{Eq:AdjointOperators}
\end{align}

In this paper, we prove convergence of the measure-valued processes $\{\mu_{t}^{\zeta, j}\}_{t \in[0, T]}, j=1,2, \cdots, J$ as $\zeta \to 0$ over an appropriate space of c\`{a}dl\`{a}g measure-valued processes. Before stating the main result of this paper, Theorem \ref{T:MainTheorem}, let us briefly motivate its conclusion. In the following we abuse notation by assuming $\mathcal{P}^{(\ell)}(\gamma \mu_{s-}^{\zeta}, i)$ produces particle state arguments consistent for the context in which it is being used (i.e. $(\bm{x},\bm{v})$ in some contexts and interlaced particle configuration states, $(x_1,v_1, \cdots, x_{\alpha_{\ell J}}, v_{\alpha_{\ell J}})$, in other contexts). By taking expectation on \eqref{Eq:PathDescription}, and proper rewriting,  we have
   \begin{align*}
    \mathbb{E}&\left[\left\langle f, \mu_{t}^{\zeta, j}\right\rangle\right]
      =\mathbb{E}\left[\fmi\right]\nonumber\\
      &+\mathbb{E}\biggr[\ginverse \int_0^t \sumall \bigg< \biggl( v \pfpq \left(x,v \right) -b_{j} v \pfpv \left(x,v \right)+ b_{j}^2 D_j \ppfppv \left(x,v \right) \biggr),\\
      &\phantom{=}  \qquad \qquad \qquad \qquad\qquad\qquad \qquad \qquad \qquad \qquad \qquad \qquad \delta_{H^{i}_Q(\gamma \mu_{s^-}^{\zeta, j})}(dx)\delta_{H^{i}_V(\gamma \mu_{s^-}^{\zeta, j})}(dv)\bigg> ds\biggl]\\
      &\phantom{=} + \sum_{\ell=1}^{L}\mathbb{E}\biggl[\int_0^t\int_{\yv} \ginverse \displaystyle \sum_{\bm{i} \in \Omega^{(\ell)}(\gamma\mu_{s^-}^{\zeta})}\left(-\sum_{r=1}^{\alj}f\left(\HQirj,\HVirj\right)+\sum_{r=1}^{\beta_{\ell j}}f(y_{r}^{(j)},{v'}_{r}^{(j)}) \right)\\
      &\phantom{=} \qquad\qquad\qquad \qquad \qquad \qquad \qquad \times K_{\ell}^{\gamma}\left(\pl\right)m_{\ell}^{\eta} \left(\bm{y},\bm{v'}|\pl \right)d\bm{y} d\bm{v'}ds\biggr]\\
      &=\mathbb{E}\left[\fmi\right]+\mathbb{E}\left[\int_0^t \left\langle (\mathcal{L}_{j}f)(x,v), \msmzj(dx,dv)\right\rangle ds\right]\\
      &\phantom{=} - \sum_{\ell=1}^{L}\mathbb{E}\biggl[\int_0^t\ginverse \displaystyle \sum_{\bm{i} \in \Omega^{(\ell)}(\gamma\mu_{s^-}^{\zeta})}K_{\ell}^{\gamma}\left(\pl\right)\left(\sum_{r=1}^{\alj}f\left(\HQirj,\HVirj\right)\right)ds \biggr]\\
      &\phantom{=} + \sum_{\ell=1}^{L}\mathbb{E}\biggl[\int_0^t \ginverse \sum_{\bm{i} \in \Omega^{(\ell)}(\gamma\mu_{s^-}^{\zeta})} K_{\ell}^{\gamma}\left(\pl\right) \biggl( \int_{\yv} \displaystyle \left(\sum_{r=1}^{\beta_{\ell j}}f(y_{r}^{(j)},{v'}_{r}^{(j)}) \right)\\
      &\phantom{=} \times m_{\ell}^{\eta} \left(\bm{y},\bm{v'}|\pl \right)d\bm{y} d\bm{v'}\biggr)ds\biggr]\\
      &= \mathbb{E}\left[\fmi\right] + \mathbb{E}\left[\int_0^t \left\langle (\mathcal{L}_{j}f)(x,v), \msmzj(dx,dv)\right\rangle ds\right]\\
      &\phantom{=} - \sum_{\ell=1}^{L}\mathbb{E}\biggl[\int_0^t \ginverse \int_{\xv}\displaystyle \sum_{\bm{i} \in \Omega^{(\ell)}(\gamma\mu_{s^-}^{\zeta})}\klg(\bm{x},\bm{v})\left(\sum_{r=1}^{\alj}f\left(x_r^{(j)},v_r^{(j)}\right)\right)\delta_{\pl}(d\bm{x}, d\bm{v})ds\biggr]\\
      &\phantom{=} + \sum_{\ell=1}^{L}\mathbb{E}\biggl[\int_0^t \ginverse \int_{\xv}\displaystyle \sum_{\bm{i} \in \Omega^{(\ell)}(\gamma\mu_{s^-}^{\zeta})}\klg(\bm{x},\bm{v})\biggl(\int_{\yv}\left(\sum_{r=1}^{\blj}f\left(y_r^{(j)},{v'}_r^{(j)}\right)\right)\\
&\phantom{=}\times m_{\ell}^{\eta}\left(\bm{y},\bm{v'}|\bm{x},\bm{v})\right)d\bm{y}d\bm{v'}\biggr)\delta_{\pl}(d\bm{x}, d\bm{v})ds\biggr]\\
            &= \mathbb{E}\left[\fmi\right] + \mathbb{E}\left[\int_0^t \left\langle (\mathcal{L}_{j}f)(x,v), \msmzj(dx,dv)\right\rangle ds\right]\\
      &\phantom{=} +\sum_{\ell=1}^{L}\mathbb{E}\biggl[\int_0^t \int_{\txtv}\frac{1}{\bm{\alpha}^{(\ell)}!}K_\ell(\bm{x},\bm{v})\biggl(-\sum_{r=1}^{\alj}f\left(x_r^{(j)},v_r^{(j)}\right)+\int_{\yv}\left(\sum_{r=1}^{\blj}f\left(y_r^{(j)},{v'}_{r}^{(j)}\right)\right)\\
&\phantom{=}\times m_{\ell}^{\eta}\left(\bm{y},\bm{v'}|\bm{x},\bm{v}\right)d\bm{y}d\bm{v'}\biggr)\lmeasure ds\biggr].
\end{align*}
We arrive at the relation
   \begin{align*}
    \mathbb{E}&\left[\left\langle f, \mu_{t}^{\bm{\zeta}, j}\right\rangle\right]
      = \mathbb{E}\left[\fmi\right] + \mathbb{E}\left[\int_0^t \left\langle (\mathcal{L}_{j}f)(x,v), \msmzj(dx,dv)\right\rangle ds\right]\\
      &\phantom{=} +\sum_{\ell=1}^{L}\mathbb{E}\biggl[\int_0^t \int_{\txtv}\frac{1}{\bm{\alpha}^{(\ell)}!}K_\ell(\bm{x},\bm{v})\biggl(-\sum_{r=1}^{\alj}f\left(x_r^{(j)},v_r^{(j)}\right)+\int_{\yv}\left(\sum_{r=1}^{\blj}f\left(y_r^{(j)},{v'}_{r}^{(j)}\right)\right)\\
&\phantom{=}\times m_{\ell}^{\eta}\left(\bm{y},\bm{v'}|\bm{x},\bm{v}\right)d\bm{y}d\bm{v'}\biggr)\lmeasure ds\biggr],
\end{align*}
which motivates the main theorem of this paper, Theorem \ref{T:MainTheorem}. To state Theorem \ref{T:MainTheorem}, we also need to define $M_{F}(\Rd \times \Rd)$,  the space of finite measures endowed with the weak topology. We prove convergence of the measure-valued processes $\{\mu_{t}^{\zeta, j}\}_{t \in[0, T]}, j=1,2, \cdots, J$ on $\mathbb{D}_{\otimes_{j=1}^{J} M_{F}(\Rd \times \Rd)}([0, T])$, the space of c\`adl\`ag paths with values in $\otimes_{j=1}^{J} M_{F}(\Rd \times \Rd)$ endowed with the Skorokhod topology, see \cite{EthierKurtz}.
 \begin{theorem}\label{T:MainTheorem}
 (Mean field large-population limit) Let $0<T<\infty$ be given and let Assumptions \ref{A:molarBdd}-\ref{A:AssumptionMomentFiniteness} hold. Then, the sequence of measure-valued processes $\{\vmu_{t}^{\zeta}\}_{t \in[0, T]} \in \mathbb{D}_{\otimes_{j=1}^{J} M_{F}(\mathbb{R}^{d}\times\mathbb{R}^{d})}([0, T])$ is relatively compact in $\mathbb{D}_{\otimes_{j=1}^{J} M_{F}(\mathbb{R}^{d}\times\mathbb{R}^{d})}([0, T])$ and it converges in distribution to $\{\vxi_t\}_{t \in [0,T]} \in C_{\otimes_{j=1}^{J} M_{F}(\mathbb{R}^{d}\times\mathbb{R}^{d})}([0, T])$ as $\zeta \rightarrow 0$, where each $\xi_t^j$ is the unique solution to
\begin{align}\label{Eq:limitingMeasures}
   &\left\langle  f, \xi_{t}^{j}\right\rangle = \left\langle f, \xi_0^{j}\right\rangle + \int_0^t \left\langle (\mathcal{L}_{j}f)(x,v), \xi_{s}^{j}(dx,dv)\right\rangle ds\notag\\
      &\phantom{=} + \sum_{\ell=1}^{L}\int_0^t \int_{\txtv}\frac{1}{\bm{\alpha}^{(\ell)}!}K_\ell(\bm{x},\bm{v})\biggl(-\sum_{r=1}^{\alj}f\left(x_r^{(j)},v_r^{(j)}\right)+\int_{\yv}\left(\sum_{r=1}^{\blj}f\left(y_r^{(j)},{v'}_{r}^{(j)}\right)\right)\\
      &\phantom{=}\times m_{\ell}\left(\bm{y},\bm{v'}|\bm{x},\bm{v}\right)d\bm{y}d\bm{v'}\biggr)\llmeasure ds.\notag
\end{align}
\end{theorem}

In the absence of velocities, i.e. overdamped motion where particle configurations are determined solely by their positions, an analogous theorem has been proven in \cite{IsaacsonSIMA2022}. Including velocity not only leads to a more physical model—since particles naturally have both position and velocity—but also introduces additional mathematical considerations. In particular, as we showed in Assumption \ref{A:PlacementDensity}, rules for velocity placement need to be specified, which necessitates properly handling and bounding of the resulting terms.

\begin{remark}
If the limiting measures $\vxi_t = \left(\xi_{t}^{1}(dx, dv), \xi_{t}^{2}(dx, dv), \ldots, \xi_{t}^{J}(dx, dv)\right)$ have marginal densities $\vrho(x,v, t) \coloneqq \left(\rho_1(x,v,t),\dots,\rho_J(x,v,t)\right)$, then these marginals solve, in a weak sense, the following kinetic PIDEs
\begin{align*}
\partial_{t} &\rho_{j}(x, v, t)=-v \nabla_{x}\rho_{j}(x, v, t) + b_{j} \nabla_{v}\cdot \left[v \rho_{j}(x, v, t) + b_{j} D_j \nabla_{v}\rho_{j}(x, v, t)\right]\\
&\phantom{=} +\sum_{\ell=1}^{L}\int_{\txtv} \frac{1}{\bm{\alpha}^{(\ell)}!} K_{\ell}(\bm{\tilde{x}},\bm{\tilde{v}})\biggl(-\sum_{r=1}^{\alpha_{\ell j}}\delta_x(\txrj)\delta_v(\tvrj)+\int_{\yv}\left(\sum_{r=1}^{\blj}\delta_x(\yrj)\delta_v(\vprj)\right)\\
&\phantom{=}\times m_{\ell}\left(\bm{y},\bm{v'}|\bm{\tilde{x}},\bm{\tilde{v}}\right)d\bm{y}d\bm{v'}\biggr) \left(\Pi_{k=1}^{J} \Pi_{s=1}^{\alpha_{\ell k}} \rho_{k}(\tilde{x}_{s}^{(k)}, \tilde{v}_{s}^{(k)},t)\right) d \bm{\tilde{x}}d \bm{\tilde{v}}. \numberthis \label{E:limitingPIDEs}
\end{align*}
\end{remark}

 We conclude this section by presenting a few examples to illustrate the limiting PIDEs for basic reaction types.

\begin{example}
Consider a system with three species, $A$, $B$, and $C$ that can undergo the reversible reaction $A+B \rightleftarrows C$. Define the measures for $A$, $B$, and $C$ particles at time $t$ respectively as $\mu_{t}^{\zeta, 1}, \mu_{t}^{\zeta, 2},$ and $\mu_{t}^{\zeta, 3} \in M(\Rd \times \Rd)$.

Let $\mathcal{R}_{1}$ be the forward reaction $A+B \rightarrow C$, with $K_{1}^{\gamma}(x_1, x_2, v_1, v_2)$ the probability per unit time one $A$ particle at position $x_1$ with velocity $v_1$ and one $B$ particle at position $x_2$ with velocity $v_2$ bind. Once reaction $\mathcal{R}_{1}$ fires, we generate a new particle $C$ at position $y$ with velocity $v'$ following the placement density $m_{1}^{\eta}(y,v' | x_1, x_2, v_1, v_2)$. For $\mathcal{R}_{1}$, the substrates are particles of species $A$ and $B$, so $\alpha_{11}=\alpha_{12}=1$ and $\alpha_{13}=0$. The product is one particle $C$, so $\beta_{11}=\beta_{12}=0$ and $\beta_{13}=1 .$

Let $\mathcal{R}_{2}$ be the backward reaction $C \rightarrow A+B$, with $K_{2}^{\gamma}(y,v')$ the probability per time one $C$ particle at position $y$ with velocity $v'$ unbinds. Once reaction $\mathcal{R}_{2}$ fires, we generate a new particle $A$ at position $x_1$ with velocity $v_1$ and a new particle $B$ at position $x_2$ with velocity $v_2$ following the placement density $m_{2}^{\eta}(x_1, x_2, v_1, v_2|y,v')$. For $\mathcal{R}_{2}$, the substrate is a $C$ particle, so $\alpha_{21}=\alpha_{22}=0$ and $\alpha_{23}=1$. The products are $A$ and $B$ particles, so $\beta_{21}=\beta_{22}=1$ and $\beta_{23}=0 .$

If the limiting spatially distributed measures for species $A, B$ and $C$ have marginal densities $\bigl(\rho_{1}(x, v, t), \rho_{2}(x, v, t), \rho_{3}(x, v, t)\bigr)$ respectively, they must solve the following reaction-diffusion equations in a weak sense:
\begin{align}
\partial_{t} &\rho_{1}(x_1,v_1,t) =-v_1 \nabla_{x_1}\rho_{1}(x_1, v_1, t) + b_{1} \nabla_{v_1}\cdot \left[v_1 \rho_{1}(x_1, v_1, t) + b_{1} D_1 \nabla_{v_1}\rho_{1}(x_1, v_1, t)\right] \notag \\
&-\biggl(\int_{\Rd\times\Rd} K_{1}(x_1,x_2,v_1,v_2) \rho_{2}(x_2, v_2, t) d x_2 dv_2\biggr) \rho_{1}(x_1,v_1,t)\notag\\
&+\int_{\Rd\times\Rd} K_{2}(y,v')\biggl(\int_{\Rd\times\Rd} m_{2}(x_1,x_2,v_1,v_2| y,v')d x_2 dv_2\biggr) \rho_{3}(y,v', t) d y dv' \notag\\
\partial_{t} &\rho_{2}(x_2,v_2, t)=-v_2 \nabla_{x_2}\rho_{2}(x_2, v_2, t) + b_{2} \nabla_{v_2}\cdot \left[v_2 \rho_{2}(x_2, v_2, t) + b_{2} D_2 \nabla_{v_2}\rho_{2}(x_2, v_2, t)\right]\notag\\
&-\biggl(\int_{\Rd\times\Rd} K_{1}(x_1,x_2,v_1,v_2) \rho_{1}(x_1,v_1,t) d x_1 dv_1 \biggr)\rho_{2}(x_2,v_2, t)\notag\\
&+\int_{\Rd\times\Rd} K_{2}(y,v')\biggl(\int_{\Rd\times\Rd} m_{2}(x_1,x_2,v_1,v_2| y,v')d x_1 dv_1\biggr) \rho_{3}(y,v',t) d y dv'  \notag\\
\partial_{t} &\rho_{3}(y,v', t)=-v' \nabla_{y}\rho_{3}(y, v', t) + b_{3} \nabla_{v'}\cdot \left[v' \rho_{3}(y, v', t) + b_{3} D_3 \nabla_{v'}\rho_{j}(y, v', t)\right]\notag\\
&+\int_{\mathbb{R}^{4d}} K_{1}(x_1,x_2,v_1,v_2) m_{1}(y,v'| x_1,x_2,v_1,v_2)\rho_{1}(x_1,v_1,t) \rho_{2}(x_2,v_2,t) d x_1 d x_2 dv_1 dv_2\notag\\
& -K_{2}(y,v')\rho_{3}(y,v', t).\notag \\
\end{align}
\end{example}
\begin{example}
Consider a system with four species, $A, B, C$ and $D$ that can undergo the reversible reaction $A+B \rightleftarrows C +D.$
Define the measures for A, B, C, and D particles at time $t$ respectively as $\mu_{t}^{\zeta, 1}, \mu_{t}^{\zeta, 2}, \mu_{t}^{\zeta, 3}$ and $\mu_{t}^{\zeta, 4} \in M(\mathbb{R}^{d} \times \Rd)$.

Let $\mathcal{R}_{1}$ be the forward reaction $A+B \rightarrow C+D$, with $K_{1}^{\gamma}(x_1, x_2, v_1, v_2)$ the probability per time one $A$ particle at position $x_1$ with velocity $v_1$ and one $B$ particle at position $x_2$ with velocity $v_2$ bind. Once reaction $\mathcal{R}_{1}$ fires, we generate one new particle $C$ at position $y_1$ with velocity $v_1'$ and one new particle $D$ at position $y_2$ with velocity $v_2'$ following the placement density $m_{1}^{\eta}(y_1,y_2,v_1', v_2'| x_1, x_2, v_1, v_2)$. For $\mathcal{R}_{1}$, the substrates are particles of species $A$ and $B$, so $\alpha_{11}=\alpha_{12}=1$ and $\alpha_{13}=\alpha_{14}=0$. The products are one particle $C$ and one particle $D$, so $\beta_{11}=\beta_{12}=0$ and $\beta_{13}=\beta_{14}=1 .$

Let $\mathcal{R}_{2}$ be the backward reaction $C+D \rightarrow A+B$, with $K_{2}^{\gamma}(y_1, y_2, v_1', v_2')$ the probability per time one $C$ particle at position $y_1$ with velocity $v_1'$ and one $D$ particle at position $y_2$ with velocity $v_2'$ bind. Once reaction $\mathcal{R}_{2}$ fires, we generate a new particle $A$ at position $x_1$ with velocity $v_1$ and a new particle $B$ at position $x_2$ with velocity $v_2$ following the placement density $m_{2}^{\eta}(x_1, x_2, v_1, v_2 | y_1,y_2,v_1', v_2')$. For $\mathcal{R}_{2}$, the substrates are one $C$ particle and one $D$ particle, so $\alpha_{21}=\alpha_{22}=0$ and $\alpha_{23}=\alpha_{24}=1$. The products are one $A$ particle and one $B$ particle, so $\beta_{21}=\beta_{22}=1$ and $\beta_{23}=\beta_{24}=0 .$

If the limiting spatially distributed measures for species $A, B, C$ and $D$ have marginal densities $\bigl(\rho_{1}(x,v,t), \rho_{2}(x,v,t), \rho_{3}(x,v,t), \rho_{4}(x,v,t)\bigr)$ respectively, they must solve the following reaction-diffusion equations in a weak sense:
\begin{align*}
&\partial_{t} \rho_{1}(x_1,v_1,t)=-v_1 \nabla_{x_1}\rho_{1}(x_1, v_1, t) + b_{1} \nabla_{v_1}\cdot \left[v_1 \rho_{1}(x_1, v_1, t) + b_{1} D_1 \nabla_{v_1}\rho_{1}(x_1, v_1, t)\right]\\
\\
&\phantom{=} -\biggl(\int_{\Rd \times \Rd} K_{1}(x_1, x_2, v_1, v_2)\rho_{2}(x_2, v_2, t) dx_2dv_2\biggr) \rho_{1}(x_1,v_1,t)+\int_{\mathbb{R}^{4d}} K_{2}(y_1,y_2,v_1', v_2')\\
&\phantom{=} \times \biggl(\int_{\Rd \times \Rd}m_{2}(x_1, x_2, v_1, v_2 | y_1,y_2,v_1', v_2')d x_2 dv_2\biggr)\rho_{3}(y_1, v_1', t) \rho_{4}(y_2, v_2', t) d y_1 dv_1' d y_2 dv_2' \\
&\partial_{t} \rho_{2}(x_2,v_2, t) =-v_2 \nabla_{x_2}\rho_{2}(x_2, v_2, t) + b_{2} \nabla_{v_2}\cdot \left[v_2 \rho_{2}(x_2, v_2, t) + b_{2} D_2 \nabla_{v_2}\rho_{1}(x_2, v_2, t)\right]  \\
&\phantom{=}-\biggl(\int_{\Rd \times \Rd} K_{1}(x_1, x_2, v_1, v_2)\rho_{1}(x_1,v_1,t) d x_1 dv_1 \biggr)\rho_{2}(x_2, v_2, t)+\int_{\mathbb{R}^{4d}} K_{2}(y_1,y_2,v_1', v_2')\\
&\phantom{=} \times \biggl(\int_{\Rd \times \Rd} m_{2}(x_1, x_2, v_1, v_2 | y_1,y_2,v_1', v_2')d x_1 dv_1\biggr)\rho_{3}(y_1, v_1', t)\rho_{4}(y_2, v_2', t) d y_1 dv_1'dy_2 dv_2' \\
&\partial_{t} \rho_{3}(y_1,v_1', t) = -v_1' \nabla_{y_1}\rho_{1}(y_1, v_1', t) + b_{3} \nabla_{v_1'}\cdot \left[v_1' \rho_{3}(y_1, v_1', t) + b_{3} D_3 \nabla_{v_1'}\rho_{3}(y_1, v_1', t)\right]\\
&\phantom{=} -\biggl(\int_{\Rd \times \Rd}K_{2}(y_1,y_2,v_1', v_2') \rho_{4}(y_2, v_2', t)dy_2 dv_2'\biggr)\rho_{3}(y_1, v_1', t)+\int_{\mathbb{R}^{4d}} K_{1}(x_1, x_2, v_1, v_2) \\
&\phantom{=} \times
\biggl(\int_{\Rd \times \Rd}m_{1}(y_1,y_2,v_1', v_2'|x_1, x_2, v_1, v_2) dy_2 dv_2'\biggr)\rho_{1}(x_1,v_1,t) \rho_{2}(x_2, v_2, t) dx_1dx_2dv_1dv_2\\
&\partial_{t} \rho_{4}(y_2,v_2', t)=-v_2' \nabla_{y_2}\rho_{4}(y_2, v_2', t) + b_{4} \nabla_{v_2'}\cdot \left[v_2' \rho_{4}(y_2, v_2', t) + b_{4} D_4 \nabla_{v_2'}\rho_{4}(y_2, v_2', t)\right]\\
&\phantom{=} -\biggl(\int_{\Rd \times \Rd}K_{2}(y_1,y_2,v_1', v_2')\rho_{3}(y_1, v_1', t)dy_1dv_1'\biggr)\rho_{4}(y_2, v_2', t) +\int_{\mathbb{R}^{4d}} K_{1}(x_1, x_2, v_1, v_2) \\
&\phantom{=} \times \biggl(\int_{\Rd \times \Rd}m_{1}(y_1,y_2,v_1', v_2'|x_1, x_2, v_1, v_2) dy_1 dv_1'\biggr)\rho_{1}(x_1,v_1,t) \rho_{2}(x_2, v_2, t) dx_1dx_2dv_1dv_2.
\end{align*}
\end{example}
\begin{example}
Consider a system with two species, $A$ and $B$ that can undergo the reversible dimerization reaction $A+A \rightleftarrows B.$
Define the measures for A and B particles at time $t$ respectively as $\mu_{t}^{\zeta, 1}$ and $\mu_{t}^{\zeta, 2} \in M(\mathbb{R}^{d}\times \Rd)$.

Let $\mathcal{R}_{1}$ be the forward reaction $A+A \rightarrow B$, with $K_{1}^{\gamma}(x_1, x_2, v_1, v_2)$ the probability per time one $A$ particle at position $x_1$ with velocity $v_1$ and another $A$ particle at position $x_2$ with velocity $v_2$ bind. Once reaction $\mathcal{R}_{1}$ fires, we generate a new particle $B$ at position $y$ following the placement density $m_{1}^{\eta}(y,v' | x_1, x_2, v_1, v_2 )$. For $\mathcal{R}_{1}$, the substrates are particles of species $A$, so $\alpha_{11}=2$ and $\alpha_{12}=0$. The product is one particle $B$, so $\beta_{11}=0$ and $\beta_{12}=1.$

Let $\mathcal{R}_{2}$ be the backward reaction $B \rightarrow A+A$, with $K_{2}^{\gamma}(y,v')$ the probability per time one $B$ particle at position $y$ with velocity $v'$ unbinds. Once reaction $\mathcal{R}_{2}$ fires, we generate two new $A$ particles at $x_1$ with velocity $v_1$ and at $x_2$ with velocity $v_2$ following the placement density $m_{2}^{\eta}(x_1, x_2, v_1, v_2|y,v')$. For $\mathcal{R}_{2}$, the substrate is one $B$ particle, so $\alpha_{21}=0$ and $\alpha_{22}=1$. The products are two $A$ particles, so $\beta_{21}=2$ and $\beta_{22}=0$.

If the limiting spatially distributed measures for species $A$ and $B$ have marginal densities $\bigl(\rho_{1}(x, v, t), \rho_{2}(x, v, t)\bigr)$ respectively, they must solve the following reaction-diffusion equations in a weak sense:
\begin{align*}
&\partial_{t} \rho_{1}(x_1,v_1,t)=-v_1 \nabla_{x_1}\rho_{1}(x_1, v_1, t) + b_{1} \nabla_{v_1}\cdot \left[v_1 \rho_{1}(x_1, v_1, t) + b_{1} D_1 \nabla_{v_1}\rho_{1}(x_1, v_1, t)\right]\\
&\phantom{=} +2\int_{\mathbb{R}^{d}\times\Rd} K_{2}(y,v')\biggl(\int_{\mathbb{R}^{d}\times \Rd}m_{2}(x_1,x_2,v_1,v_2 | y,v')dx_2dv_2\biggr)\rho_{2}(y,v', t) d y dv' \\
&\phantom{=} -\biggl(\int_{\mathbb{R}^{d}\times \Rd}K_{1}(x_1,x_2,v_1,v_2) \rho_{1}(x_2,v_2, t)dx_2dv_2\biggr)\rho_{1}(x_1,v_1, t) \\
&\partial_{t} \rho_{2}(x_2,v_2, t) =-v_2 \nabla_{x_2}\rho_{2}(x_2, v_2, t) + b_{2} \nabla_{v_2}\cdot \left[v_2 \rho_{2}(x_2, v_2, t) + b_{2} D_2 \nabla_{v_2}\rho_{2}(x_2, v_2, t)\right] \\
&\phantom{=}-K_{2}(x_2,v_2)\rho_{2}(x_2,v_2, t) \\
&\phantom{=} +\frac{1}{2}\int_{\mathbb{R}^{4d} } K_{1}(x_1,x_2,v_1,v_2) m_{1}(y,v' | x_1,x_2,v_1,v_2) \rho_{1}(x_1,v_1, t) \rho_{1}(x_2,v_2, t) d x_1dx_2dv_1dv_2,
\end{align*}
where we have assumed $K_1$ is symmetric in the two $A$ particles' coordinates.
\end{example}

\section{Identification}\label{S:Identification}

We proceed to formally identify the limiting measures. Recall that $\bm{\mu}:=(\mu^{1}, \cdots, \mu^{J}) \in \otimes_{j=1}^{J} M_{F}(\mathbb{R}^{d}\times\mathbb{R}^{d})$ and $\mu=\sum_{j=1}^{J} \mu^{j} \delta_{S_{j}} \in M_{F}(\hat{P})$ with each $\mu^{j} \in M_{F}(\mathbb{R}^{d}\times\mathbb{R}^{d})$.
Let $\mathcal{S}$ be the collection of elements $\Phi$ in the space of bounded functionals, $B\left(\otimes_{j=1}^{J} M_{F}(\mathbb{R}^{d}\times\mathbb{R}^{d})\right)$, of the form
\begin{equation} \label{eq:IdentificationPhiDef}
\Phi(\bm{\mu})=\varphi\left(\left\langle f_{1}, \bm{\mu}\right\rangle,\left\langle f_{2}, \bm{\mu} \right\rangle, \ldots ,\left\langle f_{M}, \bm{\mu}\right\rangle\right)
\end{equation}
for some $M \in \mathbb{N}$, $\varphi \in C^{\infty}(\mathbb{R}^{JM})$, $\langle f_m, \bm{\mu}\rangle = (\langle f_{1,m},\mu^1 \rangle, \cdots, \langle f_{J,m},\mu^J \rangle)$ where each $f_{j,m} \subset C_{b}^{2}(\mathbb{R}^{d}\times\mathbb{R}^{d})$ for $1 \leq j \leq J$ and $1 \leq m \leq M$. Then for $\Phi \in \mathcal{S}$ of the above form, the generator $\mathcal{A}$ of (\ref{Eq:limitingMeasures}) and of the limiting martingale problem for $1 \leq j \leq J$, is defined as
\begin{align*}
&(\mathcal{A} \Phi)(\bm{\mu}) \coloneqq
\sum_{m=1}^{M} \sum_{j=1}^{J} \frac{\partial \varphi}{\partial x_{(m-1) \times J+j}}\left(\left\langle f_{1}, \bm{\mu}\right\rangle,\left\langle f_{2}, \bm{\mu}\right\rangle, \ldots,\left\langle f_{M}, \bm{\mu}\right\rangle\right)\biggl[\left\langle (\mathcal{L}_{j} f_{j, m})(x,v), \mu^{j}(dx, dv)\right\rangle\\
&\phantom{=}+\sum_{\ell=1}^{L}\int_{\txtv}\frac{1}{\bm{\alpha}^{(\ell)}!}K_\ell(\bm{x},\bm{v})\biggl(-\sum_{r=1}^{\alj}\fjm \left(x_r^{(j)},v_r^{(j)}\right)+\int_{\yv}\left(\sum_{r=1}^{\blj}\fjm\left(y_r^{(j)},{v'}_{r}^{(j)}\right)\right)\\
&\phantom{=}\times m_{\ell}\left(\bm{y},\bm{v'}|\bm{x},\bm{v}\right)d\bm{y}d\bm{v'}\biggr)\lambda^{(\ell)}[\mu](d\bm{x}, d\bm{v})\biggr].
\end{align*}
As explained in~\cite{IsaacsonSIMA2022}, to identify the limit it suffices to show convergence of the martingale problem for functions of the form~\eqref{eq:IdentificationPhiDef}, we we now show below.
\begin{lemma}(Weak Convergence).\label{Eq:weakconvergence} For any $\Phi \in \mathcal{S}$ and $0 \leq r_{1} \leq r_{2} \cdots \leq r_{W}=s<t<T$ and $\{\psi_{w}\}_{w=1}^{W} \subset$ $B\left(\otimes_{j=1}^{J} M_{F}(\mathbb{R}^{d}\times \mathbb{R}^{d})\right)$, we have that
\begin{equation}
\lim_{\zeta \rightarrow 0} \mathbb{E}\left[\left\{\Phi(\bm{\mu}_{t}^{\zeta})-\Phi(\bm{\mu}_{s}^{\zeta})-\int_{s}^{t}(\mathcal{A} \Phi)(\bm{\mu}_{r}^{\zeta}) d r\right\} \prod_{w=1}^{W} \psi_{w}\left(\bm{\mu}_{r_{w}}^{\zeta}\right)\right]=0.\label{il}
\end{equation}
\end{lemma}

\begin{proof}
For $j=1,
\ldots,J,$ we decompose \eqref{Eq:PathDescription} into three terms
\begin{equation}
    \left\langle f, \mu_{t}^{\zeta, j}\right\rangle = \fmi + \At + \Mt ,
\end{equation}
where
\begin{align*}
    \At &= \int_0^t \left\langle (\mathcal{L}_{j}f)(x,v), \msmzj(dx,dv)\right\rangle ds \\
    &\phantom{=} + \sum_{\ell=1}^{L}\int_0^t \int_{\txtv}\frac{1}{\bm{\alpha}^{(\ell)}!}K_{\ell}(\bm{x},\bm{v})\biggl(-\sum_{r=1}^{\alj}f\left(x_r^{(j)},v_r^{(j)}\right)+\int_{\yv}\left(\sum_{r=1}^{\blj}f\left(y_r^{(j)},{v'}_{r}^{(j)}\right)\right)\\
&\phantom{=}\times m_{\ell}^{\eta}\left(\bm{y},\bm{v'}|\bm{x},\bm{v}\right)d\bm{y}d\bm{v'}\biggr)\lmeasure ds,
\end{align*}
and
\begin{align*}
    \Mt &= \ginverse \sum_{i\geq1} \int_{0}^{t} \mathbbm{1}_{\{i \leq \all\}}b_{j}\sqrt{2 D_{j}} \pfpv\left(\HQ,\HV \right) d W_{s}^{i}\\
    &\phantom{=} +\sum_{\ell=1}^{L} \int_{0}^{t} \int_{\mathbb{I}^{(\ell)}} \int_{\yv} \int_{\mathbb{R}_{+}^2}\biggl(\left\langle f, \mu_{s^-}^{\zeta, j}-\frac{1}{\gamma} \sum_{r=1}^{\alpha_{\ell j}} \delta_{H^{i_{r}^{(j)}}_Q(\gamma \mu_{s^-}^{\zeta, j}),H^{i_{r}^{(j)}}_V(\gamma \mu_{s^-}^{\zeta, j})}+\ginverse\sum_{r=1}^{\beta_{\ell j}}\delta_{y_{r}^{(j)},{v'}_{r}^{(j)}} \right\rangle\\
&\hspace{2cm} -\left\langle f, \mu_{s^-}^{\zeta, j}\right\rangle\biggr)\times 1_{\{\bm{i} \in \Omega^{(\ell)}(\gamma\mu_{s^-}^{\zeta})\}} \times 1_{\{\theta_1 \leq K_{\ell}^{\gamma}\left(\pl\right)\}}\times\nonumber\\
&\hspace{4cm}\times 1_{\{\theta_2 \leq m_{\ell}^{\eta}\left(\bm{y},\bm{v'}|\pl\right)\}}d \tilde{N}_{\ell}(s,\bm{i},\bm{y},\bm{v'},\theta_1,\theta_2).
\end{align*}

The quadratic variation of $\Mt$ is
\begin{align*}
    &\left\langle M^{f,j}\right\rangle_t = \frac{1}{\gamma^2} \int_0^t \sumall \left(b_{j}\sqrt{2 D_{j}} \pfpv\left(\HQ,\HV \right)\right)^2 ds\\
    &\phantom{=} + \frac{1}{\gamma^2} \sum_{\ell=1}^{L} \int_{0}^{t} \int_{\yv} \sum_{\bm{i} \in \Omega^{(\ell)}(\gamma\mu_{s^-}^{\zeta})}\left(-\sum_{r=1}^{\alpha_{\ell j}} f\left(\HQ,\HV\right)+\sum_{r=1}^{\beta_{\ell j}}f\left(y_r^{(j)},{v'}_{r}^{(j)}\right) \right)^2 \\
    &\phantom{=}\quad \times K_{\ell}^{\gamma}\left(\pl\right)m_{\ell}^{\eta} \left(\bm{y},\bm{v'}|\pl \right)d\bm{y}d\bm{v'}ds\\
    &\leq \frac{1}{\gamma}\int_0^t \left\langle 2(b_{j})^2 D_j \left(\pfpv\right)^2,\msmzj\right\rangle ds+ \nonumber\\
    &\phantom{=}+\frac{1}{\gamma^2} \sum_{\ell=1}^{L} \int_{0}^{t} \int_{\yv} \sum_{\bm{i} \in \Omega^{(\ell)}(\gamma\mu_{s^-}^{\zeta})} \left(\alj+\blj\right)^2\norm{f}^2_{C_b^0(\Rd\times\Rd)}\times\\
    &\phantom{=}\quad\quad \times K_{\ell}^{\gamma}\left(\pl\right)m_{\ell}^{\eta} \left(\bm{y},\bm{v'}|\pl \right)d\bm{y}d\bm{v'}ds\\
     &\leq  \frac{2b_{j}^2D_jt\norm{f}^2_{C_b^1(\Rd \times \Rd)}}{\gamma}C_{\circ}+ 16C(K)\norm{f}^2_{C_b^0(\Rd\times\Rd)} \sum_{\ell=1}^{L} \int_{0}^{t} \frac{1}{\gamma^2}\sum_{\bm{i} \in \Omega^{(\ell)}(\gamma\mu_{s^-}^{\zeta})}\gamma^{1-|\alpha^{(\ell)}|}ds\\
    &\leq \frac{2b_{j}^2D_jt\norm{f}^2_{C_b^1(\Rd \times \Rd)}}{\gamma}C_{\circ} + \frac{C(K)tL\norm{f}^2_{C_b^0(\Rd\times\Rd)}}{\gamma}C_{\circ}^2. \numberthis \label{Eq:quadraticvariation}
\end{align*}

For $\left\langle M^{f,j}\right\rangle_t$ to be uniformly bounded and vanish in the large population limit, we use Assumptions \ref{A:molarBdd} and \ref{A:kernalBdd}, where $\klg = \gamma^{1-|\alpha^{(\ell)}|}K_\ell,$ and $ K_\ell$ and $\sum_{j=1}^{J}\langle 1, \mu_{t}^{\zeta, j}\rangle$ being uniformly bounded by generic constants $C(K)$ and $C_{\circ}$ respectively.

We now define $M^{f,j}_t = \mathcal{C}_{t}^{f, j}+\mathcal{D}_{t}^{f, j},$ where
\begin{equation*}
\cfjt = \ginverse \sum_{i\geq1} \int_{0}^{t} \mathbbm{1}_{\{i \leq \all\}}b_{j}\sqrt{2 D_{j}} \pfpv\left(\HQ,\HV \right) d W_{s}^{i} \numberthis \label{Eq:continuousmartingale}
\end{equation*}
is the continuous martingale part, and
\begin{equation}\label{dfj}
\begin{aligned}
\dfjt &= \sum_{\ell=1}^{L} \int_{0}^{t} \int_{\mathbb{I}^{(\ell)}} \int_{\yv} \int_{\mathbb{R}_{+}^2}\biggl(\left\langle f, \mu_{s^-}^{\zeta, j}-\frac{1}{\gamma} \sum_{r=1}^{\alpha_{\ell j}} \delta_{H^{i_{r}^{(j)}}_Q(\gamma \mu_{s^-}^{\zeta, j}),H^{i_{r}^{(j)}}_V(\gamma \mu_{s^-}^{\zeta, j})}+\ginverse\sum_{r=1}^{\beta_{\ell j}}\delta_{y_{r}^{(j)},{v'}_{r}^{(j)}} \right\rangle\\
&\hspace{1cm} -\left\langle f, \mu_{s^-}^{\zeta, j}\right\rangle\biggr)\times 1_{\{\bm{i} \in \Omega^{(\ell)}(\gamma\mu_{s^-}^{\zeta})\}} \times 1_{\{\theta_1 \leq K_{\ell}^{\gamma}\left(\pl\right)\}}\times \\
&\hspace{3cm}\times1_{\{\theta_2 \leq m_{\ell}^{\eta}\left(\bm{y},\bm{v'}|\pl\right)\}}d \tilde{N}_{\ell}(s,\bm{i},\bm{y},\bm{v'},\theta_1,\theta_2)
\end{aligned}
\end{equation}
is the martingale part with respect to the Poisson point processes. For simplicity of notation, we define the integrand of $\dfjt$ as
\begin{equation}\label{Eq:poissonprocess}
\begin{aligned}
g^{\ell, f, \mu^{\zeta, j}}&(s, \bm{i}, \bm{y},\bm{v'},\theta_1, \theta_2) = \biggl(\left\langle f, \mu_{s^-}^{\zeta, j}-\frac{1}{\gamma} \sum_{r=1}^{\alpha_{\ell j}} \delta_{H^{i_{r}^{(j)}}_Q(\gamma \mu_{s^-}^{\zeta, j}),H^{i_{r}^{(j)}}_V(\gamma \mu_{s^-}^{\zeta, j})}+\ginverse\sum_{r=1}^{\beta_{\ell j}}\delta_{y_{r}^{(j)},{v'}_{r}^{(j)}} \right\rangle\\
&\phantom{=} -\left\langle f, \mu_{s^-}^{\zeta, j}\right\rangle\biggr)\times 1_{\{\bm{i} \in \Omega^{(\ell)}(\gamma\mu_{s^-}^{\zeta})\}} \times 1_{\{\theta_1 \leq K_{\ell}^{\gamma}\left(\pl\right)\}}\times 1_{\{\theta_2 \leq m_{\ell}^{\eta}\left(\bm{y},\bm{v'}|\pl\right)\}}\\
&=\ginverse\left(-\sum_{r=1}^{\alj}f\left(\HQirj,\HVirj\right)+\sum_{r=1}^{\beta_{\ell j}}f(y_{r}^{(j)},{v'}_{r}^{(j)}) \right)\\
&\phantom{=}\times 1_{\{\bm{i} \in \Omega^{(\ell)}(\gamma\mu_{s^-}^{\zeta})\}} \times 1_{\{\theta_1 \leq K_{\ell}^{\gamma}\left(\pl\right)\}}\times 1_{\{\theta_2 \leq m_{\ell}^{\eta}\left(\bm{y},\bm{v'}|\pl\right)\}},
\end{aligned}
\end{equation}
which represents the jumps and is uniformly bounded by $\mathcal{O}(\frac{1}{\gamma})$. With some abuse of notation we shall write $\bm{g}^{\ell, f, \bm{\mu}^{\zeta}}$ for the vector $(g^{\ell, f, \mu^{\zeta, 1}}, \cdots, g^{\ell, f, \mu^{\zeta, J}})$. Then \eqref{dfj} becomes
$$
\mathcal{D}_{t}^{f, j}=\sum_{\ell=1}^{L} \int_{0}^{t} \int_{\mathbb{I}^{(\ell)}} \int_{\yv} \int_{\mathbb{R}_{+}^2} g^{\ell, f, \mu^{\zeta, j}}(s, \bm{i}, \bm{y},\bm{v'},\theta_1, \theta_2) d \tilde{N}_{\ell}(s,\bm{i},\bm{y},\bm{v'},\theta_1,\theta_2).
$$

Applying Itô's formula (see Theorem $5.1$ in \cite{IkedaWatanabe2014}) to $\Phi(\bm{\mu}_{t}^{\zeta})$ we obtain
\begin{align*}
&\Phi(\bm{\mu}_{t}^{\zeta})-\Phi(\bm{\mu}_{s}^{\zeta})-\int_{s}^{t}(\mathcal{A} \Phi)(\bm{\mu}_{r}^{\zeta}) d r\\
=&\int_{s}^{t} \sum_{m=1}^{M} \sum_{j=1}^{J} \frac{\partial \varphi}{\partial x_{(m-1) \times J+j}}\left(\langle f_{1}, \bm{\mu}_{r}^{\zeta}\rangle,\langle f_{2}, \bm{\mu}_{r}^{\zeta}\rangle \ldots\langle f_{M}, \bm{\mu}_{r}^{\zeta}\rangle\right)d \mathcal{C}_{r}^{f_{j, m}, j}
\\
&+\frac{1}{2} \int_{s}^{t} \sum_{m=1}^{M} \sum_{j=1}^{J} \frac{\partial^{2} \varphi}{\partial x_{(m-1) \times J+j}^{2}}\left(\langle f_{1}, \bm{\mu}_{r}^{\zeta}\rangle,\langle f_{2}, \bm{\mu}_{r}^{\zeta}\rangle \cdots\langle f_{M}, \bm{\mu}_{r}^{\zeta}\rangle\right) d\langle\mathcal{C}^{f_{j, m}, j}\rangle_{r} 
\\
& \phantom{=} +\sum_{\ell=1}^{L} \int_{s}^{t} \int_{\mathbb{I}^{(\ell)}} \int_{\yv} \int_{\mathbb{R}_{+}^2}\biggl(\varphi\bigl(\langle f_{1}, \bm{\mu}_{r}^{\bm{\zeta}}\rangle+\bm{g}^{\ell, f_{1}, \bm{\mu}^{\zeta}}(r, \bm{i}, \bm{y},\bm{v'},\theta_1, \theta_2), \ldots,\langle f_{M}, \bm{\mu}_{r}^{\zeta}\rangle\\
&\phantom{=}+\bm{g}^{\ell, f_{M}, \bm{\mu}^{\zeta}}(r, \bm{i}, \bm{y},\bm{v'},\theta_1, \theta_2)\bigr)-\varphi\left(\langle f_{1}, \bm{\mu}_{r}^{\zeta}\rangle,\langle f_{2}, \bm{\mu}_{r}^{\zeta}\rangle \ldots\langle f_{M}, \bm{\mu}_{r}^{\zeta}\rangle\right)\biggr) d \tilde{N}_{\ell}(r, \bm{i}, \bm{y},\bm{v'},\theta_1, \theta_2) 
\\
& +\sum_{\ell=1}^{L} \int_{s}^{t} \int_{\mathbb{I}^{(\ell)}} \int_{\yv} \int_{\mathbb{R}_{+}^2}\biggl(\varphi\bigl(\langle f_{1}, \bm{\mu}_{r}^{\bm{\zeta}}\rangle+\bm{g}^{\ell, f_{1}, \bm{\mu}^{\zeta}}(r, \bm{i}, \bm{y},\bm{v'},\theta_1, \theta_2), \ldots,\langle f_{M}, \bm{\mu}_{r}^{\zeta}\rangle\\
&\phantom{=}+\bm{g}^{\ell, f_{M}, \bm{\mu}^{\zeta}}(r, \bm{i}, \bm{y},\bm{v'},\theta_1, \theta_2)\bigr) -\varphi\left(\langle f_{1}, \bm{\mu}_{r}^{\zeta}\rangle,\langle f_{2}, \bm{\mu}_{r}^{\zeta}\rangle \ldots\langle f_{M}, \bm{\mu}_{r}^{\zeta}\rangle\right) \\
& \phantom{=}-\sum_{m=1}^{M} \sum_{j=1}^{J} g^{\ell, f_{j, m}, \mu^{\zeta,j}}(r, \bm{i}, \bm{y},\bm{v'},\theta_1, \theta_2)\\
&\phantom{=}\times \frac{\partial \varphi}{\partial x_{(m-1)\times J+j}}\left(\langle f_{1}, \bm{\mu}_{r}^{\zeta}\rangle,\ldots,\langle f_{M}, \bm{\mu}_{r}^{\zeta}\rangle\right)\biggr) d \bar{N}_{\ell}(r, \bm{i}, \bm{y},\bm{v'},\theta_1, \theta_2) \\
&+\sum_{m=1}^{M} \sum_{j=1}^{J} \sum_{\ell=1}^{L}\int_0^t\frac{\partial \varphi}{\partial x_{(m-1) \times J+j}}\left(\left\langle f_{1}, \bm{\mu}_{r}^{\bm{\zeta}}\right\rangle,\left\langle f_{2}, \bm{\mu}_{r}^{\bm{\zeta}}\right\rangle, \ldots,\left\langle f_{M}, \bm{\mu}_{r}^{\bm{\zeta}}\right\rangle\right)\\
&\phantom{=}\times \biggl[\int_{\txtv}\frac{1}{\bm{\alpha}^{(\ell)}!}K_\ell(\bm{x},\bm{v})\times \biggl(\int_{\yv}\left(\sum_{r=1}^{\blj}\fjm\left(y_r^{(j)},{v'}_{r}^{(j)}\right)\right)\\
&\phantom{=}\times \left( m_{\ell}^{\eta}\left(\bm{y},\bm{v'}|\bm{x},\bm{v}\right)-m_{\ell}\left(\bm{y},\bm{v'}|\bm{x},\bm{v}\right)\right)d\bm{y}d\bm{v'}\biggr)\lambda^{(\ell)}[\mu](d\bm{x}, d\bm{v})\biggr]\\
& =\sum_{\kappa=1}^{5} \Lambda_{\kappa}^{\zeta}(t), \numberthis\label{Eq:itoforidentification}
\end{align*}
where $\Lambda_{\kappa}^{\zeta}(t)$ represents the $\kappa$th additive term on the right-hand side. We now use the Skorokhod representation theorem (Theorem $1.8$ in \cite{EthierKurtz}) which, for the purposes of identifying the limit and proving \eqref{Eq:weakconvergence}, allows us to assume that the aforementioned claimed convergence of $\bm{\mu}_{t}^{\bm{\zeta}}$ holds with probability one in the topology of weak convergence of measures. The Skorokhod representation theorem involves the introduction of another probability space, but we ignore this distinction in the notation. To show \eqref{Eq:weakconvergence}, it is then sufficient to prove that the left-hand side of \eqref{Eq:itoforidentification} goes to zero in probability. We now proceed to prove convergence in probability to zero for $\Lambda_{\kappa}^{\zeta}(t)$ for $\kappa=1, \cdots, 5$.

First, note that $\Lambda_{1}^{\zeta}$ and $\Lambda_{3}^{\zeta}$ are both square integrable martingales. In fact, by \eqref{Eq:quadraticvariation} and \eqref{Eq:poissonprocess}, we obtain
\begin{equation*}
    \lim_{\zeta\rightarrow 0}\displaystyle \sup_{t\in [0,T]} \mathbb{E}|\Lambda_{1}^{\zeta}(t) + \Lambda_{3}^{\zeta}(t)|^2 = 0.
\end{equation*}
Similarly, we have by \eqref{Eq:continuousmartingale} that
\begin{equation*}
    \lim_{\zeta\rightarrow 0}\displaystyle \sup_{t\in [0,T]} \mathbb{E}|\Lambda_{2}^{\zeta}(t)| = 0,
\end{equation*}
and by \eqref{Eq:poissonprocess} that
\begin{equation*}
    \lim_{\zeta\rightarrow 0}\displaystyle \sup_{t\in [0,T]} \mathbb{E}|\Lambda_{4}^{\zeta}(t)| = 0.
\end{equation*}
Finally, $\Lambda_{5}^{\zeta}$ goes to zero in probability by Lemma \ref{L:placementDensityDifference}. Consequently, the left-hand side of \eqref{Eq:itoforidentification}
 goes to zero in probability, concluding the proof of the lemma.

\end{proof}
\section{Tightness} \label{S:tightness}
 Recall that $M_{F}(\Rd \times \Rd)$ denotes the space of finite measures endowed with the weak topology, and denote by $M_{F}'(\Rd \times \Rd)$ the space of finite measures endowed with the vague topology. In this section, we prove tightness of the measure-valued processes $\{\mu_{t}^{\zeta, j}\}_{t \in[0, T]}, j=1,2, \cdots, J$ on $\mathbb{D}_{M_{F}(\Rd \times \Rd)}[0, T]$, the space of c\`adl\`ag paths with values in $M_{F}(\Rd \times \Rd)$ endowed with Skorokhod topology. Towards this aim, we first show that the processes $\{\mu_{t}^{\zeta, j}\}_{t \in[0, T]}, j=1, \cdots, J$, are tight on $\mathbb{D}_{M_{F}'(\Rd \times \Rd)}[0, T]$ in Lemma \ref{L:vagueTightness} followed by tightness on $\mathbb{D}_{M_{F}(\Rd \times \Rd)}[0, T]$ in Theorem \ref{T:MainTightnessTheorem}.

\subsubsection{Tightness in $\mathbb{D}_{M'_{F}(\mathbb{R}^d \times \Rd)}[0,T]$.} It suffices to show the real-valued processes $\{\langle f, \mu_{t}^{\zeta, j}\rangle\}, j=1, \cdots, J$, for any test function $f \in C_0^{2}(\mathbb{R}^{d}\times \mathbb{R}^{d})$, which is dense in $C_0(\mathbb{R}^{d} \times \mathbb{R}^{d})$, are tight in $\mathbb{D}_{\mathbb{R}}[0, T]$, see \cite{Coppoletta1986}. In establishing the above statement, we use analogs of the Rebolledo Criterion \cite{Joffe1986} and the Aldous Condition \cite{Aldous1978}.

\begin{lemma}\label{L:AldousA_M}
    For any $T>0$ and $\delta>0$, there exists generic constants $C$ such that for any pair of stopping times $(\sigma, \tau)$ with $0 \leq \sigma \leq \tau \leq \sigma+\delta \leq T$, we have for $j=1, \cdots, J$,
$$
\mathbb{E}\left[\langle M^{f, j}\rangle_{\tau}-\langle M^{f, j}\rangle_{\sigma}\right] \leq C \delta,
$$
\end{lemma}
\begin{proof}[Proof of Lemma \ref{L:AldousA_M}]
The upper bound for $\mathbb{E}\left[\langle M^{f, j}\rangle_{\tau}-\langle M^{f, j}\rangle_{\sigma}\right]$ follows essentially the same arguments as in~\cite{IsaacsonSIMA2022}, which we give below for completeness:
\begin{align*}
\mathbb{E}&\left[\langle M^{f, j}\rangle_{\tau}-\langle M^{f, j}\rangle_{\sigma}\right]=\frac{1}{\gamma^{2}} \mathbb{E}\left[\int_{\sigma}^{\tau} \sum_{i=1}^{\gamma\langle 1, \mu_{s-}^{\zeta, j}\rangle}\left(b_{j}\sqrt{2 D_{j}} \frac{\partial f}{\partial V}\left(\HQ,\HV\right)\right)^{2} d s\right] \\
&\phantom{=}+\sum_{\ell=1}^{L} \mathbb{E}\biggl[\int_{\sigma}^{\tau} \int_{\yv} \frac{1}{\gamma^{2}}\sum_{\bm{i} \in \Omega^{(\ell)}(\gamma\mu_{s^-}^{\zeta})}\biggl(-\sum_{r=1}^{\alpha_{\ell j}} f\left(\HQirj,\HVirj\right)\\
&\phantom{=}+\sum_{r=1}^{\beta_{\ell j}}f\left(y_r^{(j)},{v'}_{r}^{(j)}\right) \biggr)^2 \times K_{\ell}^{\gamma}\left(\mathcal{P}^{(\ell)}(\gamma \mu_{s-}^{\zeta}, i)\right) m_{\ell}^{\eta}\left(\bm{y}, \bm{v'} | \mathcal{P}^{(\ell)}(\gamma \mu_{s-}^{\zeta}, \bm{i})\right)d \bm{y} d \bm{v'}d s\biggr] \\
&\leq \frac{C_{\circ} \mathbb{E}[\tau-\sigma]\|f\|_{C_0^{1}(\Rd \times \Rd)}^{2}}{\gamma}\left(C(D_{j})b_{j}^2+C(K) L\right) \\
&\leq C \delta.
\end{align*}
\end{proof}

To obtain an analogous upper bound on $\mathbb{E}\left[\left|A_{\tau}^{f, j}-A_{\sigma}^{f, j}\right|^{2}\right]$ will require us to first establish a-priori uniform bounds in $\zeta$ of appropriate moments with respect to both the position and velocity components of the empirical measure $\mu_{t}^{\bm{\zeta}, j}$ for $j=1,\cdots, J$. This is one of the technical difficulties seen in the underdamped case studied in this paper, which does not appear in the overdamped case of \cite{IsaacsonSIMA2022}.

Note that Assumption \ref{A:AssumptionMomentFiniteness} does not assume that $\sum_{j=1}^{J}\mathbb{E}\la |x|^{4} +|v|^{4},\mu_{t}^{\zeta, j}(dx,dv)\ra \newline<\infty$ is uniformly bounded with respect to $\zeta\in \mathbb{R}_{+} \times \mathbb{R}_{+}$ (apart from the requirement of a uniform bound at time $t=0$). Instead, Assumption \ref{A:AssumptionMomentFiniteness} only assumes that for each $\zeta\in\mathbb{R}_{+} \times \mathbb{R}_{+} $, the $4$th-moment is finite. We prove in Lemma \ref{L:MomentBounds} that these moments (assuming that they are well defined per Assumption \ref{A:AssumptionMomentFiniteness}) are indeed uniformly bounded in $\zeta\in\mathbb{R}_{+} \times \mathbb{R}_{+}$.

\begin{lemma}\label{L:AldousA_A}
    For any $T>0$ and $\delta>0$, there exists generic constants $C$ such that for any pair of stopping times $(\sigma, \tau)$ with $0 \leq \sigma \leq \tau \leq \sigma+\delta \leq T$, we have for $j=1, \cdots, J$,
$$
\mathbb{E}\left[\left|A_{\tau}^{f, j}-A_{\sigma}^{f, j}\right|^{2}\right] \leq C \delta.
$$
\end{lemma}
\begin{proof}[Proof of Lemma \ref{L:AldousA_A}]
We shall decompose $A_{t}^{f, j}$ into two terms. In particular we write that $A_{t}^{f, j}=A_{t}^{f, j,1}+A_{t}^{f, j,2}$ where
\begin{align*}
     A_{t}^{f, j,1}&= \int_0^t \left\langle (\mathcal{L}_{j}f)(x,v), \msmzj(dx,dv)\right\rangle ds \\
    A_{t}^{f, j,2}&=\sum_{\ell=1}^{L}\int_0^t \int_{\txtv}\frac{1}{\bm{\alpha}^{(\ell)}!}K_{\ell}(\bm{x},\bm{v})\biggl(-\sum_{r=1}^{\alj}f\left(x_r^{(j)},v_r^{(j)}\right)+\int_{\yv}\left(\sum_{r=1}^{\blj}f\left(y_r^{(j)},{v'}_{r}^{(j)}\right)\right)\\
&\phantom{=}\times m_{\ell}^{\eta}\left(\bm{y},\bm{v'}|\bm{x},\bm{v}\right)d\bm{y}d\bm{v'}\biggr)\lmeasure ds,
\end{align*}
By triangle inequality, we have that
\begin{align}
\mathbb{E}\left[\left|A_{\tau}^{f, j}-A_{\sigma}^{f, j}\right|^{2}\right]&\leq 2 \mathbb{E}\left[\left|A_{\tau}^{f, j,1}-A_{\sigma}^{f, j,1}\right|^{2}\right]+2\mathbb{E}\left[\left|A_{\tau}^{f, j,2}-A_{\sigma}^{f, j,2}\right|^{2}\right].
\end{align}

Let us first focus on the term $\mathbb{E}\left[\left|A_{\tau}^{f, j,1}-A_{\sigma}^{f, j,1}\right|^{2}\right]$. To properly bound this term, we use the bound $\la 1, \mu_t^{{\zeta}, j} \ra \leq C_{\circ}$ from Assumption \ref{A:molarBdd} and Lemma \ref{L:MomentBounds} with $p=1$. Then, we have
\begin{align}
\mathbb{E}\left[\left|A_{\tau}^{f, j,1}-A_{\sigma}^{f, j,1}\right|^{2}\right]&=\mathbb{E}\left|\int_{\sigma}^{\tau} \left\langle (\mathcal{L}_{j}f)(x,v), \msmzj(dx,dv)\right\rangle ds\right|^{2}\nonumber\\
&\leq\delta \mathbb{E}\int_{\sigma}^{\tau} \left|\left\langle (\mathcal{L}_{j}f)(x,v), \msmzj(dx,dv)\right\rangle \right|^{2} ds\nonumber\\
&=\delta \mathbb{E}\int_{\sigma}^{\tau} \left|\left\langle v \nabla_{x}f(x,v) - b_{j} v \nabla_{v}f(x,v) +  b_{j}^2 D_j \Delta_{v} f(x,v), \msmzj(dx,dv)\right\rangle \right|^{2} ds\nonumber\\
&\leq C (D_j, b_j, C_{\circ}) \delta  \|f\|^{2}_{C^{2}_0(\Rd \times \Rd)} \mathbb{E}\int_{\sigma}^{\tau} \left\langle |v|^{2} +1, \msmzj(dx,dv)\right\rangle  ds\nonumber\\
&\leq C (D_j, b_j, C_{\circ}) \delta  \|f\|^{2}_{C^{2}_0(\Rd \times \Rd)}\left[ \mathbb{E}\int_{0}^{T} \left\langle |v|^{2}, \msmzj(dx,dv)\right\rangle  ds+\delta C_{\circ}\right]\nonumber\\
&\leq C(D_j, b_{j},C_{\circ}) \delta  \|f\|^{2}_{C^{2}_0(\Rd \times \Rd)}\left[ \int_{0}^{T} \theta^{\zeta,j,2}  ds+\delta\right],
\end{align}
which by Assumption \ref{A:AssumptionMomentFiniteness} and Lemma \ref{L:MomentBounds} with $p=1$ leads to the bound
\begin{align}
\mathbb{E}\left[\left|A_{\tau}^{f, j,1}-A_{\sigma}^{f, j,1}\right|^{2}\right]&\leq C \delta  \|f\|^{2}_{C^{2}_0(\Rd \times \Rd)}
\end{align}


We prove a similar bound for $\mathbb{E}\left[\left|A_{\tau}^{f, j,2}-A_{\sigma}^{f, j,2}\right|^{2}\right]$. Indeed, we obtain
\begin{align*}
    &\mathbb{E}\left[\left|A_{\tau}^{f, j,2}-A_{\sigma}^{f, j,2}\right|^{2}\right] =  \mathbb{E}\biggl[\biggl|\sum_{\ell=1}^{L}\int_\sigma^\tau \int_{\txtv}\frac{1}{\bm{\alpha}^{(\ell)}!}K_{\ell}(\bm{x},\bm{v})\biggl(-\sum_{r=1}^{\alj}f\left(x_r^{(j)},v_r^{(j)}\right)\\
&\phantom{=}+\int_{\yv}\left(\sum_{r=1}^{\blj}f\left(y_r^{(j)},{v'}_{r}^{(j)}\right)\right)m_{\ell}^{\eta}\left(\bm{y},\bm{v'}|\bm{x},\bm{v}\right)d\bm{y}d\bm{v'}\biggr)\lmeasure ds\biggr|^2\biggr] \\
&\leq C(L) \sum_{\ell = 1}^L\mathbb{E}\biggl[\left|\tau-\sigma \right|\int_\sigma^\tau \int_{\txtv}\biggl|\frac{1}{\bm{\alpha}^{(\ell)}!}K_{\ell}(\bm{x},\bm{v})\biggl(-\sum_{r=1}^{\alj}f\left(x_r^{(j)},v_r^{(j)}\right)\\
&\phantom{=}+\int_{\yv}\left(\sum_{r=1}^{\blj}f\left(y_r^{(j)},{v'}_{r}^{(j)}\right)\right)m_{\ell}^{\eta}\left(\bm{y},\bm{v'}|\bm{x},\bm{v}\right)d\bm{y}d\bm{v'}\biggr)\biggr|^2\lmeasure\\
&\phantom{=}\times \int_{\txtv} 1^2 \lmeasure ds\biggr] \\
&\leq C(L,K, C_{\circ}) \norm{f}^{2}_{C^{0}_0(\Rd \times \Rd)} \sum_{\ell = 1}^L \mathbb{E}\biggl[\left|\tau-\sigma \right|\int_\sigma^\tau \int_{\txtv}\lmeasure ds\biggr]\\
&\leq C(L,K, C_{\circ}) \norm{f}^{2}_{C^{0}_0(\Rd \times \Rd)} \mathbb{E}\left[\left|\tau-\sigma \right|^2\right]\\
&\leq C(L,K,C_{\circ})\delta \norm{f}^{2}_{C^{0}_0(\Rd \times \Rd)}.
\end{align*}
Here we have absorbed terms involving $L$ and $C_{\circ}$ into the generic constant $C(L,K,C_{\circ})$ from line to line.
\end{proof}

\begin{lemma}\label{L:Aldous}  For any $T>0, \epsilon_{1}>0, \epsilon_{2}>0, j=1,2, \cdots, J$, there exists $\delta>0$ and $n_{0} \in \mathbb{N}$ such that for any sequence $(\sigma_{n}, \tau_{n})_{n \in \mathbb{N}}$ of pairs of stopping times with $\sigma_{n} \leq \tau_{n} \leq T$,
$$
\sup _{n \geq n_{0}} \mathbb{P}\{|\langle f, \mu_{\sigma_{n}}^{n, j}\rangle-\langle f, \mu_{\tau_{n}}^{n, j}\rangle| \geq \epsilon_{2}, \tau_{n} \leq \sigma_{n}+\delta\} \leq \epsilon_{1}.
$$
\end{lemma}
\begin{proof}
    For $\tau_{n} \leq \sigma_{n}+\delta$, we have
$$
\begin{aligned}
&\sup _{n \geq n_{0}} \mathbb{P}\{|M_{\sigma_{n}}^{f, j}-M_{\tau_{n}}^{f, j}|+|A_{\sigma_{n}}^{f, j}-A_{\tau_{n}}^{f, j}| \geq \epsilon_{2}\} \\
&\leq \sup _{n \geq n_{0}} 2 \frac{1}{\epsilon_{2}^{2}} \mathbb{E}[\langle M^{f, j}\rangle_{\tau_{n}}-\langle M^{f, j}\rangle_{\sigma_{n}}+|A_{\tau_{n}}^{f, j}-A_{\sigma_{n}}^{f, j}|^{2}] \quad \text { (by Markov inequality) } \\
&\leq \frac{2}{\epsilon_{2}^{2}}C \delta \leq \epsilon_{1} \quad \text { (by Lemmas } \ref{L:AldousA_M} \text { and } \ref{L:AldousA_A} \text{ with } \delta \text { sufficiently small) }.
\end{aligned}
$$
\end{proof}

\begin{lemma} \label{L:vagueTightness}
For each $j=1, \cdots, J$, the sequence of real-valued processes $\{\langle f, \mu_{t}^{\zeta, j}\rangle\}_{\zeta \in(0,1)^{2}}$ is tight in $\mathbb{D}_{\mathbb{R}}[0, T]$.
\end{lemma}

\begin{proof}
    By Assumption \ref{A:molarBdd}, $\langle 1, \mu_{t}^{\zeta, j}\rangle, j=1, \cdots, J$ are uniformly bounded. Since $f \in C_{0}^{2}(\Rd \times \Rd)$, we then get
$$
\sup _{\|\zeta\| \leq 1} \mathbb{E}[\sup _{t \in[0, T]} |\langle f, \mu_{t}^{\zeta, j}\rangle | ]<\infty.
$$
Combined with the Aldous condition from Lemma
\ref{L:Aldous}, we obtain that the sequence of real-valued processes $\{\langle f, \mu_t^{\zeta,j}\rangle\}_{\zeta \in (0,1)^2},$ is tight in $\mathbb{D}_\mathbb{R}[0,T]$ by the Rebolledo Criterion \cite{Joffe1986} for each $j = 1,2,\dots,J.$
\end{proof}

\subsubsection{Tightness in $\mathbb{D}_{M_{F}(\mathbb{R}^d \times \Rd)}[0,T]$.} Using the next Lemma \ref{L:massOutofCompact}, we manage to control the mass of measures outside of compact sets so that we can go from tightness in $\mathbb{D}_{M_{F}'(\Rd \times \Rd)}[0, T]$ to tightness in $\mathbb{D}_{M_{F}(\Rd \times \Rd)}[0, T]$.


\begin{lemma} \label{L:massOutofCompact}
There exists a sequence of $C_{b}^{2}(\Rd \times \Rd)$ functions $\{f_{m}(x,v)\}_{m \geq 0}$ with $f_{0} \equiv 1$ such that
$$
\begin{aligned}
f_{m}(x,v) &=0 \text { when }\|x\| + \|v\| \leq m-1 \\
f_{m}(x,v) &=1 \text { when }\|x\|+ \|v\|>m \\
0 \leq f_{m}(x,v) & \leq 1 \text { when } m-1<\|x\| + \|v\|\leq m .
\end{aligned}
$$
Additionally, $\displaystyle\sup _{m \geq 0}\|f_{m}(x,v)\|_{C_{b}^{2}(\Rd \times \Rd)}:=\sup _{m \geq 0} \mkern22mu \smashoperator{\sup_{\substack{ x \in \mathbb{R}^{d},v \in \mathbb{R}^{d},\\ |\alpha| \leq 2\strut}}}|D_{\alpha} f_{m}(x,v)|<\infty$. For such a sequence of functions $\{f_{m}(x,v)\}_{m \geq 0}$, we have
$$
\lim _{m \rightarrow \infty} \limsup _{\zeta \rightarrow 0} \mathbb{E}\left[\sup _{t \in[0, T]}\langle f_{m}, \mu_{t}^{\zeta, j}\rangle \right]=0
$$
for all $j=1,2, \cdots, J$.
\end{lemma}
\begin{proof} Even though the structure of the proof of Lemma \ref{L:massOutofCompact} is similar to Lemma 7.5 in \cite{IsaacsonSIMA2022}, there are some delicate differences due to the inclusion of the velocity components. We present the proof below emphasizing the differences.

As in \cite{Jourdain2012}, consider $\psi(x,v) = 6(x+v)^5-15(x+v)^4+10(x+v)^3 \in C^2([0,1]\times [0,1])$, but now we set $f_m(x,v) = \psi\bigl(0\vee(\norm{x}+\norm{v}-(m-1))\wedge1\bigr),$ for $x, v\in \mathbb{R}^d, m \geq1.$ Recall that by Assumption \ref{A:AssumptionRho}, for given $\epsilon>0$ there exists a radius $R \in \mathbb{N}$ large enough such that $\int_{r>R} \rho(r)r^{d-1}  d r \le \epsilon$ and $\int_{|w|+|u|>R} \rho(|w|,|u|)dwdu < \epsilon$. Furthermore, by Assumption \ref{A:molarBdd}, we assume that the total molar concentration $\sum_{j=1}^{J}\langle 1, \mu_{t}^{\zeta, j}\rangle$ is uniformly in time bounded by some constant $C_{\circ}$. Therefore, for $m$ sufficiently large, we obtain
\begin{align*}
&\langle f_{m}, \mu_{t}^{\bm{\zeta}, j}\rangle=\langle f_{m}, \mu_{0}^{\bm{\zeta}, j}\rangle+M_{t}^{f_{m}, j}\nonumber\\
&\phantom{=}+\int_0^t \left\langle v \nabla_{x}f_m(x,v) - b_{j} v \nabla_{v}f_m(x,v) + b_{j}^2 D_j \Delta_{v} f_m(x,v), \msmzj(dx,dv)\right\rangle ds\\
&\phantom{=} +\sum_{\ell=1}^{L}\int_0^t \int_{\txtv}\frac{1}{\bm{\alpha}^{(\ell)}!}K_\ell(\bm{x},\bm{v})\biggl(-\sum_{r=1}^{\alj}f_m\left(x_r^{(j)},v_r^{(j)}\right)+\int_{\yv}\left(\sum_{r=1}^{\blj}f_m\left(y_r^{(j)},{v'}_{r}^{(j)}\right)\right)\nonumber\\
&\phantom{=}\times m_{\ell}^{\eta}\left(\bm{y},\bm{v'}|\bm{x},\bm{v}\right)d\bm{y}d\bm{v'}\biggr)\lmeasure ds\\
&\leq\langle f_{m}, \mu_{0}^{\bm{\zeta}, j}\rangle+M_{t}^{f_{m}, j}+C(b_{j})\norm{f_m}_{C_b^2(\Rd \times \Rd)}\int_0^t \left\langle (1 + |v|) \mathbbm{1}_{\{m-1 \leq \norm{x} + \norm{v} \leq m\}} , \msmzj(dx,dv)\right\rangle ds\\
&\phantom{=} +\sum_{\ell=1}^{L}\int_0^t \int_{\txtv}\frac{1}{\bm{\alpha}^{(\ell)}!}K_\ell(\bm{x},\bm{v})\biggl(\int_{\yv}\left(\sum_{r=1}^{\blj}f_m\left(y_r^{(j)},{v'}_{r}^{(j)}\right)\right)\nonumber\\
&\phantom{=}\times m_{\ell}^{\eta}\left(\bm{y},\bm{v'}|\bm{x},\bm{v}\right)d\bm{y}d\bm{v'}\biggr)\lmeasure ds \numberthis\label{Eq:fmmu}
\end{align*}

Consider $g_t(m,\zeta) \coloneqq \mathbb{E}\left[\int_0^t \left\langle |v|^2 \mathbbm{1}_{\{m-1 \leq \norm{x} + \norm{v} \leq m\}} , \msmzj(dx,dv)\right\rangle ds\right]$. Then we obtain the following bound
\begin{align*}
    &\mathbb{E}\left[\supint\int_0^t \left\langle |v| \mathbbm{1}_{\{m-1 \leq \norm{x} + \norm{v} \leq m\}} , \msmzj(dx,dv)\right\rangle ds\right]\\
    \leq&\left(\mathbb{E}\left[\supint\int_0^t \left\langle |v|^2 \mathbbm{1}_{\{m-1 \leq \norm{x} + \norm{v} \leq m\}} , \msmzj(dx,dv)\right\rangle ds\right]\right)^{\frac{1}{2}}\\
    &\phantom{=}\times\left(\mathbb{E}\left[\supint\int_0^t \left\langle \mathbbm{1}_{\{m-1 \leq \norm{x} + \norm{v} \leq m\}} , \msmzj(dx,dv)\right\rangle ds\right]\right)^{\frac{1}{2}}\\
    \leq&\left(\mathbb{E}\left[\int_0^T \left\langle |v|^2 \mathbbm{1}_{\{m-1 \leq \norm{x} + \norm{v} \leq m\}} , \msmzj(dx,dv)\right\rangle ds\right]\right)^{\frac{1}{2}}\nonumber\\
    &\phantom{=}\times\left(\mathbb{E}\left[\supint\int_0^t \left\langle f_{m-1}(x,v), \msmzj(dx,dv)\right\rangle ds\right]\right)^{\frac{1}{2}}\\
        \leq&g_T^{\frac{1}{2}}(m,\zeta)\left(L + \mathbb{E}\left[\supint\int_0^T \left\langle f_{m-1}(x,v), \msmzj(dx,dv)\right\rangle ds\right]\right).
\end{align*}

Next, by Lemma \ref{L:boundoffm} we bound the expectation of the supremum over time of Eq \eqref{Eq:fmmu}
\begin{align*}
    \mathbb{E}&\left[\displaystyle \sup_{t\in [0,T]}\langle f_{m}, \mu_{t}^{\bm{\zeta}, j}\rangle\right]\leq  \mathbb{E}\left[\langle f_{m}, \mu_{0}^{\bm{\zeta}, j}\rangle \right] + \mathbb{E}\left[\sup_{t\in [0,T]}M_{t}^{f_{m}, j} \right]\\
&\phantom{=}+C(b_{j})\norm{f_m}_{C_b^2(\Rd \times \Rd)}\left(L g_T^{\frac{1}{2}}(m,\zeta) + \paren{1 + g_T^{\frac{1}{2}}(m,\zeta)}\mathbb{E}\left[\supint\int_0^T \left\langle f_{m-1}(x,v), \msmzj(dx,dv)\right\rangle ds\right]\right)\\
&\phantom{=} +\mathbb{E}\biggl[\supint\sum_{\ell=1}^{L}\int_0^t \int_{\txtv}\frac{1}{\bm{\alpha}^{(\ell)}!}K_\ell(\bm{x},\bm{v})\biggl(\int_{\yv}\left(\sum_{r=1}^{\blj}f_m\left(y_r^{(j)},{v'}_{r}^{(j)}\right)\right)\nonumber\\
&\phantom{=}\times m_{\ell}^{\eta}\left(\bm{y},\bm{v'}|\bm{x},\bm{v}\right)d\bm{y}d\bm{v'}\biggr)\lmeasure dt\biggr]\\
&\leq  \mathbb{E}\left[\langle f_{m}, \mu_{0}^{\bm{\zeta}, j}\rangle \right] + \mathbb{E}\left[\sup_{t\in [0,T]}M_{t}^{f_{m}, j} \right] + C_2 T \left(\epsilon + C\eta + g_T^{\frac{1}{2}}(m,\zeta)\right) \\
&\phantom{=}+ C_1 \int_0^T \displaystyle \sup_{1\leq i\leq J}\mathbb{E}\left[\sup_{s\in[0,t]} \left\langle f_{\psi(m)}(x,v), \msmzj(dx,dv)\right\rangle +\sup_{s\in[0,t]} \left\langle f_{\phi(m)}(x,v), \msmzj(dx,dv)\right\rangle \right] ds, \numberthis \label{E:EsupBoundfm}
\end{align*}
where for, $R$ from Assumption \ref{A:AssumptionRho}, and fixed constants $D> \left\{\frac{2m_{1}}{m_{1}+m_{2}}, \frac{2m_{2}}{m_{1}+m_{2}}\right\}$ and $\sigma^{*}\in [1,2)$ both defined in Lemma \ref{L:boundoffm}, we have set $\phi(m)=\floor{\frac{m-1}{1+\sigma^{*}}}$, $\psi(m)=m-1-(D+2)R$. The other constants are defined as follows $C=2 L C(K)(C_{\circ} \vee 1)$, $C_{1}=\sup_{m,\zeta} 2\left(C \vee\left\|f_{m}\right\|_{C_{b}^{2}\left(\mathbb{R}^{d}\times \mathbb{R}^{d}\right)}(1 + g_T^{\frac{1}{2}}(m,\zeta))\right)$, and  $C_{2}= 2 L C(K) C_{\circ}^2\sup_{m} \left\|f_{m}\right\|_{C_{b}^{2}\left(\mathbb{R}^{d}\times \mathbb{R}^{d}\right)}$.

Since the constants $\sigma^{*},D,R<\infty$ are fixed, for $m$ large enough we shall have that $\phi(m)\leq \psi(m)+1$. This then implies that for $m$ large enough, we shall have that $f_{\psi(m)}(x,v)\leq f_{\phi(m)}(x,v)$. Therefore,  defining the quantity $Y_{T}^{m, \bm{\zeta}}=\sup_{1\leq j\leq J}\mathbb{E}\left[\sup_{t\in[0,T]}\left(\left\langle f_{m}, \mu_{t}^{\zeta, j}\right\rangle\right)\right]$, we get  the inequality
\begin{align}
&Y_{T}^{m, \bm{\zeta}}  \leq Y_{0}^{m, \bm{\zeta}}+\displaystyle\sup _{1 \leq j \leq J} \mathbb{E}\left[\sup _{t \in[0, T]}\left|M_{t}^{f_{m}, j}\right|\right]+C_{1} \int_{0}^{T} Y_{t}^{\phi(m), \bm{\zeta}} d t+C_2 T \left(\epsilon + C\eta + g_T^{\frac{1}{2}}(m,\zeta)\right) \nonumber \\
&\leq Y_{0}^{m, \bm{\zeta}}+C_{3} \frac{1}{\sqrt{\gamma}}+C_{1} \int_{0}^{T} Y_{t}^{\phi(m), \bm{\zeta}} d t+C_2 T \left(\epsilon + C\eta + g_T^{\frac{1}{2}}(m,\zeta)\right)
\end{align}

We note that by Lemma \ref{L:interchangelimit}, $g_T^{\frac{1}{2}}(m,\zeta)$ can be made arbitrarily small for $m$ large enough. We set for notational simplicity $C_{4}(\epsilon)=C_{3} \frac{1}{\sqrt{\gamma}}+C_2 T \left(\epsilon + C\eta + g_T^{\frac{1}{2}}(m,\zeta)\right)$, suppressing the dependence of $C_{4}$ on $m,\zeta$ and we note that $\lim_{\epsilon,\zeta\rightarrow 0, m\rightarrow\infty}C_{4}(\epsilon)=0$.  Iterating the latter expression in $m$ one time gives
\begin{align}
&Y_{T}^{m, \bm{\zeta}}  \leq Y_{0}^{m, \bm{\zeta}}+C_{4}(\epsilon)+ C_{1}\int_{0}^{T} \left(Y_{0}^{\phi(m), \bm{\zeta}} +C_{4}(\epsilon)\right)d t+ C_{1}^{2}\int_{0}^{T}\int_{0}^{t} Y_{s}^{\phi(\phi(m)), \bm{\zeta}}dsdt\nonumber\\
&\leq Y_{0}^{m, \bm{\zeta}}+C_{4}(\epsilon)+ C_{1}T (Y_{0}^{m, \bm{\zeta}}+C_{4}(\epsilon))+\frac{(C_{1}T)^{2}}{2!}\sup_{s\leq T}Y_{s}^{\phi(\phi(m)), \bm{\zeta}}
\end{align}

Denote now by $\phi^{k}(m)$ the $k-$th convolution of $\phi(\phi(\cdots\phi(m))\cdots))$. Iterating the latter $n$ times in total (with $n$ to be determined) will yield
\begin{align}
&Y_{T}^{m, \bm{\zeta}} \leq \sum_{k=0}^{n-1}\frac{(C_{1}T)^{k}}{k!}(Y_{0}^{\phi^{k}(m), \bm{\zeta}}+C_{4}(\epsilon))+\frac{(C_{1}T)^{n}}{n!}\sup_{s\leq T}Y_{s}^{\phi^{k}(m), \bm{\zeta}}
\end{align}

Next we need to determine the right value for $n$. We want to achieve two things at the same time
\begin{itemize}
\item{$Y_{0}^{\phi^{n}(m), \bm{\zeta}}\rightarrow 0$ as $\phi^{n}(m)\rightarrow\infty$, and}
\item{$\frac{(C_{1}T)^{n}}{n!}\rightarrow 0$ as $n\rightarrow\infty$.}
\end{itemize}

Note that $\phi^{n}(m)\approx \frac{m}{(1+\sigma^{*})^{n}}$. So, if we choose $n=n(m)=\floor{\frac{\log m}{2\log(1+\sigma^{*})}}$ we obtain that $\phi^{n(m)}(m)\approx m^{1/2}$ and that $\frac{(C_{1}T)^{n(m)}}{n(m)!}\approx e^{- cn(m)\log n(m)}\rightarrow 0$. Thus,  with this choice for $n(m)$, both objectives are materialized.

Hence setting  $n(m)=\floor{\frac{\log m}{2\log(1+\sigma^{*})}}$ and using the fact that $\tilde{Y}_{T}^{0}=\displaystyle\limsup _{\zeta \rightarrow 0} \displaystyle\sup _{1 \leq j \leq J} \mathbb{E}\left[\sup _{t \in[0, T]}\left\langle 1, \mu_{t}^{\zeta, j}\right\rangle\right]$ is bounded by $C_\circ$, we have
\begin{align}
&Y_{T}^{m, \bm{\zeta}} \leq \sum_{k=0}^{n(m)-1}\frac{(C_{1}T)^{k}}{k!} Y_{0}^{\phi^{k}(m), \bm{\zeta}}+C_{4}(\epsilon) e^{C_{1}T}+o(1),
\end{align}
and we obtain
\begin{align}
&\lim_{m\rightarrow\infty}\limsup_{\zeta\rightarrow 0} Y_{T}^{m, \bm{\zeta}} \leq C_2 T   e^{C_{1}T} \epsilon.
\end{align}

Note, $C_{1}$ and $C_{2}$ do not depend on $\zeta, R$ or $\epsilon$. This implies
$$
\lim _{m \rightarrow \infty} \limsup _{\zeta \rightarrow 0} \mathbb{E}\left[\sup _{t \in[0, T]}\left\langle f_{m}, \mu_{t}^{\zeta, j}\right\rangle\right]
$$
is less than an arbitrary small number $C_{2} T e^{C_{1} T} \epsilon$, i.e. the limit is zero, for all $1 \leq j \leq J$, concluding the proof of the lemma.

\end{proof}

Denote by $\vxi_t = \left(\xi_{t}^{1}, \xi_{t}^{2}, \ldots, \xi_{t}^{J}\right)$ the weak limit of a subsequence of $(\mu_t^{\zeta,1},\cdots, \mu_t^{\zeta,J})$ in $\mathbb{D}_{M_{F}'(\Rd \times \Rd)}[0, T]$   (which exists by Lemma \ref{L:vagueTightness}). We then have the following continuity result.
\begin{lemma}\label{L:ContinuityLimitingMeasure}
The process $\left\{\left(\xi_{t}^{1}, \xi_{t}^{2}, \ldots, \xi_{t}^{J}\right)\right\}_{t\in[0,T]}$ is a continuous process from $[0,T]$ to both $M_{F}'(\Rd \times \Rd)$ and $M_{F}(\Rd \times \Rd)$.
\end{lemma}
\begin{proof}
Given Lemma \ref{L:vagueTightness} and Lemma \ref{L:massOutofCompact}, the proof of this result follows directly from that of Lemma 7.6 in \cite{IsaacsonSIMA2022}. Therefore the details are omitted.
\end{proof}

Now we are ready to present the main tightness result.
\begin{theorem}\label{T:MainTightnessTheorem}
For any $j\in\{1,\cdots,J\}$, the family of measure-valued stochastic processes $\left\{\mu_t^{\zeta,j}\right\}_{t\in[0,T]}$ is tight in $\mathbb{D}_{M_{F}(\Rd \times \Rd)}[0, T]$.
\end{theorem}
\begin{proof}
Given Lemmas \ref{L:vagueTightness}, \ref{L:massOutofCompact} and \ref{L:ContinuityLimitingMeasure}, the proof of this result follows by the arguments in \cite{meleard1993convergences}, as in Theorem 7.7 of \cite{IsaacsonSIMA2022}. Details are omitted.
\end{proof}

\section{Uniqueness}\label{S:uniqueness}
We now show that the solution to~\eqref{Eq:limitingMeasures} is unique in $C_{M_{F}(\Rd \times \Rd)}([0,T])$. $C$ will subsequently denote a generic constant. Suppose, by contradiction, that we have two different solutions to~\eqref{Eq:limitingMeasures}, $\{\vxi_t \coloneqq (\xi_{t}^{1}, \xi_{t}^{2}, \cdots, \xi_{t}^{J})\}_{t \in[0, T]}$ and $\{\vxibar_t \coloneqq (\bar{\xi}_{t}^{1}, \bar{\xi}_{t}^{2}, \cdots, \bar{\xi}_{t}^{J})\}_{t \in[0, T]}$, with the same initial condition $\vxi_{0}=\vxibar_{0}$. Parallel to Eq (\ref{Eq:limitingMeasures}), for a test function of the form of $\psi_{t}(x,v) \in C_{b}^{1,2}(\mathbb{R}_{+} \times \mathbb{R}^{2d})$, we get
\begin{equation}\label{sol}
\begin{aligned}
\langle\psi_{t}, \xi_{t}^{j}\rangle &= \langle\psi_{0}, \xi_{0}^{j}\rangle+\int_{0}^{t}\langle\partial_{s} \psi_{s}+(\mathcal{L}_{j} \psi_{s})(x,v), \xi_{s}^{j}(d x, dv)\rangle d s \\
&\phantom{=} +\sum_{\ell=1}^{L}\int_{0}^{t}\int_{\txtv} \frac{1}{\bm{\alpha}^{(\ell)}!} K_{\ell}(\bm{x},\bm{v})\biggl(-\sum_{r=1}^{\alpha_{\ell j}}\psi_{t}(\xrj,\vrj) +\int_{\yv}\left(\sum_{r=1}^{\blj}\psi_{t}(\yrj,\vprj)\right)\\
&\phantom{=}\times m_{\ell}\left(\bm{y},\bm{v'}|\bm{x},\bm{v}\right)d\bm{y}d\bm{v'}\biggr)\llmeasure ds.
\end{aligned}
\end{equation}

Let $\mathcal{P}_{j, t}, t \geq 0$, be the semigroup generated by $\mathcal{L}_{j}$, defined in (\ref{Eq:Operators}),
for  $j=1,2, \cdots, J$. Choose $\psi_{s}(x,v;t)=\mathcal{P}_{j, t-s} f(x,v)$, respectively for each $1 \leq j \leq J$, where $f \in C_{b}^{2}(\Rd \times \Rd) \text{ and }\|f\|_{L^{\infty}} \leq 1$, with $s\leq t$. Using the semigroup property, we obtain
\begin{equation}\label{group}
\begin{aligned}
\langle &f, \xi_{t}^{j}\rangle =\langle \mathcal{P}_{j,t}f, \xi_{0}^{j}\rangle +\sum_{\ell=1}^{L}\int_{0}^{t}\int_{\txtv} \frac{1}{\bm{\alpha}^{(\ell)}!} K_{\ell}(\bm{x},\bm{v})\biggl(-\sum_{r=1}^{\alpha_{\ell j}}\mathcal{P}_{\ell, t-s}f(\xrj,\vrj)\\
&\phantom{=}+\int_{\yv}\left(\sum_{r=1}^{\blj}\mathcal{P}_{\ell, t-s}f(\yrj,\vprj)\right)m_{\ell}\left(\bm{y},\bm{v'}|\bm{x},\bm{v}\right)d\bm{y}d\bm{v'}\biggr)\llmeasure ds.
\end{aligned}
\end{equation}
Because we work on finite time intervals, we subsequently make use of the bound that
$\displaystyle \sup_{t \in [0,T]} \|\mathcal{P}_{j,t} f\|_{L^\infty} \leq C\norm{f}_{L^\infty}<C<\infty$ as $\norm{f}_{L^\infty} \leq 1$ (see Chapter 4 of \cite{Pazy1983}) for some finite constant $C=C(T)<\infty$ that depends on the time horizon $T$.

From~\eqref{group},  the remainder of the argument essentially follows the proof of uniqueness in~\cite{IsaacsonSIMA2022}. For completeness we reproduce the argument here. With $M=1 \vee \displaystyle\sup_{\{t \in[0, T], j=1,2, \cdots, J\}}|\langle 1, \xi_{t}^{j}\rangle| \vee|\langle 1, \bar{\xi}_{t}^{j}\rangle|<\infty$, we get the following estimates for $\langle f, \xi_{t}^{j}-\bar{\xi}_{t}^{j}\rangle$
\begin{align*}
&|\langle f, \xi_{t}^{j}-\bar{\xi}_{t}^{j}\rangle|  \leq \sum_{\ell=1}^{L}\int_{0}^{t}\biggl|\int_{\txtv} \frac{1}{\bm{\alpha}^{(\ell)}!} K_{\ell}(\bm{x},\bm{v})\biggl(-\sum_{r=1}^{\alpha_{\ell j}}\mathcal{P}_{\ell, t-s}f(\xrj,\vrj)\\
&\phantom{=}+\int_{\yv}\left(\sum_{r=1}^{\blj}\mathcal{P}_{\ell, t-s}f(\yrj,\vprj)\right)m_{\ell}\left(\bm{y},\bm{v'}|\bm{x},\bm{v}\right)d\bm{y}d\bm{v'}\biggr)\nonumber\\
&\phantom{=}\quad\times\left|\llmeasure-\llbmeasure\right|\biggr| ds\\
&\leq C(K) \sum_{\ell=1}^{L} \frac{\alpha_{\ell j}+\beta_{\ell j}}{\bm{\alpha}^{(\ell)}!} \int_{0}^{t} \norm{\llmeasure-\llbmeasure}_{M_{F}(\xv)} d s \\
&\leq C(K) \sum_{\ell=1}^{L} \frac{\alpha_{\ell j}+\beta_{\ell j}}{\bm{\alpha}^{(\ell)}!} \int_{0}^{t}\|\otimes_{i=1}^{J}(\otimes_{r=1}^{\alpha_{\ell i}} \xi_{s}^{i})-\otimes_{i=1}^{J}(\otimes_{r=1}^{\alpha_{\ell i}} \bar{\xi}_{s}^{i})\|_{M_{F}(\xv)} d s \\
& \leq C(K) \sum_{\ell=1}^{L} \frac{\alpha_{\ell j}+\beta_{\ell j}}{\bm{\alpha}^{(\ell)}!} \int_{0}^{t} M^{|\alpha^{(\ell)}|-1} \sum_{i=1}^{J} \alpha_{\ell i}\|\xi_{s}^{i}-\bar{\xi}_{s}^{i}\|_{M_{F}(\Rd \times \Rd)} d s,\numberthis \label{Eq:limitingMeasureDifference}
\end{align*}
where the last inequality is due to the fact that $\left\langle 1, \xi_{s}^{i}\right\rangle$ or $\left\langle 1, \bar{\xi}_{s}^{i}\right\rangle$ are uniformly bounded by $M$ for all $1 \leq i \leq J$. In the second equality of Eq$\eqref{Eq:limitingMeasureDifference}$, we have used the following estimates
\begin{align*}
    &\biggl|\frac{1}{\bm{\alpha}^{(\ell)}!} K_{\ell}(\bm{x},\bm{v})\biggl(-\sum_{r=1}^{\alpha_{\ell j}}\mathcal{P}_{\ell, t-s}f(\xrj,\vrj)\\
&\phantom{=}+\int_{\yv}\left(\sum_{r=1}^{\blj}\mathcal{P}_{\ell, t-s}f(\yrj,\vprj)\right)m_{\ell}\left(\bm{y},\bm{v'}|\bm{x},\bm{v}\right)d\bm{y}d\bm{v'}\biggr)\biggr|\\
\leq&\biggl|\frac{1}{\bm{\alpha}^{(\ell)}!} K_{\ell}(\bm{x},\bm{v})\biggl(-\sum_{r=1}^{\alpha_{\ell j}}\left|\mathcal{P}_{\ell, t-s}f(\xrj,\vrj)\right|\\
&\phantom{=}+\int_{\yv}\sum_{r=1}^{\blj}\left|\mathcal{P}_{\ell, t-s}f(\yrj,\vprj)\right|m_{\ell}\left(\bm{y},\bm{v'}|\bm{x},\bm{v}\right)d\bm{y}d\bm{v'}\biggr)\biggr|\\
\leq&C(K) \sum_{\ell=1}^{L} \frac{\alpha_{\ell j}+\beta_{\ell j}}{\bm{\alpha}^{(\ell)}!}. \numberthis \label{Eq:negativeSemigroup}
\end{align*}
As in~\cite{IsaacsonSIMA2022}, we introduce a norm on $M_{F}(\Rd \times \Rd)$ defined by
\begin{equation*}
    \|\xi\|_{M_{F}(\mathbb{R}^{2d})}=\sup_{f \in C_b^2(\mathbb{R}^{2d}), \|f\|_{L^{\infty}} \leq 1}|\langle f, \xi\rangle|.
\end{equation*}
Summing over all particle types, we get
\begin{align*}
&\sum_{j=1}^{J}\|\xi_{t}^{j}-\bar{\xi}_{t}^{j}\|_{M_{F}(\mathbb{R}^{2d})} \leq C(K) \sum_{j=1}^{J} \sum_{\ell=1}^{L} \frac{\alpha_{\ell j}+\beta_{\ell j}}{\bm{\alpha}^{(\ell)}!} \int_{0}^{t} M^{|\alpha^{(\ell)}|-1} \sum_{i=1}^{J} \alpha_{\ell i}\|\xi_{s}^{i}-\bar{\xi}_{s}^{i}\|_{M_{F}(\mathbb{R}^{2d})} d s\\
\leq& C(K)LJ \max _{1 \leq \ell \leq L, 1 \leq j \leq J}\bigl(\frac{\alpha_{\ell j}+\beta_{\ell j}}{\bm{\alpha}^{(\ell)}!} M^{| \alpha^{(\ell)} |-1} \alpha_{\ell j}\bigr)\int_{0}^{t} \sum_{i=1}^{J} \|\xi_{s}^{i}-\bar{\xi}_{s}^{i}\|_{M_{F}(\mathbb{R}^{2d})} d s&
\end{align*}
for some generic constant $C<\infty$. Applying Gronwall's inequality, we get $\sum_{j=1}^{J}\norm{\xi_{t}^{j}-\bar{\xi}_{t}^{j}}_{M_{F}(\mathbb{R}^{2d})}=0$ for all $t \in[0, T]$,
  concluding the proof of the uniqueness of the limiting solution.

\section{Preliminary moment bounds}\label{S:MomentBounds}
The purpose of this section is to obtain moment bounds on $\mu_{t}^{\zeta,j}$ that are uniform with respect to $\zeta\in \mathbb{R}_{+}\times \mathbb{R}_{+}$, for all $t \in [0,T]$ and for any $T<\infty$, i.e., to bound
\[
\sup_{t\in[0,T]}\theta^{\vec{\zeta},j,4}_{t}=\sup_{t\in[0,T]}\mathbb{E}\la |x|^{4} +|v|^{4},\mu^{\zeta,j}_{t}(dx,dv)\ra
\]
uniformly with respect to $\zeta\in\mathbb{R}_{+} \times \mathbb{R}_{+}$.

The first step is to prove that the regularization by $\eta$ appropriately recovers the placement measures in the limit as $\eta\downarrow 0$. In particular, we have the following lemma.
\begin{lemma}\label{L:placementDensityDifference}
For any $\eta \geq 0$ small enough, $1 \leq \ell \leq L, \bm{y} \in \mathbb{Y}^{(\ell)}, \bm{v'} \in \mathbb{V'}^{(\ell)}, \bm{x} \in \mathbb{X}^{(\ell)}, \bm{v} \in \mathbb{V}^{(\ell)}$, and $f \in C_{b}^{2}\left(\mathbb{Y}^{(\ell)},\mathbb{V'}^{(\ell)}\right)$, there exists a constant $C$ such that
$$
\left|\int_{\yv} f(\bm{y},\bm{v'})\left(m_{\ell}^{\eta}(\bm{y},\bm{v'} \mid \bm{x},\bm{v})-m_{\ell}(\bm{y},\bm{v'} \mid \bm{x},\bm{v})\right) d \bm{y} d \bm{v'}\right| \leq C\eta \|f\|_{C_{b}^{2}\left(\yv\right)}.
$$
\end{lemma}

\begin{proof}\label{P:Identification}
In all cases below, $B(z,\eta)$ will represent the ball centered at the point $z$ belonging to the appropriate Euclidean space with radius $\eta>0$.

\underline{Case 1: Reaction of the form $S_{i} \rightarrow S_{j}$.}
\begin{align*}
&\left|\int_{\yv} f(\bm{y},\bm{v'})\left(m_{\ell}^{\eta}(\bm{y},\bm{v'} \mid \bm{x},\bm{v})-m_{\ell}(\bm{y},\bm{v'} \mid \bm{x},\bm{v})\right) d \bm{y} d \bm{v'}\right|\\
=&\biggl|\biggl[\int_{\mathbb{R}^{2 d}} f\left(y,v'\right) G_{\eta}\left(y-x\right) G_{\eta}\left(v'-v\right)d ydv' - f(x,v)\biggr]\biggr|\\
=&\biggl|\biggl[\int_{\mathbb{R}^{2 d}} \left(f\left(y,v'\right)- f(x,v)\right) G_{\eta}\left(y-x\right) G_{\eta}\left(v'-v\right)d ydv' \biggr]\biggr|\\
\leq&\int_{B\left(\left(x,v\right),\sqrt{2}\eta\right)} \left|f\left(y,v'\right)- f(x,v)\right| G_{\eta}\left(y-x\right) G_{\eta}\left(v'-v\right)d ydv' \\
\leq&\int_{B\left(\left(x,v\right),\sqrt{2}\eta\right)} \sqrt{2}\eta \norm{f}_{C^1_b(\Rd \times \Rd)}G_{\eta}\left(y-x\right) G_{\eta}\left(v'-v\right)d ydv' \\
\leq&\sqrt{2}\eta\norm{f}_{C^1_b(\Rd \times \Rd)}.
\end{align*}

\underline{Case 2: Reaction of the form $S_{i} \rightarrow S_{j}+S_{k}$.}
\begin{align*}
&\left|\int_{\yv} f(\bm{y},\bm{v'})\left(m_{\ell}^{\eta}(\bm{y},\bm{v'} \mid \bm{x},\bm{v})-m_{\ell}(\bm{y},\bm{v'} \mid \bm{x},\bm{v})\right) d \bm{y} d \bm{v'}\right|\\
=&\biggl| \sum_{i=1}^{I} p_{i} \times\biggl[\int_{\mathbb{R}^{4 d}} f\left(y_{1}, y_{2}, v_1', v_2'\right) \rho\left(\left|y_{1}-y_{2}\right|,\left|{v}_{1}-{v}_{2}\right|\right) G_{\eta}\left(x-\left(\alpha_{i} y_{1}+\left(1-\alpha_{i}\right) y_{2}\right)\right)\\
&\phantom{=}\times G_{\eta}\left(v-\frac{m_1 v_1+ m_2 v_2}{m_3}\right)d y_{1} d y_{2} d v_{1} d v_{2}\\
&\phantom{=}- \int_{\mathbb{R}^{4 d}} f\left(y_{1}, y_{2}, v_1', v_2'\right) \rho\left(\left|y_{1}-y_{2}\right|,\left|{v}_{1}-{v}_{2}\right|\right) \delta\left(x-\left(\alpha_{i} y_{1}+\left(1-\alpha_{i}\right) y_{2}\right)\right)\\
&\phantom{=}\times \delta\left(v-\frac{m_1 v_1+ m_2 v_2}{m_3}\right)d y_{1} d y_{2} d v_{1} d v_{2} \biggr]\biggr|\\
=&\biggl| \sum_{i=1}^{I} p_{i} \times\biggl[\int_{\mathbb{R}^{4 d}} f\left(w+y_{2}, y_{2}, u+{v}_2, {v}_2\right) \rho\left(\left|w\right|,\left|u\right|\right) G_{\eta}\left(x-\alpha_i w-y_2\right)\\
&\phantom{=}\times G_{\eta}\left(v-\frac{m_1}{m_3}u-\frac{m_1+m_2}{m_3}v_2\right)d y_{2} d v_{2} d w d u\\
&\phantom{=}- \int_{\mathbb{R}^{4 d}} f\left(w+y_{2}, y_{2}, u+{v}_2, {v}_2\right) \rho\left(\left|w\right|,\left|u\right|\right) \delta\left(x-\alpha_i w-y_2\right)\nonumber\\
&\phantom{=}\qquad\times\delta\left(v-\frac{m_1}{m_3}u-\frac{m_1+m_2}{m_3}v_2\right)d y_{2} d v_{2} d w d u \biggr]\biggr|\\
=&\biggl| \sum_{i=1}^{I} p_{i} \times\biggl[\int_{\mathbb{R}^{2 d}} \rho\left(\left|w\right|,\left|u\right|\right) \biggl(\int_{\mathbb{R}^{2 d}}\biggl(f\left(w+y_{2}, y_{2}, u+{v}_2, {v}_2\right)\\
&\phantom{=}-\frac{m_3}{m_1+m_2}f\left(w+x-\alpha_i w, x-\alpha_i w, u + \frac{m_3}{m_1+m_2}v - \frac{m_1}{m_1+m_2}u,\frac{m_3}{m_1+m_2}v - \frac{m_1}{m_1+m_2}u\right)\biggr)\\
&\phantom{=}\times G_{\eta}\left(x-\alpha_i w-y_2\right)G_{\eta}\left(v-\frac{m_1}{m_3}u-v_2\right) d y_{2} d v_{2}\biggr)d w d u\biggr]\biggr|\\
\leq&\biggl| \sum_{i=1}^{I} p_{i} \times\biggl[\int_{\mathbb{R}^{2 d}} \rho\left(\left|w\right|,\left|u\right|\right) \biggl(\int_{B\left(\left(x-\alpha_i w,v-\frac{m_1}{m_3}u\right),\sqrt{2}\eta\right)}\biggl|f\left(w+y_{2}, y_{2}, u+{v}_2, {v}_2\right) \\
&\phantom{=}-f\left(w+x-\alpha_i w, x-\alpha_i w, u + v - \frac{m_1}{m_3}u,v - \frac{m_1}{m_3}u\right)\biggr|\\
&\phantom{=}\times G_{\eta}\left(x-\alpha_i w-y_2\right)G_{\eta}\left(v-\frac{m_1}{m_3}u-v_2\right) d y_{2} d v_{2}\biggr)d w d u\biggr]\biggr|\text{ (by Assumption \ref{A:conservationofMomentum})}\\
\leq&\biggl| \sum_{i=1}^{I} p_{i} \times\biggl[\int_{\mathbb{R}^{2 d}} \rho\left(\left|w\right|,\left|u\right|\right) \biggl(\int_{B\left(\left(x-\alpha_i w,v-\frac{m_1}{m_3}u\right),\sqrt{2}\eta\right)} \sqrt{2}\eta\norm{f}_{C^1_b(R^{4d})} G_{\eta}\left(x-\alpha_i w-y_2\right)\\
&\phantom{=}\times G_{\eta}\left(v-\frac{m_1}{m_3}u-v_2\right) d y_{2} d v_{2}\biggr)d w d u\biggr]\biggr|\\
\leq&\sqrt{2}\eta\norm{f}_{C^1_b(R^{4d})}.
\end{align*}

\underline{Case 3: Reaction of the form $S_{i}+S_{k} \rightarrow S_{j}$.}
\begin{align*}
&\left|\int_{\yv} f(\bm{y},\bm{v'})\left(m_{\ell}^{\eta}(\bm{y},\bm{v'} \mid \bm{x},\bm{v})-m_{\ell}(\bm{y},\bm{v'} \mid \bm{x},\bm{v})\right) d \bm{y} d \bm{v'}\right|\\
=&\biggl| \sum_{i=1}^{I} p_{i} \times\biggl[\int_{\mathbb{R}^{2 d}} f\left(y,v'\right) G_{\eta}\left(y-\left(\alpha_{i} x_{1}+\left(1-\alpha_{i}\right) x_{2}\right)\right)G_{\eta}\left(v'-\frac{m_1 v_1+ m_2 v_2}{m_3}\right)d ydv'\\
&\phantom{=}- f\left(\alpha_{i} x_{1}+\left(1-\alpha_{i}\right) x_{2},\frac{m_1 v_1+ m_2 v_2}{m_3}\right)\biggr]\biggr|\\
=&\biggl| \sum_{i=1}^{I} p_{i} \times\biggl[\int_{\mathbb{R}^{2 d}} \left(f\left(y,v'\right) - f\left(\alpha_{i} x_{1}+\left(1-\alpha_{i}\right) x_{2},\frac{m_1 v_1+ m_2 v_2}{m_3}\right)\right)\\
&\phantom{=}\times G_{\eta}\left(y-\left(\alpha_{i} x_{1}+\left(1-\alpha_{i}\right) x_{2}\right)\right)G_{\eta}\left(v'-\frac{m_1 v_1+ m_2 v_2}{m_3}\right)d ydv'\biggr]\biggr|\\
\leq&\biggl| \sum_{i=1}^{I} p_{i} \times\biggl[\int_{B\left(\left(\alpha_{i} x_{1}+\left(1-\alpha_{i}\right) x_{2},\frac{m_1 v_1+ m_2 v_2}{m_3}\right),\sqrt{2}\eta\right)} \left|f\left(y,v'\right) - f\left(\alpha_{i} x_{1}+\left(1-\alpha_{i}\right) x_{2},\frac{m_1 v_1+ m_2 v_2}{m_3}\right)\right|\\
&\phantom{=}\times G_{\eta}\left(y-\left(\alpha_{i} x_{1}+\left(1-\alpha_{i}\right) x_{2}\right)\right)G_{\eta}\left(v'-\frac{m_1 v_1+ m_2 v_2}{m_3}\right)d ydv'\biggr]\biggr|\\
\leq&\biggl| \sum_{i=1}^{I} p_{i} \times\biggl[\int_{B\left(\left(\alpha_{i} x_{1}+\left(1-\alpha_{i}\right) x_{2},\frac{m_1 v_1+ m_2 v_2}{m_3}\right),\sqrt{2}\eta\right)} \sqrt{2}\eta\norm{f}_{C^1_b(R^{4d})} G_{\eta}\left(y-\left(\alpha_{i} x_{1}+\left(1-\alpha_{i}\right) x_{2}\right)\right)\\
&\phantom{=}\times G_{\eta}\left(v'-\frac{m_1 v_1+ m_2 v_2}{m_3}\right)d ydv'\biggr]\biggr|\\
\leq&\sqrt{2}\eta\norm{f}_{C^1_b(R^{4d})}.
\end{align*}

\underline{Case 4: Reaction of the form $S_i + S_k \rightarrow S_j + S_r$.}  For notational convenience, let us define the set
\begin{align*}
\mathcal{B}(x_1,x_2,\bm{v},w;\eta)&=B\left(\left(x_1,x_2,\frac{m_1 v_1 + m_2 v_2}{m_3 + m_4} - \frac{m_3}{m_3 + m_4}w\right),\sqrt{3}\eta\right),
\end{align*}
denoting the ball centered at $\left(x_1,x_2,\frac{m_1 v_1 + m_2 v_2}{m_3 + m_4} - \frac{m_3}{m_3 + m_4}w\right)\in\mathbb{R}^{3d}$ with radius $\sqrt{3}\eta$.  Then, we have
\begin{align*}
&\left|\int_{\yv} f(\bm{y},\bm{v'})\left(m_{\ell}^{\eta}(\bm{y},\bm{v'} \mid \bm{x},\bm{v})-m_{\ell}(\bm{y},\bm{v'} \mid \bm{x},\bm{v})\right) d \bm{y} d \bm{v'}\right|\\
=&\biggl|\int_{\mathbb{R}^{4 d}} (m_3 + m_4) f\left(y_{1}, y_{2}, v_1', v_2'\right) \rho(|v_1' - v_2'|)\biggl[G_{\eta}(m_3 v_1' + m_4 v_2' - m_1 v_1 - m_2 v_2)\\
&\phantom{=} \times \biggl(p \times G_{\eta}(x_1-y_1) G_{\eta}(x_2-y_2)+(1-p) \times G_{\eta}(x_1-y_2) G_{\eta}(x_2-y_1)\biggr)\nonumber\\
&\phantom{=}\qquad-\delta(m_3 v_1' + m_4 v_2' - m_1 v_1 - m_2 v_2)\\
&\phantom{=} \times \biggl(p \times \delta_{(x_1, x_2)}\left((y_1, y_2)\right)+(1-p) \times \delta_{(x_1, x_2)}\left((y_2, y_1)\right)\biggr)\biggr]d y_{1} d y_{2} d v_{1}' d v_{2}'\biggr|\\
\leq& C \biggl|p\int_{\mathbb{R}^{4 d}} \biggl[f\left(y_{1}, y_{2}, w+{v}_2', {v}_2'\right) - f\left(x_{1}, x_{2}, \frac{m_1 v_1 + m_2 v_2}{m_3 + m_4} + \frac{m_4}{m_3 + m_4}w, \frac{m_1 v_1 + m_2 v_2}{m_3 + m_4} - \frac{m_3}{m_3 + m_4}w\right)\biggr]\\
&\times \rho(|w|)G_{\eta}\left(v_2' - \frac{m_1 v_1 + m_2 v_2}{m_3 + m_4} + \frac{m_3}{m_3 + m_4}w\right)G_{\eta}(x_1-y_1) G_{\eta}(x_2-y_2) d v_{2}' d y_{1} d y_{2} d w+(1-p)\\
&\times\int_{\mathbb{R}^{4 d}} \biggl[f\left(y_{1}, y_{2}, w+{v}_2', {v}_2'\right) - f\left(x_{2}, x_{1}, \frac{m_1 v_1 + m_2 v_2}{m_3 + m_4} + \frac{m_4}{m_3 + m_4}w, \frac{m_1 v_1 + m_2 v_2}{m_3 + m_4} - \frac{m_3}{m_3 + m_4}w\right)\biggr]\\
&\phantom{=}\times \rho(|w|)G_{\eta}\left(v_2' - \frac{m_1 v_1 + m_2 v_2}{m_3 + m_4} + \frac{m_3}{m_3 + m_4}w\right)G_{\eta}(x_1-y_2) G_{\eta}(x_2-y_1) d v_{2}' d y_{1} d y_{2} d w \biggr|\\
\leq&C\biggl|p\int_{\Rd}\rho(|w|)\int_{\mathcal{B}(x_1,x_2,\bm{v},w;\eta)}
\biggl[f\left(y_{1}, y_{2}, w+{v}_2', {v}_2'\right) -  \\
&\hspace{4cm}-f\left(x_{1}, x_{2}, \frac{m_1 v_1 + m_2 v_2}{m_3 + m_4} + \frac{m_4}{m_3 + m_4}w, \frac{m_1 v_1 + m_2 v_2}{m_3 + m_4} -\frac{m_3}{m_3 + m_4}w\right)\biggr]\\
&\phantom{=}\times G_{\eta}\left(v_2' - \frac{m_1 v_1 + m_2 v_2}{m_3 + m_4} + \frac{m_3}{m_3 + m_4}w\right)G_{\eta}(x_1-y_1) G_{\eta}(x_2-y_2) d v_{2}' d y_{1} d y_{2} d w+\\
&+(1-p)\int_{\Rd}\rho(|w|)
\int_{\mathcal{B}(x_2,x_1,\bm{v},w;\eta)}
\biggl[f\left(y_{1}, y_{2}, w+{v}_2', {v}_2'\right) - \\
&\hspace{4cm} -f\left(x_{2}, x_{1}, \frac{m_1 v_1 + m_2 v_2}{m_3 + m_4} + \frac{m_4}{m_3 + m_4}w, \frac{m_1 v_1 + m_2 v_2}{m_3 + m_4} - \frac{m_3}{m_3 + m_4}w\right)\biggr]\\
&\phantom{=}G_{\eta}\left(v_2' - \frac{m_1 v_1 + m_2 v_2}{m_3 + m_4} + \frac{m_3}{m_3 + m_4}w\right)G_{\eta}(x_1-y_2) G_{\eta}(x_2-y_1) d v_{2}' d y_{1} d y_{2} d w \biggr|\\
\leq&C\sqrt{3}\eta\norm{f}_{C^1_b(R^{4d})}\biggl|p\int_{\Rd}\rho(|w|)
\int_{\mathcal{B}(x_1,x_2,\bm{v},w;\eta)} G_{\eta}\left(v_2' - \frac{m_1 v_1 + m_2 v_2}{m_3 + m_4} + \frac{m_3}{m_3 + m_4}w\right)\times\\
&\hspace{9cm}\times G_{\eta}(x_1-y_1) G_{\eta}(x_2-y_2) d v_{2}' d y_{1} d y_{2} d w\\
&\phantom{=}+(1-p)\int_{\Rd}\rho(|w|)
\int_{\mathcal{B}(x_2,x_1,\bm{v},w;\eta)}
G_{\eta}\left(v_2' - \frac{m_1 v_1 + m_2 v_2}{m_3 + m_4} + \frac{m_3}{m_3 + m_4}w\right)\times\\
&\hspace{9cm}\times G_{\eta}(x_1-y_2) G_{\eta}(x_2-y_1) d v_{2}' d y_{1} d y_{2} d w \biggr|\\
\leq&C\eta\norm{f}_{C^1_b(R^{4d})}
\end{align*}
for a generic constant, $C$.
\end{proof}

Next, we discuss the uniform moment bounds of $\displaystyle\sup_{\zeta\in (0,1)^{2}}\sup_{t\in[0,T]}\theta^{\vec{\zeta},j,4}_{t}$ in  Lemma \ref{L:MomentBounds}, which actually proves something slightly stronger. It proves that if a given $2p-$moment is finite pointwise in $\vec{\zeta}$, then it will be uniformly bounded in $\vec{\zeta}$. In this paper we use Lemma \ref{L:MomentBounds} with $p=2$. For the purpose of the more general result we make a slightly weaker assumption than Assumption \ref{A:AssumptionMomentFiniteness}.
\begin{assumption}\label{A:AssumptionMomentFinitenessGeneral}
Let $p$ be a given integer. In the context of Assumption \ref{A:AssumptionRho}, replace the last requirement of Assumption \ref{A:AssumptionRho} by
$\int_{\Rd \times \Rd}\left(|w|^{2p}+|u|^{2p}\right)\rho(w,u)dwdu<\infty$. In addition assume that
 $\sup_{t\in(0,T]}\sum_{j=1}^{J}\mathbb{E}\la |x|^{2p} +|v|^{2p},\mu_{t}^{\zeta, j}(dx,dv)\ra<\infty$. For $t=0$ specifically, we assume that $\displaystyle\sup_{\zeta\in\mathbb{R}_{+} \times \mathbb{R}_{+}}\sum_{j=1}^{J}\mathbb{E}\la |x|^{2p} +|v|^{2p},\mu_{0}^{\zeta, j}(dx,dv)\ra<C<\infty$.
\end{assumption}

\begin{lemma}\label{L:MomentBounds}
Let Assumptions \ref{A:AssumptionRho} as well as Assumption \ref{A:AssumptionMomentFinitenessGeneral} hold.
Then there exists a finite constant $C<\infty$ (not depending on $\zeta$) such that the following moment bound holds
\begin{align*}
\sup_{t\in[0,T]}\sum_{j=1}^{J} \theta^{\zeta,j,2p}_{t}&=\sup_{t\in[0,T]}\sum_{j=1}^{J}\mathbb{E}\la |x|^{2p} +|v|^{2p},\mu^{\zeta,j}_{t}(dx,dv)\ra \leq C<\infty,
\end{align*}
where the $p$ used above corresponds to the one from Assumption \ref{A:AssumptionRho} and Assumption \ref{A:AssumptionMomentFinitenessGeneral}.
\end{lemma}
\begin{proof}
Recall that in Assumption \ref{A:AssumptionMomentFinitenessGeneral} we only assume that $\theta^{\zeta,j,2p}_{t}$ is finite (not necessarily uniformly bounded in $\zeta\in\mathbb{R}_{+} \times \mathbb{R}_{+}$). Our aim is to prove  $\displaystyle\sup_{t\in[0,T]}\sum_{j=1}^{J} \theta^{\zeta,j,2p}_{t}<C,$ i.e., these moments are in fact uniformly bounded in $\zeta\in\mathbb{R}_{+} \times \mathbb{R}_{+}$.

We will use Eq \eqref{Eq:PathDescription} for $f(x,v)=|x|^{2p} + |v|^{2p},$ and the validity of such operation essentially relies on the finiteness assumption of the involved moments per Assumption \ref{A:AssumptionMomentFinitenessGeneral}.
 \begin{align}
    &\mathbb{E}\left[\left\langle f, \mu_{t}^{\bm{\zeta}, j}\right\rangle\right]= \mathbb{E}\left[\fmi\right] + \mathbb{E}\left[\int_0^t \left\langle (\mathcal{L}_{j}f)(x,v), \msmzj(dx,dv)\right\rangle ds\right]\nonumber\\
      &\phantom{=} +\sum_{\ell=1}^{L}\mathbb{E}\biggl[\int_0^t \int_{\txtv}\frac{1}{\bm{\alpha}^{(\ell)}!}K_\ell(\bm{x},\bm{v})\biggl(-\sum_{r=1}^{\alj}f\left(x_r^{(j)},v_r^{(j)}\right)+\int_{\yv}\left(\sum_{r=1}^{\blj}f\left(y_r^{(j)},{v'}_{r}^{(j)}\right)\right)\nonumber\\
&\phantom{=}\times m_{\ell}^{\eta}\left(\bm{y},\bm{v'}|\bm{x},\bm{v}\right)d\bm{y}d\bm{v'}\biggr)\lmeasure ds\biggr],\label{Eq:Expectation_Mean_Formula}
\end{align}
with the operator $(\mathcal{L}_{j}f)(x,v)=v \nabla_{x}f(x,v) - b_{j} v \nabla_{v}f(x,v) + b_{j}^2 D_j \Delta_{v} f(x,v)$ defined via (\ref{Eq:Operators}).

Without loss of generality, we now use $f(x,v)=  |v|^{2p}$ in Eq \eqref{Eq:Expectation_Mean_Formula}. Note that by Assumption \ref{A:molarBdd} we have that $\la 1, \mu_t^{{\zeta}, j} \ra \leq C_{\circ}$ and by Assumption \ref{A:AssumptionMomentFinitenessGeneral} that $\displaystyle\sup_{\zeta\in\mathbb{R}_{+} \times \mathbb{R}_{+}}\theta^{\zeta,j,2p}_{0}\leq C<\infty$.

Firstly, we compute an upper bound for $\mathbb{E}\left[\int_0^t \left\langle (\mathcal{L}_{j}f)(x,v), \msmzj(dx,dv)\right\rangle ds\right]$. For this purpose, we have
\begin{align}
 &\mathbb{E}\left[\int_0^t \left\langle (\mathcal{L}_{j}f)(x,v), \msmzj(dx,dv)\right\rangle ds\right]\nonumber\\
 =&\mathbb{E}\left[\int_0^t \left\langle v \nabla_{x}f(x,v) - b_{j} v \nabla_{v}f(x,v) + b_{j}^2 D_j \Delta_{v} f(x,v), \msmzj(dx,dv)\right\rangle ds\right]\nonumber\\
 =&\mathbb{E}\left[\int_0^t \left\langle  - b_{j} 2p |v|^{2p} + b_{j}^2 D_j 2p(2p-1)|v|^{2p-2}, \msmzj(dx,dv)\right\rangle ds\right]\nonumber\\
 =&- b_{j} 2p \int_0^t \theta^{\zeta,j,2p}_{s}ds+b_{j}^2 D_j 2p(2p-1) \mathbb{E}\left[\int_0^t \left\langle  |v|^{2p-2}, \msmzj(dx,dv)\right\rangle ds\right]\nonumber\\
 \leq& - b_{j} 2p \int_0^t \theta^{\zeta,j,2p}_{s}ds\nonumber\\
 &\quad+b_{j}^2 D_j 2p(2p-1)\int_{0}^{t}\left[\frac{1}{p}\mathbb{E}\left\langle  1, \msmzj(dx,dv)\right\rangle+\frac{p-1}{p}\mathbb{E} \left\langle  |v|^{2p}, \msmzj(dx,dv)\right\rangle \right]ds\nonumber\\
\leq& - b_{j} 2p \int_0^t \theta^{\zeta,j,2p}_{s}ds+b_{j}^2 D_j 2p(2p-1) \int_{0}^{t} \left[\frac{C_{\circ}}{p}+\frac{p-1}{p}\theta^{\zeta,j,2p}_{s}\right]ds \nonumber\\
\leq& C\left[1+ \int_{0}^{t}\theta^{\zeta,j,2p}_{s} ds\right] \label{Eq:UniformOperatorBound}
\end{align}
for some constant $C<\infty$ that is uniform in $\zeta\in\mathbb{\R}_{+}\times \mathbb{\R}_{+}$.

We claim that a similar bound holds for the reaction term. Namely, there is a $C<\infty$ uniform in $\zeta \in\mathbb{\R}_{+}\times \mathbb{\R}_{+}$ such that
\begin{align}
&\sum_{\ell=1}^{L}\mathbb{E}\biggl[\int_0^t \int_{\txtv}\frac{1}{\bm{\alpha}^{(\ell)}!}K_\ell(\bm{x},\bm{v})\biggl(-\sum_{r=1}^{\alj}f\left(x_r^{(j)},v_r^{(j)}\right)+\int_{\yv}\left(\sum_{r=1}^{\blj}f\left(y_r^{(j)},{v'}_{r}^{(j)}\right)\right)\nonumber\\
&\phantom{=}\times m_{\ell}^{\eta}\left(\bm{y},\bm{v'}|\bm{x},\bm{v}\right)d\bm{y}d\bm{v'}\biggr)\lmeasure ds\biggr]\nonumber\\
&\leq C\paren{1 + \sum_{j=1}^{J}\int_{0}^{t}\theta^{\zeta,j,2p}_{s} ds} \label{Eq:UniformReactionsBound}
\end{align}
Indeed,
\begin{align}
&\sum_{\ell=1}^{L}\mathbb{E}\biggl[\int_0^t \int_{\txtv}\frac{1}{\bm{\alpha}^{(\ell)}!}K_\ell(\bm{x},\bm{v})\biggl(-\sum_{r=1}^{\alj}f\left(x_r^{(j)},v_r^{(j)}\right)+\int_{\yv}\left(\sum_{r=1}^{\blj}f\left(y_r^{(j)},{v'}_{r}^{(j)}\right)\right)\nonumber\\
&\phantom{=}\times m_{\ell}^{\eta}\left(\bm{y},\bm{v'}|\bm{x},\bm{v}\right)d\bm{y}d\bm{v'}\biggr)\lmeasure ds\biggr]\nonumber\\
\leq& C(K) \sum_{\ell=1}^{L}\mathbb{E}\biggl[\int_0^t \sum_{r=1}^{\alj}\left\langle |\vrj|^{2p},\msmzj(d\xrj,d\vrj)\right\rangle\biggr]+\sum_{\ell=1}^{L}\mathbb{E}\biggl[\int_0^t \int_{\txtv}\frac{1}{\bm{\alpha}^{(\ell)}!}K_\ell(\bm{x},\bm{v})\nonumber\\
&\phantom{=}\times\biggl(\int_{\yv}\left(\sum_{r=1}^{\blj}|\vprj|^{2p}\right)m_{\ell}\left(\bm{y},\bm{v'}|\bm{x},\bm{v}\right)d\bm{y}d\bm{v'}\biggr)\lmeasure ds\biggr]\nonumber\\
&\phantom{=}+\sum_{\ell=1}^{L}\mathbb{E}\biggl[\int_0^t \int_{\txtv}\frac{1}{\bm{\alpha}^{(\ell)}!}K_\ell(\bm{x},\bm{v})\biggl(\int_{\yv}\left(\sum_{r=1}^{\blj}|\vprj|^{2p}\right)\bigl(m_{\ell}^{\eta}\left(\bm{y},\bm{v'}|\bm{x},\bm{v}\right)\nonumber\\
&\phantom{=}-m_{\ell}\left(\bm{y},\bm{v'}|\bm{x},\bm{v}\right)\bigr)d\bm{y}d\bm{v'}\biggr)\lmeasure ds\biggr].\nonumber
\end{align}
We first estimate
\begin{align*}
    I &\coloneqq \mathbb{E}\biggl[ \int_{\txtv}\frac{1}{\bm{\alpha}^{(\ell)}!}K_\ell(\bm{x},\bm{v})\\
&\phantom{=} \times \biggl(\int_{\yv}\biggl(\sum_{r=1}^{\blj}|\vprj|^{2p}\biggr)\bigl(m_{\ell}\left(\bm{y},\bm{v'}|\bm{x},\bm{v}\right)d\bm{y}d\bm{v'}\biggr)\lmeasure \biggr]
\end{align*}
by analyzing each reaction case.

\underline{Case 1: Reaction of the form $S_{i} \rightarrow S_{j}$.}
\begin{align*}
    I &= \mathbb{E}\biggl[ \int_{\Rd \times \Rd}K_\ell(x,v)\left(\int_{\Rd \times \Rd}|v'|^{2p}\delta_{x}(y)\delta_{v}(v')dy dv' \right)\mszi(dx,dv) \biggr]\\
    &\leq C(K)\mathbb{E}\biggl[\int_{\Rd \times \Rd}|v|^{2p}\mszi(dx,dv) \biggr]\\
    &\leq C(K)\theta^{\zeta,j,2p}_{s}\\
\end{align*}

\underline{Case 2: Reaction of the form $S_{i} \rightarrow S_{j}+S_{k}$.}
\begin{align*}
    I &\leq \mathbb{E}\biggl[ \int_{\Rd \times \Rd}K_\ell(x,v)\biggl(\int_{\mathbb{R}^{4d}}\left(|v_1|^{2p} + |v_2|^{2p}\right)\rho\left(|y_1-y_2|,|v_1-v_2|\right) \sum_{i'=1}^{I} p_{i'} \times \delta\left(x-\left(\alpha_{i'} y_1+(1-\alpha_{i'}) y_2\right)\right)\\
    &\phantom{=}\times \delta\left(v-\frac{m_1 v_1+ m_2 v_2}{m_3}\right)dy_1 dy_2 dv_1 dv_2 \biggr)\mszi(dx,dv) \biggr]\\
    &\leq \sum_{i'=1}^{I} p_{i'} \mathbb{E}\biggl[ \int_{\Rd \times \Rd}K_\ell(x,v)\biggl(\int_{\mathbb{R}^{4d}}\left(|u+v_2|^{2p} + |v_2|^{2p}\right)\rho\left(|w|,|u|\right) \delta\left(x-\left(\alpha_{i'} w+ y_2\right)\right)\\
    &\phantom{=}\times \delta\left(v_2 - \left(v- \frac{m_1}{m_3}u\right)\right)dy_2dv_2 dw du \biggr)\mszi(dx,dv) \biggr]\text{ (by Assumption \ref{A:conservationofMomentum})}\\
    &\leq \sum_{i'=1}^{I} p_{i'} \mathbb{E}\biggl[ \int_{\Rd \times \Rd}K_\ell(x,v)\biggl(\int_{\Rd \times \Rd}\left(|v + \frac{m_2}{m_3}u|^{2p} + |v - \frac{m_1}{m_3}u|^{2p}\right)\rho\left(|w|,|u|\right)dw du \biggr)\mszi(dx,dv) \biggr]\\
    &\leq C\sum_{i'=1}^{I} p_{i'} \mathbb{E}\biggl[ \int_{\Rd \times \Rd}K_\ell(x,v)\biggl(\int_{\Rd \times \Rd}\left(|v|^{2p} + |u|^{2p}\right)\rho\left(|w|,|u|\right)dw du \biggr)\mszi(dx,dv) \biggr]\\
    &\leq C(C_{\circ},K)(1+\theta^{\zeta,i,2p}_{s})
\end{align*}

\underline{Case 3: Reaction of the form $S_{i}+S_{k} \rightarrow S_{j}$.}
\begin{align*}
    I &= \sum_{i'=1}^{I} p_{i'} \mathbb{E}\biggl[ \int_{\mathbb{R}^{4d}}K_\ell(x_1,x_2,v_1,v_2)\biggl(\int_{\Rd \times \Rd}|v'|^{2p}\delta\left(y-\left(\alpha_{i'} x_1+(1-\alpha_{i'}) x_2\right)\right)\\
    &\phantom{=}\times \delta\left(v'-\frac{m_1 v_1+ m_2 v_2}{m_3}\right)dy dv' \biggr)\mszi(dx_1,dv_1) \mszk(dx_2,dv_2) \biggr]\\
    &\leq \sum_{i'=1}^{I} p_{i'} \mathbb{E}\biggl[ \int_{\mathbb{R}^{4d}}K_\ell(x_1,x_2,v_1,v_2)\left|\frac{m_1 v_1+ m_2 v_2}{m_3}\right|^{2p}\mszi(dx_1,dv_1) \mszk(dx_2,dv_2) \biggr]\\
    &\leq C(C_{\circ},K)\sum_{j=1}^J\theta^{\zeta,j,2p}_{s}
\end{align*}

\underline{Case 4: Reaction of the form $S_i + S_k \rightarrow S_j + S_r$.}
\begin{align*}
    I &\leq \mathbb{E}\biggl[ \int_{\mathbb{R}^{4d}}K_\ell(x_1,x_2,v_1,v_2)\\
    &\phantom{=}\times \biggl(\int_{\mathbb{R}^{4d}}\left(|v_1'|^{2p} +|v_2'|^{2p}\right)(m_3 + m_4) \rho(|v_1' - v_2'|)\delta(m_3 v_1' + m_4 v_2' - m_1 v_1 - m_2 v_2)\\
    &\phantom{=}\times \biggl(p \times \delta_{(x_1, x_2)}\left((y_1, y_2)\right)\\
    &\phantom{=}\qquad +(1-p) \times \delta_{(x_1, x_2)}\left((y_2, y_1)\right)\biggr)dy_1 dy_2 dv_1' dv_2'\biggr) \mszi(dx_1,dv_1) \mszk(dx_2,dv_2)\biggr]\\
    &\leq C(K)\mathbb{E}\biggl[ \int_{\mathbb{R}^{4d}}\biggl(\int_{\Rd \times \Rd}\left(|v_1'|^{2p} +|v_2'|^{2p}\right)(m_3 + m_4) \rho(|v_1' - v_2'|)\\
    &\phantom{=}\times \delta(m_3 v_1' + m_4 v_2' - m_1 v_1 - m_2 v_2)dv_1'dv_2'\biggr) \mszi(dx_1,dv_1) \mszk(dx_2,dv_2)\biggr]\\
     &\leq C(K)\mathbb{E}\biggl[ \int_{\mathbb{R}^{4d}}\biggl(\frac{m_3+m_4}{m_3}\int_{\mathbb{R}^{d}}\left(\left|
    \frac{m_1v_1 + m_2v_2 -m_4v_2'}{m_3}\right|^{2p} +|v_2'|^{2p}\right)\\
    &\phantom{=}\times \rho\left(\left| \frac{m_1v_1 + m_2v_2}{m_3} -\frac{m_3 + m_4}{m_3}v_2'\right|\right)dv_2'\biggr) \mszi(dx_1,dv_1) \mszk(dx_2,dv_2)\biggr]\\
     &\leq C(K)\mathbb{E}\biggl[ \int_{\mathbb{R}^{4d}}\biggl(\frac{m_3+m_4}{m_3}\int_{\mathbb{R}^{d}}\left(\left|
    \frac{m_1v_1 + m_2v_2}{m_3}\right|^{2p} + \left|
    \frac{m_3 + m_4}{m_3}v'_2\right|^{2p} \right)\\
    &\phantom{=}\times \rho\left(\left|\frac{m_1v_1 + m_2v_2}{m_3} -\frac{m_3 + m_4}{m_3}v_2'\right|\right)dv_2'\biggr) \mszi(dx_1,dv_1) \mszk(dx_2,dv_2)\biggr]\\
    &\phantom{=}\left(u = \frac{m_1v_1 + m_2v_2}{m_3} -\frac{m_3 + m_4}{m_3}v_2'\right)\\
     &\leq C(K) \mathbb{E}\biggl[ \int_{\mathbb{R}^{4d}}\biggl(\int_{\mathbb{R}^{d}}\left(\left|
    \frac{m_1v_1 + m_2v_2}{m_3}\right|^{2p} + \left|
   u\right|^{2p} \right)\rho(\left| u\right|)du\biggr) \mszi(dx_1,dv_1) \mszk(dx_2,dv_2)\biggr]\\
    &\leq C(K) \mathbb{E}\biggl[ \int_{\mathbb{R}^{4d}}\left(1+\left|\frac{m_1 v_1+ m_2 v_2}{m_3}\right|^{2p}\right)\mszi(dx_1,dv_1) \mszk(dx_2,dv_2) \biggr]\\
&\leq C(C_{\circ},K)\left(1+\sum_{j=1}^J\theta^{\zeta,j,2p}_{s}\right)
\end{align*}

It remains to bound
\begin{align*}
  II &\coloneqq \mathbb{E}\biggl[ \int_{\txtv}\frac{1}{\bm{\alpha}^{(\ell)}!}K_\ell(\bm{x},\bm{v})\biggl(\int_{\yv}\left(\sum_{r=1}^{\blj}|\vprj|^{2p}\right)\bigl(m_{\ell}^{\eta}\left(\bm{y},\bm{v'}|\bm{x},\bm{v}\right)\\
&\phantom{=}-m_{\ell}\left(\bm{y},\bm{v'}|\bm{x},\bm{v}\right)\bigr)d\bm{y}d\bm{v'}\biggr)\lmeasure \biggr].
\end{align*}

The analysis of the innermost integral is done as in the proof of Lemma \ref{L:placementDensityDifference} using the decomposition in terms of the different types of reactions (similar to the approach above). The role of the test function $f(y,v')$ in that lemma is played here by  $f(y,v')= \sum_{r = 1}^{\beta_{\ell j}}\left(|\yrj|^{2p} + |{v'}_r^{(j)}|^{2p}\right)$. The primary change in the estimates is that we can no longer use the uniform $C_b^{1}(\R^{d}\times \R^{d})$ norm when bounding $\abs{f(y,v') - f(x,v)}$ over balls about $(x,v)$, since the moments are unbounded over free-space. We instead use $(x,v)$-dependent estimates over the balls combined with Assumption~\ref{A:AssumptionRho} to obtain bounds of the form
\begin{equation*}
\left|\int_{\mathbb{Y}^{(\ell)}}   |\vprj|^{2p}  \left(m^{\eta}_\ell\left(\vec{y}, \bm{v'} |\vec{x},\bm{v} \right)-m_\ell\left(\vec{y}, \bm{v'} |\vec{x},\bm{v} \right)\right)\, d \vec{y}d \vec{v'}\right|  \leq C(p) \eta \paren{1 + \sum_{j=1}^J \sum_{r = 1}^{\alpha_{\ell j}}\abs{v_{r}^{(j)}}^{2p-1}}.
\end{equation*}

Combining this with (\ref{Eq:UniformOperatorBound}) and (\ref{Eq:UniformReactionsBound}) we get from (\ref{Eq:Expectation_Mean_Formula}) with $f(x,v)=|x|^{2p} +|v|^{2p}$ the bound
\begin{align}
\theta^{\zeta,j,2p}_{t}
&
\leq\theta^{\zeta,j,2p}_{0} + C + C\int_{0}^{t} \theta^{\zeta,j,2p}_{s} ds+C\sum_{j'=1}^{J}\int_{0}^{t} \theta^{\zeta,j',2p}_{s} ds.
\end{align}
Summing the later over $j=1,\cdots,J$ and applying standard Gronwall lemma then gives as desired
\begin{align}
\sup_{t\in[0,T]}\sum_{j=1}^{J} \theta^{\zeta,j,2p}_{t}&\leq C,
\end{align}
for a finite constant $C<\infty$ that does not depend on $\zeta$.

\end{proof}

\begin{lemma}\label{L:boundoffm}
For the $\ell-$th reaction, $1 \leq \ell \leq L,$ let $\eta$ be sufficiently small, $\epsilon >0,$ and $R \in \mathbb{R}$ large enough as in Assumption \ref{A:AssumptionRho}. Consider a fixed constant $D> \left\{\frac{2m_{1}}{m_{1}+m_{2}}, \frac{2m_{2}}{m_{1}+m_{2}}\right\}$ entering only in the case of a second order reaction $\mathcal{R}_{\ell}$ of the form $S_i + S_k \rightarrow S_j + S_r$. Also, let $\sigma^{*}\in [1,2)$ be  another fixed constant depending only on the stoichiometric coefficients $\alpha_{i}$ and particle masses $m_1,m_2,m_3$ for a second order reaction $\mathcal{R}_{\ell}$ of the form $S_i + S_k \rightarrow S_j$ (to be explicitly determined in the proof). The following estimates hold for $m$ large enough,
\begin{align*}
&\mathbb{E}\biggl[\supint\int_0^t \int_{\txtv}\frac{1}{\bm{\alpha}^{(\ell)}!}K_\ell(\bm{x},\bm{v})\biggl(\int_{\yv}\left(\sum_{r=1}^{\blj}f_m\left(y_r^{(j)},{v'}_{r}^{(j)}\right)\right)\\
&\phantom{=}\times m_{\ell}^{\eta}\left(\bm{y},\bm{v'}|\bm{x},\bm{v}\right)d\bm{y}d\bm{v'}\biggr)\lmeasure ds\biggr]\\
\leq&2 C(K)(C_{\circ} \vee 1) \displaystyle \sup_{1\leq i\leq J} \mathbb{E}\left[\sup_{t\in[0,T]}\int_0^t \left\langle f_{m-1-(D+2)R}(x,v), \msmzi(dx,dv)\right\rangle ds \right]\\
&\phantom{=}+2C(K)C_{\circ}\displaystyle \sup_{1\leq i\leq J} \mathbb{E}\left[\sup_{t\in[0,T]}\int_0^t  \left\langle f_{\floor{\frac{m-1}{1+\sigma^{*}}}}(x,v), \msmzi(dx,dv)ds \right \rangle ds\right]\\
&\phantom{=}+ 2 C(K) C_{\circ}^2\left\|f_{m}\right\|_{C_{b}^{2}\left(\mathbb{R}^{d}\times \mathbb{R}^{d}\right)} T \left(\epsilon + C\eta \right).
\end{align*}
\end{lemma}

\begin{proof}
    \underline{Case 1: Reaction of the form $S_{i} \rightarrow S_{j}$.}
\begin{align*}
&\int_0^t \int_{\txtv}\frac{1}{\bm{\alpha}^{(\ell)}!}K_\ell(\bm{x},\bm{v})\biggl(\int_{\yv}\left(\sum_{r=1}^{\blj}f_m\left(y_r^{(j)},{v'}_{r}^{(j)}\right)\right)\\
&\phantom{=}\times m_{\ell}\left(\bm{y},\bm{v'}|\bm{x},\bm{v}\right)d\bm{y}d\bm{v'}\biggr)\lmeasure ds\\
=&\int_0^t\int_{\Rd \times \Rd}K_\ell(x,v)\left(\int_{\Rd \times \Rd}f_m\left(y,v'\right)\delta_{x}(y)\delta_{v}(v')dy dv' \right)\msmzi(dx,dv)\\
=&\int_0^t\int_{\Rd \times \Rd}K_\ell(x,v)f_m\left(x,v\right)\msmzi(dx,dv)\\
\leq&C(K)\int_0^t \left\langle f_{m}(x,v), \msmzi(dx,dv)\right\rangle ds.
\end{align*}

\underline{Case 2: Reaction of the form $S_{i} \rightarrow S_{j}+S_{k}$.}
\begin{align*}
&\int_0^t \int_{\txtv}\frac{1}{\bm{\alpha}^{(\ell)}!}K_\ell(\bm{x},\bm{v})\biggl(\int_{\yv}\left(\sum_{r=1}^{\blj}f_m\left(y_r^{(j)},{v'}_{r}^{(j)}\right)\right)\\
&\phantom{=}\times m_{\ell}\left(\bm{y},\bm{v'}|\bm{x},\bm{v}\right)d\bm{y}d\bm{v'}\biggr)\lmeasure ds\\
=&\int_0^t\int_{\Rd \times \Rd}K_\ell(x,v)\biggl(\int_{\mathbb{R}^{4d}}\left(f_m\left(y_1,v_1\right)+f_m\left(y_2,v_2\right)\right)\rho\left(|y_1-y_2|,|v_1-v_2|\right) \\
&\phantom{=}\times\sum_{i=1}^{I} p_{i} \delta\left(x-\left(\alpha_{i} y_1+(1-\alpha_{i}) y_2\right)\right)\delta\left(v-\frac{m_1 v_1+ m_2 v_2}{m_3}\right)dy_1 dy_2 dv_1 dv_2 \biggr)\msmzi(dx,dv)\\
\leq&C(K)\int_0^t\biggl\langle\biggl(\int_{|y_1-y_2|+|v_1-v_2|\leq R}\left(f_m\left(y_1,v_1\right)+f_m\left(y_2,v_2\right)\right)\rho\left(|y_1-y_2|,|v_1-v_2|\right) \\
&\phantom{=}\times\sum_{i=1}^{I} p_{i} \delta\left(x-\left(\alpha_{i} y_1+(1-\alpha_{i}) y_2\right)\right)\delta\left(v-\frac{m_1 v_1+ m_2 v_2}{m_3}\right)dy_1 dy_2 dv_1 dv_2 \biggr),\msmzi(dx,dv)\biggr\rangle ds\\
&\phantom{=}+C(K)\int_0^t\biggl\langle\biggl(\int_{|y_1-y_2|+|v_1-v_2|> R}\left(f_m\left(y_1,v_1\right)+f_m\left(y_2,v_2\right)\right)\rho\left(|y_1-y_2|,|v_1-v_2|\right) \\
&\phantom{=}\times\sum_{i=1}^{I} p_{i} \delta\left(x-\left(\alpha_{i} y_1+(1-\alpha_{i}) y_2\right)\right)\delta\left(v-\frac{m_1 v_1+ m_2 v_2}{m_3}\right)dy_1 dy_2 dv_1 dv_2 \biggr),\msmzi(dx,dv)\biggr\rangle ds\\
\leq&C(K)\int_0^t\biggl\langle\biggl(\int_{|y_1-y_2|+|v_1-v_2|\leq R}2f_{m-1-R}(x,v)\rho\left(|y_1-y_2|,|v_1-v_2|\right) \\
&\phantom{=}\times\sum_{i=1}^{I} p_{i} \delta\left(x-\left(\alpha_{i} y_1+(1-\alpha_{i}) y_2\right)\right)\delta\left(v-\frac{m_1 v_1+ m_2 v_2}{m_3}\right)dy_1 dy_2 dv_1 dv_2 \biggr),\msmzi(dx,dv)\biggr\rangle ds\\
&\phantom{=}+C(K)\norm{f}_{C_b^0(\Rd \times \Rd)}\int_0^t\biggl\langle\sum_{i=1}^{I} p_{i}\biggl(\int_{|w|+|u|> R}\int_{\R^{2d}}\rho\left(|w|,|u|\right) \\
&\phantom{=}\times \delta\left(x-\alpha_{i}w-y_2\right)\delta\left(v-\frac{m_1}{m_3}u-v_2\right)dy_2 dv_2 dw du  \biggr),\msmzi(dx,dv)\biggr\rangle ds \text{ (by Assumption \ref{A:conservationofMomentum})}\\
\leq&2C(K)\int_0^t\biggl\langle f_{m-1-R}(x,v)\biggl(\int_{|y_1-y_2|+|v_1-v_2|\leq R}m_{\ell}^{\eta}\left(y_1,y_2,v_1,v_2|x,v\right)dy_1 dy_2 dv_1 dv_2 \biggr),\msmzi(dx,dv)\biggr\rangle ds\\
&\phantom{=}+C(K)\norm{f}_{C_b^0(\Rd \times \Rd)}\int_0^t\biggl\langle\sum_{i=1}^{I} p_{i}\biggl(\int_{|w|+|u|> R}\int_{\R^{2d}}\rho\left(|w|,|u|\right) \\
&\phantom{=}\times \delta\left(x-\alpha_{i}w-y_2\right)\delta\left(v-\frac{m_1}{m_3}u-v_2\right)dy_2 dv_2 dw du  \biggr),\msmzi(dx,dv)\biggr\rangle ds\\
\leq&2C(K)\int_0^t \left\langle f_{m-1-R}(x,v), \msmzi(dx,dv)\right\rangle ds + C(K)\norm{f}_{C^0_b(\Rd \times \Rd)}C_{\circ}t\epsilon.
\end{align*}
   \underline{ Case 3: Reaction of the form $S_{i}+S_{k} \rightarrow S_{j}$.}

First we notice that the conservation of momentum $m_{3}v_{3}=m_1 v_1+ m_2 v_2$ together with the constraint $m_{3}=m_{1}+m_{2}$ imply that we can write
   $v_{3}=b v_1+ (1-b) v_2$ with $b\in(0,1)$. The latter, together with the relation $x_{3}=\alpha_{i} x_1+(1-\alpha_{i}) x_2$ imply that for $c_{i}=\alpha_{i}-b$
   \begin{align*}
|x_{3}|+ |v_{3}|&\leq \alpha_{i} |x_1|+(1-\alpha_{i}) |x_2| +b |v_1|+ (1-b) |v_2|\nonumber\\
&=(c_{i}+b) |x_1|+(1-c_{i}-b) |x_2| +b |v_1|+ (1-b) |v_2|\nonumber\\
&=b(|x_{1}|+|v_{1}|)+(1-b)(|x_{2}|+|v_{2}|)+c_{i}(|x_{1}|-|x_{2}|)\nonumber\\
&\leq \max_{j=1,2}\{|x_{j}|+|v_{j}|\}+|c_{i}|(|x_{1}|+|x_{2}|)\nonumber\\
&\leq (1+2\max_{i}|c_{i}|)\max_{j=1,2}\{|x_{j}|+|v_{j}|\}\nonumber\\
&\leq (1+\sigma)\max_{j=1,2}\{|x_{j}|+|v_{j}|\},
   \end{align*}
where $\sigma^{*}=\max\{2 \max_{i}|c_{i}|,1\} \in[1,2)$. This means that if $m-1< |x_{3}|+|v_{3}|$ then we have
\begin{align*}
\frac{m-1}{1+\sigma^{*}}&< \max_{j=1,2}\{|x_{j}|+|v_{j}|\}.
\end{align*}
Consequently, this yields that for pairs $(x_{j},v_{j}), j=1,2,3$ such that $m_{3}v_{3}=m_1 v_1+ m_2 v_2$ and $x_{3}=\alpha_{i} x_1+(1-\alpha_{i}) x_2$, we have
\begin{align*}
f_{m}(x_{3}, v_{3})&\leq f_{\floor{\frac{m-1}{1+\sigma^{*}}}}(x_{1},v_{1})+ f_{\floor{\frac{m-1}{1+\sigma^{*}}}}(x_{2},v_{2})
\end{align*}
Therefore, we obtain for the term in question
    \begin{align*}
&\int_0^t \int_{\txtv}\frac{1}{\bm{\alpha}^{(\ell)}!}K_\ell(\bm{x},\bm{v})\biggl(\int_{\yv}\left(\sum_{r=1}^{\blj}f_m\left(y_r^{(j)},{v'}_{r}^{(j)}\right)\right)\\
&\phantom{=}\times m_{\ell}\left(\bm{y},\bm{v'}|\bm{x},\bm{v}\right)d\bm{y}d\bm{v'}\biggr)\lmeasure ds\\
\leq&\int_0^t\biggl\langle \biggl\langle K_\ell(x_1,x_2,v_1,v_2)\biggl(\int_{\Rd \times \Rd}f_m(y,v')\sum_{i=1}^{I} p_{i}\delta\left(y-\left(\alpha_{i} x_1+(1-\alpha_{i}) x_2\right)\right)\\
    &\phantom{=}\times \delta\left(v'-\frac{m_1 v_1+ m_2 v_2}{m_3}\right)dy dv' \biggr),\msmzi(dx_1,dv_1)\biggr\rangle, \msmzk(dx_2,dv_2)\biggr\rangle ds\\
\leq&C(K)\int_0^t\left\langle \left\langle \left(f_{\floor{\frac{m-1}{1+\sigma^{*}}}}(x_1,v_1)+f_{\floor{\frac{m-1}{1+\sigma^{*}}}}(x_2,v_2)\right),\msmzi(dx_1,dv_1)\right\rangle, \msmzk(dx_2,dv_2)\right\rangle ds \\
&\qquad\text{ (by Assumption \ref{A:conservationofMomentum})}\\
\leq&C(K)C_{\circ}\int_0^t\left( \left\langle f_{\floor{\frac{m-1}{1+\sigma^{*}}}}(x_1,v_1), \msmzi(dx_1,dv_1)\right \rangle + \left \langle f_{\floor{\frac{m-1}{1+\sigma^{*}}}}(x_2,v_2), \msmzk(dx_2,dv_2)\right \rangle \right)ds.
\end{align*}

 \underline{Case 4: Reaction of the form $S_i + S_k \rightarrow S_j + S_r$.}
 
Recall Assumption~\ref{A:conservationofMomentum} that conservation of momentum, $m_3 v_1' + m_4 v_2' = m_1 v_1 + m_2 v_2$, and conservation of mass, $m_3+m_4=m_1+m_2$, hold. We define the (normalized) velocity component of the placement density by
\begin{equation*}
m_{\ell}(v'_1,v'_2 |v_1,v_2)=(m_3 + m_4)^d \rho(|v_1' - v_2'|)\delta(m_3 v_1' + m_4 v_2' - m_1 v_1 - m_2 v_2)
\end{equation*}
and let
\begin{equation*}
g_m(x_1,x_2,v_1',v_2')=p f_m\left(x_{1}, v_1'\right)+p f_m\left(x_{2}, v_2'\right) + (1-p)f_m\left(x_{2}, v_1'\right)+(1-p) f_m\left(x_{1}, v_2'\right).
\end{equation*}
We will estimate the reaction integral term by splitting it into three components as follows.
\begin{align*}
&\int_0^t \int_{\txtv}\frac{1}{\bm{\alpha}^{(\ell)}!}K_\ell(\bm{x},\bm{v})\biggl(\int_{\yv}\left(\sum_{r=1}^{\blj}f_m\left(y_r^{(j)},{v'}_{r}^{(j)}\right)\right)\\
&\phantom{=}\times m_{\ell}\left(\bm{y},\bm{v'}|\bm{x},\bm{v}\right)d\bm{y}d\bm{v'}\biggr)\lmeasure ds\\
&\leq\int_0^t\biggl \langle \biggl \langle K_\ell(x_1, x_2,v_1,v_2)\biggl(\int_{\mathbb{R}^{4d}}(m_3 + m_4)^d \left(f_m\left(y_{1}, v_1'\right)+f_m\left(y_{2}, v_2'\right)\right) \rho(|v_1' - v_2'|)\\
&\phantom{=}\times\delta(m_3 v_1' + m_4 v_2' - m_1 v_1 - m_2 v_2)\biggl(p \times \delta_{(x_1, x_2)}\left((y_1, y_2)\right)\\
&\phantom{=}+(1-p) \times \delta_{(x_1, x_2)}\left((y_2, y_1)\right)\biggr)d y_{1} d y_{2} d v_{1}' d v_{2}'\biggr),\msmzi(dx_1,dv_1)\biggr\rangle, \msmzk(dx_2,dv_2)\biggr\rangle ds\\
&\leq C(K)\biggl(\int_0^t\biggl \langle \biggl \langle \biggl(\int_{\Rd \times \Rd} g_m(x_1, x_2, v_1', v_2')
m_\ell(v_1',v_2' | v_1,v_2) d v_{1}'d v_{2}' \biggr),\msmzi(dx_1,dv_1)\biggr\rangle, \msmzk(dx_2,dv_2)\biggr\rangle ds\\
&= I + II + III,
\end{align*}
where we partition the velocity space into three regions that define $I$, $II$, and $III$ as now described. Let
\begin{align*}
\Gamma_{D}(v_1,v_2)=\left\{(v_1',v_2') \in \R^{2d} \,|\,\max\{|v_1-v'_1|, |v_2-v'_2|\}\leq D R\right\},
\end{align*}
and denote by $\Gamma_{D}^c(v_1,v_2)$ its complement in $\R^{2d}$. We will choose the constant $0 < D < \infty$ later.
We split the integral over the regions
\begin{align*}
    A_I &= \{(v_1,v_2,v_1',v_2') \in \R^{4d} \,|\, \max\{|v_1|,|v_2|\} < R, (v_1',v_2') \in \Gamma_{D}(v_1,v_2)\} \\
    A_{II} &= \{(v_1,v_2,v_1',v_2') \in \R^{4d} \,|\, \max\{|v_1|,|v_2|\} < R, (v_1',v_2') \in \Gamma_{D}^c(v_1,v_2)\}  \\
    A_{III} &=\{(v_1,v_2,v_1',v_2') \in \R^{4d} \,|\, \max\{|v_1|,|v_2|\} \geq R, (v_1',v_2') \in \R^{2d}\}.
\end{align*}
Then for $J \in \{I, II, III\}$ the associated integral is
\begin{align*}
    J &= C(K)\biggl(\int_0^t\biggl< \biggl< \biggl(\int_{\Rd \times \Rd} \ind_{A_J}(v_1,v_2,v_1',v_2') g_m(x_1, x_2, v_1', v_2') m_\ell(v_1',v_2' | v_1,v_2) d v_{1}'d v_{2}' \biggr)\nonumber \\
    & \hspace{5cm} \msmzi(dx_1,dv_1)\biggr\rangle,\msmzk(dx_2,dv_2)\biggr\rangle ds\biggr).
\end{align*}

Consider first the integral $I$. On the set $\Gamma_{D}$,  we have for $i=1,2$
\begin{align*}
\left| |v_i|-|v'_i|\right|\leq |v_i-v'_i|\leq D R,
\end{align*}
which then gives
\begin{align*}
\left|(|x_{i}|+|v_i|)-(|x_{i}|+|v'_i|)\right|\leq D R
\end{align*}
We want to bound $g_m(x_1,x_2,v_1',v_2')$ by $f_k$ for some $k$ at the points $(x_i,v_i)$. If $f_m(x_i,v'_i)>0$, then we have that $|x_{i}|+|v'_i|>m-1$, which gives $|x_{i}|+|v_i|>m-1- D R$. So, we will always have
\begin{align*}
f_m(x_i,v'_i)&\leq f_{m-1- D R}(x_i,v_i).
\end{align*}
Next we treat the cross pairs. Our domain assumptions and the triangular inequality give
\begin{align*}
\left| |v'_2|-|v_1|\right|\leq|v'_{2}-v_{1}|&\leq |v'_{2}-v_{2}|+|v_{2}-v_{1}|\leq D R + 2R=(D+2)R
\end{align*}
Hence, the same argument as above gives
\begin{align*}
f_m(x_1,v'_2)&\leq f_{m-1- (D+2) R}(x_1,v_1).
\end{align*}
Similarly, we obtain
\begin{align*}
f_m(x_2,v'_1)&\leq f_{m-1- (D+2) R}(x_2,v_2).
\end{align*}
Combining the above estimates and using the monotonicity of $f_{m}$ in $m$ gives
\begin{align*}
g_m(x_1,x_2,v_1',v_2') \leq f_{m-1-(D+2) R}(x_1,v_1)+f_{m-1-(D+2)R}(x_2, v_2)
\end{align*}
so that
\begin{align*}
& \int_{\Gamma_{D}(v_1,v_2)}g_m(x_1,x_2,v_1',v_2') m_{\ell}(v'_1,v'_2 | v_1,v_2) d v'_1\,d v'_2 \nonumber\\
&\leq \left(f_{m-1-(D+2) R}(x_1,v_1)+f_{m-1-(D+2)R}(x_2, v_2)\right) \int_{\Gamma_{D}(v_1,v_2)} m_{\ell}(v'_1,v'_2 | v_1,v_2) d v'_1\,d v'_2 \nonumber\\
&\leq \left(f_{m-1-(D+2) R}(x_1,v_1)+f_{m-1-(D+2)R}(x_2, v_2)\right) \int_{\R^{2d}} m_{\ell}(v'_1,v'_2 | v_1,v_2) d v'_1\,d v'_2 \nonumber\\
&= \left(f_{m-1-(D+2) R}(x_1,v_1)+f_{m-1-(D+2)R}(x_2, v_2)\right).
\end{align*}
We therefore conclude that
\begin{align*}
    I&\leq C(K)\int_0^t\biggl \langle \biggl \langle \left(f_{m-1-(D+2) R}(x_1,v_1)+f_{m-1-(D+2)R}(x_2, v_2)\right),\msmzi(dx_1,dv_1)\biggr\rangle, \msmzk(dx_2,dv_2)\biggr\rangle ds \\
    &\leq C(K)C_{\circ}\int_0^t\left( \left\langle f_{m-1-(D+2) R}(x_1,v_1), \msmzi(dx_1,dv_1)\right \rangle + \left \langle f_{m-1-(D+2)R}(x_2,v_2), \msmzk(dx_2,dv_2)\right \rangle \right)ds.
\end{align*}

Next we consider integral II. For this purpose we define $Z_{1}=m_{1}+m_{2}=m_{3}+m_{4}$ and $Z_{2}=m_3 v_1' + m_4 v_2'=m_1 v_1 + m_2 v_2$ for the total mass and total momentum respectively. Setting $v'_{12}=v'_{1}-v'_{2}$ we then obtain that
\begin{align*}
v'_{1}&=\frac{Z_{2}+m_{4}v'_{12}}{Z_{1}},\quad \text{ and }\quad v'_{2}=\frac{Z_{2}-m_{3}v'_{12}}{Z_{1}}.
\end{align*}
We change variables in the integration, from $(v'_{1},v'_{2})$ to  $(Z_{2}, v'_{12})$. The Jacobian of this transformation is $1/Z_{1}^{d}$ which we absorb into a generic constant $C$. Then
\begin{align*}
\int_{\Gamma^{c}_{D}(v_1,v_2)} &g_m(x_1,x_2,v_1',v_2') m_{\ell}(v'_1,v'_2 | v_1,v_2) d v'_1\,d v'_2 \\
&\leq 2 \norm{f_m}_{C_b^2(\R^{2d})}\int_{\Gamma^{c}_{D}(v_1,v_2)} m_{\ell}(v'_1,v'_2 | v_1,v_2) d v'_1,d v'_2 \\
&\leq C \norm{f_m}_{C_b^2(\R^{2d})}\int_{\Gamma^{c}_{D}(v_1,v_2)} \rho(|v'_{12}|) \delta(Z_2 - (m_1 v_1 + m_2 v_2)) d v'_{12} \, dZ_2,
\end{align*}
where we have abused notation and denoted by $\Gamma^{c}_{D}(v_1,v_2)$ the set of points in  $(Z_2,v'_{12})$ corresponding to the set $\Gamma^{c}_{D}(v_1,v_2)$ in $(v'_{1},v'_{2})$ via the change of variables.

To further simplify this integral, we now relate the regions $\{|v'_{12}|\leq R\}$ and  $\Gamma^{c}_{D}(v_1,v_2)$. First, we note that
\begin{align*}
v'_{1}-v_{1}&=\frac{Z_{2}+m_{4}v'_{12}}{Z_{1}}-v_{1}=\frac{m_{2}(v_2-v_1)+m_4 v'_{12}}{Z_{1}}\nonumber\\
v'_{2}-v_{2}&=\frac{Z_{2}-m_{3}v'_{12}}{Z_{1}}-v_{2}=\frac{m_{1}(v_1-v_2)-m_3 v'_{12}}{Z_{1}}.
\end{align*}
We have two cases for points in $\Gamma^{c}_{D}(v_1,v_2)$. If $|v'_{1}-v_{1}|> D R$, then the previous expression for $v'_{1}-v_{1}$ and the reverse triangular inequality yield
\begin{align*}
|v'_{12}|&\geq \frac{Z_{1} D R- m_{2} |v_{1}-v_{2}|}{m_{4}} \geq \frac{Z_{1} D R- 2 m_{2} R}{m_{4}}.
\end{align*}
Similarly, if $|v'_{2}-v_{2}|> D R$, then the previous expression for $v'_{2}-v_{2}$ and the reverse triangular inequality yield
\begin{align*}
|v'_{12}|&\geq \frac{Z_{1} D R- m_{1} |v_{1}-v_{2}|}{m_{3}} \geq \frac{Z_{1} D R- 2 m_{1} R}{m_{3}}.
\end{align*}
Hence for $D$ sufficiently large (specifically $D> \left\{\frac{2m_{1}}{m_{1}+m_{2}}, \frac{2m_{2}}{m_{1}+m_{2}}\right\}$) we have that there exists a constant $C(D)>0$ that depends on $D$ and is independent of $R$ such that
\begin{align*}
|v'_{12}|&\geq C(D) R.
\end{align*}
We then have that
\begin{align*}
\Gamma^{c}_{D}(v_1,v_2)\subset \{(Z_2,v'_{12}) \in \R^{2d}\,|\, Z_2 \in \R^{d}, |v'_{12}|\geq C(D)R\}.
\end{align*}

This and the integral tail bound on the density via Assumption \ref{A:AssumptionRho} then yield the integral bound
\begin{align*}
\int_{\Gamma^{c}_{D}(v_1,v_2)} \rho(|v'_{12}|) \delta(Z_2 - (m_1 v_1 + m_2 v_2)) d v'_{12} \, dZ_2
\leq \int_{C(D) R\leq |v'_{12}}\rho(|v'_{12}|) d v'_{12}
\leq \epsilon
\end{align*}
for $R$ sufficiently large so that
\begin{align*}
    II
    \leq C(K) C_{\circ}^2 \norm{f_m}_{C_b^2(\R^{2d})} T \epsilon.
\end{align*}

Finally, we now consider the integral $III$. Note that
\begin{align*}
\int_{\R^{2d}} g_m(x_1,x_2,v_1',v_2') m_{\ell}(v'_1,v'_2 | v_1,v_2) d v'_1\,d v'_2 \leq 2 \norm{f_m}_{C_b^2(\R^{2d})}.
\end{align*}
Then
\begin{align*}
    III &\leq C(K) \norm{f_m}_{C_b^2(\R^{2d})} \int_0^t\biggl \langle \biggl \langle \ind_{\max\{|v_1|,|v_2|\}\geq R} ,\msmzi(dx_1,dv_1)\biggr\rangle, \msmzk(dx_2,dv_2)\biggr\rangle ds,\\
    &\leq C(K)C_{\circ} \norm{f_m}_{C_b^2(\R^{2d})} \int_0^t \biggl\langle \ind_{|v_1|\geq R},\msmzi(dx_1,dv_1)\biggr\rangle  + \biggl \langle \ind_{|v_2|\geq R}, \msmzk(dx_2,dv_2)\biggr\rangle ds, \\
    &\leq C(K)C_{\circ} \norm{f_m}_{C_b^2(\R^{2d})}  \frac{1}{R} \int_0^t \biggl\langle |v_1|,\msmzi(dx_1,dv_1)\biggr\rangle  + \biggl \langle |v_2|, \msmzk(dx_2,dv_2)\biggr\rangle ds,\\
    &\leq C C(K) C_{\circ} \norm{f_m}_{C_b^2(\R^{2d})}  \frac{T}{R} \\
    &\leq C C(K) C_{\circ} \norm{f_m}_{C_b^2(\R^{2d})} T \epsilon
\end{align*}
where we used the uniform moment bound of Lemma~\ref{L:MomentBounds} and assumed $R$ sufficiently large.

\end{proof}

\begin{lemma} \label{L:interchangelimit}
Consider $\gtgamma \coloneqq \mathbb{E}\left[\int_0^t \left\langle |v|^2 \mathbbm{1}_{\{m-1 \leq |x| + |v| \leq m\}} , \msmzj(dx,dv)\right\rangle ds\right]$. Then
\begin{equation*}
\lim_{m\rightarrow\infty} \sup_{\zeta \in (0,1)^2} \sup_{t \in [0,T]} \gtgamma = 0.
\end{equation*}
\end{lemma}
\begin{proof}
Let $A_m = \{m-1 \leq |x| + |v| \leq m\}$ and note that
\begin{equation*}
\gtgamma = \mathbb{E}\left[\int_0^t \left\langle |v|^2 \mathbbm{1}_{A_m}, \msmzj(dx,dv)\right\rangle ds\right].
\end{equation*}
For $m \geq 2$, we have
\begin{align*}
|v|^2 \mathbbm{1}_{A_m} \leq (|x| + |v|)^2 \mathbbm{1}_{\{|x| + |v| \geq m-1\}} \leq \frac{(|x| + |v|)^4}{(m-1)^2}.
\end{align*}
Since $(|x| + |v|)^4 \leq C (|x|^4 + |v|^4)$ for some constant $C>0$, we have
\begin{align*}
\gtgamma &\leq \frac{1}{(m-1)^2} \mathbb{E}\left[\int_0^t \left\langle (|x| + |v|)^4 , \msmzj(dx,dv)\right\rangle ds\right] \\
&\leq \frac{C}{(m-1)^2} \mathbb{E}\left[\int_0^t \left\langle (|x|^4 + |v|^4) , \msmzj(dx,dv)\right\rangle ds\right] \\
&\leq \frac{C'}{(m-1)^2},
\end{align*}
where $C'$ is independent of $m$, $\zeta$ and $t$ due to the uniform moment bounds established in Lemma \ref{L:MomentBounds}.

\end{proof}

%
%

\section{Conclusion}
We have proven the rigorous mean-field, large-population limit for (underdamped) particle-based reactive Langevin Dynamics (PBRLD) models. Such models were previously studied in~\cite{SamChenLanlanKostas2025}, where it was shown how reaction kernels within them could be formulated such that PBRLD models are consistent in the over-damped limit with volume reactivity particle-based stochastic reaction-diffusion (PBSRD) models. This work complements~\cite{SamChenLanlanKostas2025} by showing that the mean-field limit of PBRLD models is well-defined and that the resulting macroscopic, mean-field equations are a novel set of kinetic reaction-diffusion type partial-integro differential equations (PIDEs). A number of interesting open questions remain, including how the kinetic PIDEs relate to standard reaction-diffusion PDEs, i.e. are the PIDEs that represent the mean-field limit of PBSRD models the overdamped limit of our derived kinetic PIDEs, and does one recover standard reaction-diffusion PDE models when considering overdamped limits combined with short range limits? For the latter, it would be of interest to consider unifying our work on the overdamped limit in~\cite{SamChenLanlanKostas2025} with our work establishing that the short-range limit of the PBSRD mean-field PIDEs recovers standard reaction-diffusion PDE models for certain classes of reaction kernels~\cite{IsaacsonSIAP2021}.

\section{Declaration of generative AI and AI-assisted technology use}
During the editing of this manuscript SAI used Github Copilot as part of the
Visual Studio Code editor. This included auto-completion suggestions for text,
and suggestions for revisions to written text to improve clarity of exposition
and grammar. SAI also had select proofs in a draft manuscript reviewed by GPT 5.5
Pro, which was used to find and verify gaps and mistakes, and suggested
corrections / improvements in several of the lemmas (in particular, it suggested
using the uniform moment bound to control the velocity tail behavior
of $\msmzj$ as in the final proof of Lemma \ref{L:interchangelimit}, which
improved and shortened an earlier argument). After using these tools, the
authors reviewed and edited the content as needed and take full responsibility
for the final content of the article.

\bibliographystyle{plain}
\bibliography{lib.bib}

\end{document}